\documentclass[a4paper,11pt]{amsart}

\usepackage[utf8x]{inputenc}
\usepackage[T1]{fontenc}
\usepackage{amsfonts}
\usepackage{amsmath}
\usepackage{amsthm}
\usepackage{amssymb}
\usepackage{graphicx}
\usepackage[frenchb,english]{babel}
\usepackage{verbatim}
\usepackage{xy}
\usepackage{graphics}
\usepackage{bbm}
\usepackage{epstopdf}

\let\~=\tilde

\newcommand{\sign}{\mathrm{sign}}
\newcommand{\R}{\mathbb{R}}
\newcommand{\Z}{\mathbb{Z}}
\newcommand{\N}{\mathbb{N}}
\newcommand{\Imm}{\mathrm{Im}}
\newcommand{\Deg}{\mathrm{deg}}

\newtheorem{Thm}{\textbf{Theorem}}[section]
\newtheorem{Lemma}[Thm]{\textbf{Lemma}}
\newtheorem{Cor}[Thm]{\textbf{Corollary}}
\newtheorem{Prop}[Thm]{\textbf{Proposition}}

\theoremstyle{remark}

\numberwithin{equation}{section}
\newtheorem{Ex}[Thm]{\textbf{Example}}
\newtheorem{Def}[Thm]{\textbf{Definition}}
\newtheorem{Obs}[Thm]{\textbf{Observation}}

\newtheorem*{Notation}{\textbf{Notation}}
\newtheorem{Cnj}[Thm]{\textbf{Conjecture}}
\newtheorem*{Notations}{\textbf{Notations}}

\title[A categorification of the Alexander polynomial in ECH]{A categorification of the Alexander polynomial in embedded contact homology}
\date{}
\author{Gilberto Spano}

\begin{document}

\let\thefootnote\relax\footnote{Date: \today}
\let\thefootnote\relax
\let\thefootnote\relax\footnote{The author was partially supported by University of Nantes and ERC Geodycon.}

\begin{abstract}
 Given a transverse knot $K$ in a three dimensional contact manifold $(Y,\alpha)$, in \cite{CGHH} Colin, Ghiggini, Honda and Hutchings define a hat version
 of embedded contact homology for $K$, that we call $\widehat{ECK}(K,Y,\alpha)$, and conjecture that it is isomorphic to the knot Floer homology
 $\widehat{HFK}(K,Y)$.
 
 We define here a full version $ECK(K,Y,\alpha)$ and generalise the definitions to the case of links. We prove then that, if $Y = S^3$, $ECK$ 
 and $\widehat{ECK}$ categorify the (multivariable) Alexander polynomial of knots and links, obtaining expressions analogue to that for
 knot and link Floer homologies in the plus and, respectively, hat versions.
\end{abstract}
\maketitle

\tableofcontents

\section*{Introduction}

Given a 3-manifold $Y$ in \cite{OS1} Ozsv\'{a}th and Szab\'{o} defined topological invariants of $Y$, indicated
$HF^{\infty}(Y)$, $HF^{+}(Y)$, $HF^{-}(Y)$ and $\widehat{HF}(Y)$. These groups are the \emph{Heegaard Floer homologies of $Y$} in the respective
versions.

Moreover Ozsv\'{a}th and Szab\'{o} in \cite{OS3} and Rasmussen in \cite{Ra}
proved that any homologically trivial knot $K$ in $Y$ induces a ``knot filtration'' on the Heegaard Floer chain complexes. The first pages of
the associated spectral sequences (in each versions) result then to be topological invariants of $K$: these are bigraded homology groups
$HFK^{\infty}(K,Y)$, $HFK^{+}(K,Y)$, $HFK^{-}(K,Y)$ and $\widehat{HFK}(K,Y)$ called \emph{Heegaard Floer knot homologies}
(in the respective versions).

These homologies are powerful invariants
for the couple $(K,Y)$. For instance in \cite{OS3} and \cite{Ra}, it has been proved that $\widehat{HFK}(K,S^3)$ \emph{categorifies the
Alexander polynomial $\Delta_K$ of $K$}, i.e. 
$$\chi(\widehat{HFK}(K,S^3)) \doteq \Delta(K),$$
where $\doteq$ means that the two sides are equal up to change sign and multiply by a monic monomial and $\chi$ denotes the
\emph{graded Euler characteristic}.

This was the first categorification of the Alexander polynomial; a second one (in Seiberg-Witten-Floer homology) was
discovered later by Kronheimer and Mrowka (\cite{KM2}).

In \cite{OS5} Ozsv\'{a}th and Szab\'{o} developed a similar construction for any link $L$ in $S^3$ and got invariants
$HFL^{-}(L,S^3)$ and $\widehat{HFL}(L,S^3)$ for $L$, which they called \emph{Heegaard Floer link homologies}.
Now these homologies come with an additional $\Z^n$ degree, where $n$ is the
number of the connected components of $L$. Ozsv\'{a}th and Szab\'{o} proved moreover that \emph{$HFL^{-}(L,S^3)$ categorifies the
multivariable Alexander polynomial of $L$}, which is a generalization of the classic Alexander polynomial. They found in
particular that:
\begin{eqnarray} \label{EquationIntro: Euler characteristics of HFL^-}
  \chi\left(HFL^-(L,S^3)\right) \doteq \left\{ \begin{array}{cc}
                                                \Delta_L(t_1,\ldots,t_n) & \mbox{if } n > 1\\
                                                 & \\
                                                \Delta_L(t)/({1-t})  & \mbox{if } n = 1.
                                               \end{array}\right. 
\end{eqnarray}
 and 
 \begin{eqnarray} \label{EquationIntro: Euler characteristics of widehat HFL}
  \chi\left(\widehat{HFL}(L,S^3)\right) \doteq \left\{ \begin{array}{cc}
                                                \Delta_L \cdot \prod_{i=1}^n(t_i^{\frac{1}{2}} - t_i^{-\frac{1}{2}}) & \mbox{if } n > 1\\
                                                 & \\
                                                \Delta_L(t)  & \mbox{if } n = 1.
                                               \end{array}\right. 
 \end{eqnarray}

\vspace{0.3 cm}

In the series of papers \cite{CGH1}-\cite{CGH5}, Colin, Ghiggini and Honda prove the equivalence between Heegaard Floer
homology and \emph{embedded contact homology} for three manifolds. The last one is another Floer homology theory, first defined by Hutchings,
which associates to a contact manifold $(Y,\alpha)$ two graded modules $ECH(Y,\alpha)$ and  $\widehat{ECH}(Y,\alpha)$. 

\begin{Thm}[Colin, Ghiggini, Honda, \cite{CGH1}-\cite{CGH5}] \label{Theorem: HF and ECH in intro}
\begin{eqnarray*}
  HF^+(-Y) \cong ECH(Y,\alpha) \\
  \widehat{HF}(-Y) \cong \widehat{ECH}(Y,\alpha) \nonumber \label{Equation: HF iso to ECH in hat version in intro},
 \end{eqnarray*}
 where $-Y$ is the manifold $Y$ with the inverted orientation.
\end{Thm}
In light of Theorem \ref{Theorem: HF and ECH in intro}, it is a natural problem to find an embedded contact counterpart of Heegard Floer knot homology.
In analogy with the \emph{sutured Heegaard Floer theory} developed by Juh\'{a}sz (\cite{Ju}), in \cite{CGHH} the authors define a sutured version of embedded contact
homology. This can be thought of as a version of embedded contact homology for manifolds with boundary. In particular, given a knot $K$ in
a contact three manifold $(Y,\xi)$, using sutures they define a \emph{hat version $\widehat{ECK}(K,Y,\alpha)$ of embedded contact knot homology}.

Roughly speaking, this is the hat version of $ECH$ homology for the contact manifold with boundary $(Y \setminus \mathcal{N}(K),\alpha)$,
where $\mathcal{N}(K)$ is a suitable thin tubular neighborhood of $K$ in $Y$ and $\alpha$ is a contact form satisfying specific
compatibility conditions with $K$. In \cite{CGHH} the following conjecture is stated:
\begin{Cnj} \label{ConjectureIntro: ECK iso to HFK hat}
  \begin{eqnarray*}
  \widehat{ECK}(K,Y,\alpha) & \cong & \widehat{HFK}(-K,-Y).
 \end{eqnarray*}
\end{Cnj}

\vspace{0.3 cm}

In this paper we first define a \emph{full version of embedded contact knot homology} 
$$ECK(K,Y,\alpha)$$
for knots $K$ in any contact three manifold $(Y,\xi)$ endowed with a (suitable) contact form $\alpha$ for $\xi$. Moreover we generalize the
definitions to the case of links $L$ with more then one components to obtain homologies
\begin{eqnarray*}
 ECK(L,Y,\alpha)\ &\ \mbox{ and }\ &\ \widehat{ECK}(L,Y,\alpha). 
\end{eqnarray*}
We state then the following:
\begin{Cnj} \label{ConjectureIntro: ECK iso to HFK} For any link $L$ in $Y$, there exist contact forms for which:
  \begin{eqnarray*}
  \widehat{ECK}(L,Y,\alpha) & \cong & \widehat{HFK}(-L,-Y),\\
  ECK(L,Y,\alpha) & \cong & HFK^{-}(L,Y).
 \end{eqnarray*}
\end{Cnj}

Next we compute the graded Euler characteristics of the $ECK$ homologies for knots and links in homology three-spheres and we prove the following:
\begin{Thm} \label{TheoremIntro: Euler characteristics of ECK and ALEX}
 Let $L$ be an $n$-component link in a homology three-sphere $Y$.
 Then there exists a contact form $\alpha$ such that
 \begin{equation*}
  \chi(ECK(L,Y,\alpha)) \doteq \mathrm{ALEX}(Y \setminus L).
 \end{equation*}
\end{Thm}
Here $\mathrm{ALEX}(Y \setminus L)$ is the \emph{Alexander quotient} of the complement of $L$ in $Y$. The theorem is proved using Fried's
dynamic reformulation of $\mathrm{ALEX}$ (\cite{Fri}). Classical relations between $\mathrm{ALEX}(S^3 \setminus L)$ and $\Delta_L$ imply the
following result:
\begin{Thm} \label{TheoremIntro: Euler characteristic of ECK}
 Let $L$ be any $n$-component link in $S^3$. Then there exists a contact form $\alpha$ for which:
 \begin{eqnarray*} 
  \chi\left(ECK(L,S^3,\alpha)\right) \doteq \left\{ \begin{array}{cc}
                                                \Delta_L(t_1,\ldots,t_n) & \mbox{if } n > 1\\
                                                 & \\
                                                \Delta_L(t)/({1-t})  & \mbox{if } n = 1
                                               \end{array}\right.
 \end{eqnarray*}
  and 
 \begin{eqnarray*} 
  \chi\left(\widehat{ECK}(L,S^3,\alpha)\right) \doteq \left\{ \begin{array}{cc}
                                                \Delta_L(t_1,\ldots,t_n) \cdot \prod_{i=1}^n(1 - t_i) & \mbox{if } n > 1\\
                                                 & \\
                                                \Delta_L(t)  & \mbox{if } n = 1.
                                               \end{array}\right. 
 \end{eqnarray*}
\end{Thm}
This implies that the homology \emph{$ECK$ is a categorification of the multivariable Alexander polynomial.}

A straightforward application of last theorem is obtained by comparing it with equations
\ref{EquationIntro: Euler characteristics of HFL^-} and \ref{EquationIntro: Euler characteristics of widehat HFL}:
\begin{Cor}
 In $S^3$, Conjectures \ref{ConjectureIntro: ECK iso to HFK hat} and \ref{ConjectureIntro: ECK iso to HFK} hold at level of Euler characteristics.
\end{Cor}

\addtocontents{toc}{\protect\setcounter{tocdepth}{1}}
\subsubsection*{Acknowledgements}
\addtocontents{toc}{\protect\setcounter{tocdepth}{2}}

I first thank my advisors Paolo Ghiggini and Vincent Colin for the patient and constant help and the trust they accorded to me during the years
I spent in Nantes, where most of this work has been done. I also thank the referees for having reported my Ph.D. thesis, 
containing essentially all the results presented here. I finally thank Vinicius Gripp, Thomas Guyard, Christine Lescop, Paolo Lisca,
Margherita Sandon and Vera Vértesi for all the advice, support or stimulating conversations we had.

\section{Review of embedded contact homology}

\subsection{Preliminaries}

This subsection is devoted to remind some basic notions about contact geometry, holomorphic curves, Morse-Bott theory and open books.

\subsubsection{Contact geometry} \label{Subsection: Contact geometry}

A (co-oriented) \emph{contact form} on a three dimensional oriented manifold $Y$ is a $\alpha \in \Omega^1(Y)$ such that
$\alpha \wedge d\alpha$ is a positive volume form. A \emph{contact structure} is a smooth plane field $\xi$ on $Y$ such that there exists a
contact form $\alpha$ for which $\xi = \ker \alpha$. The \emph{Reeb vector field} of $\alpha$ is the (unique) vector field $R_{\alpha}$
determined by the equations $d\alpha(R_{\alpha},\cdot) = 0$ and $\alpha(R_{\alpha})=1$.  
A \emph{simple Reeb orbit} is a closed oriented orbit of $R = R_{\alpha}$, i.e. it is the image $\delta$ of an embedding $S^1 \hookrightarrow Y$ such that
$R_{P}$ is positively tangent to $\delta$ in any $P \in \delta$. A \emph{Reeb orbit} is an $m$-fold cover of a simple Reeb orbit, with $m \geq 1$.\\
The form $\alpha$ determines an \emph{action} $\mathcal{A}$ on the set of its Reeb orbits defined by $\mathcal{A}(\gamma) = \int_{\gamma}\alpha$.
By definition $\mathcal{A}(\gamma) > 0$ for any non empty orbit $\gamma$.

\vspace{0.3 cm}

A basic result in contact geometry asserts that the flow of the Reeb vector field (abbreviated Reeb flow) $\phi=\phi_R$ preserves $\xi$, that is
$(\phi_t)_*(\xi_P) = \xi_{\phi_t(P)}$ for any $t \in \R$ (see \cite[Chapter 1]{Ge}). Given a Reeb orbit $\delta$, there exists $T \in \R^+$ such that ${(\phi_T)}_*(\xi_P) = \xi_P$ for any
$P \in \delta$; if $T$ is the smallest possible, the isomorphism $\mathfrak{L}_{\delta}:={(\phi_T)}_*: \xi_P \rightarrow \xi_P$ is called the (\emph{symplectic})
\emph{linearized first return map} of $R$ in $P$.

The orbit $\delta$ is called \emph{non-degenerate} if $1$ is not an eigenvalue $\mathfrak{L}_{\delta}$.
There are two types of non-degenerate Reeb orbits: \emph{elliptic} and \emph{hyperbolic}. $\delta$ is elliptic if the eigenvalues of
$\mathfrak{L}_{\delta}$ are on the unit circle and is hyperbolic if they are real. In the last case we can make a further distinction:
$\delta$ is called \emph{positive} (\emph{negative}) hyperbolic if the eigenvalues are both positive (resp. negative).
\begin{Def}
 The \emph{Lefschetz sign} of a non-degenerate Reeb orbit $\delta$ is
 $$\epsilon(\delta) := \sign(\det(\mathbbm{1} - \mathfrak{L}_{\delta})) \in \{+1,-1\}.$$
\end{Def}

\begin{Obs}
 It is easy to check that $\epsilon(\delta) = +1$ if $\delta$ is elliptic or negative hyperbolic and $\epsilon(\delta) = -1$ if $\delta$ is positive hyperbolic.
\end{Obs}

To any non-degenerate orbit $\delta$ and a trivialization $\tau$ of $\xi|_{\delta}$ we can associate also the \emph{Conley-Zehnder index}
$\mu_{\tau}(\delta) \in \mathbb{Z}$ of $\delta$ with respect to $\tau$. Even if we do not give a precise definition (that can be found for
example in \cite{ET} or \cite{Gutt}) we will provide an explicit description of this index (see \cite[Section 3.2]{Hu3}).

Given $P \in \delta$, using the basis of $\xi|_{\delta}$ determined by $\tau$ we can regard the
differentials ${\phi_t}_*:\xi_P \rightarrow \xi_{\phi_t(P)}$ of the Reeb flow as a path in $t \in [0,T]$ of $2 \times 2$ symplectic matrices.
In particular ${\phi_0}_* : \xi_P \rightarrow \xi_P$ is the identity matrix and, if $T$ is as above, ${\phi_T}_* : \xi_P \rightarrow \xi_P$ is a matrix representation 
for $\mathfrak{L}_{\delta}$. 

If $\delta$ is elliptic, following this path for $t \in [0,T]$, ${\phi_T}_*$ will represent a rotation
by some angle $2\pi \theta$ with $\theta \in \R \setminus \mathbb{Z}$ (since $\delta$ is non degenerate).
Then $\mu_{\tau}(\delta) = 2\lfloor \theta \rfloor + 1$, where $\lfloor \theta \rfloor$ is the highest integer smaller then $\theta$. 

Otherwise, if $\delta$ is hyperbolic, then the symplectic matrix of ${\phi_T}_*$ rotates the eigenvectors
of $\mathfrak{L}_{\delta}$ by an angle $k\pi$ with $k \in 2\mathbb{Z}$ if $\delta$ is positive hyperbolic and $k \in 2\mathbb{Z}+1$ if $\delta$ is negative
hyperbolic. Then $\mu_{\tau}(\delta) = k$.

\begin{Obs}
 Even if $\mu_{\tau}(\delta)$ depends on $\tau$, its parity depends only on $\delta$. Indeed, if $\delta$ is elliptic, then
 $\mu_{\tau}(\delta) \equiv 1 \mod 2$. Moreover
 suppose that $\delta$ is hyperbolic and $\mu_{\tau}(\delta) = k$; if $\tau'$ differs from $\tau$ by a twist of an angle $2n\pi$ with $n \in \mathbb{Z}$, the
 rotation by $k\pi$ on the eigenvectors will be composed with a rotation by $2n\pi$. Then $\mu_{\tau'}(\delta) = k + 2n \equiv k \mod 2$. 
\end{Obs}

\begin{Cor} If $\delta$ is non-degenerate then for any $\tau$
 $$(-1)^{\mu_{\tau}(\delta)} = -\epsilon(\delta).$$
\end{Cor}

\begin{Def} \label{Definition: Orbit sets}
 Given $X\subseteq Y$, we will indicate by $\mathcal{P}(X)$ the set of simple Reeb orbits of $\alpha$ contained in $X$. An \emph{orbit set}
 (or \emph{multiorbit}) in $X$ is a formal finite product $\gamma = \prod_i \gamma_i^{k_i}$, where $\gamma_i \in \mathcal{P}(X)$ and
 $k_i \in \mathbb{N}$ is the \emph{multiplicity} of  $\gamma_i$ in $\gamma$, with $k_i \in \{0,1\}$ whenever $\gamma_i$ is hyperbolic.
 The set of multiorbits in $X$ will be denoted by $\mathcal{O}(X)$. 
\end{Def}

Note that the empty set is considered as an orbit, called \emph{empty orbit} and it is indicated by $\emptyset$. 

An orbit set $\gamma = \prod_i \gamma_i^{k_i}$ belongs to the homology class
$[\gamma] = \sum_i k_i [\gamma_i] \in H_1(Y)$ (unless stated otherwise, all homology groups will be taken with integer coefficients).
Moreover the action of $\gamma$ is defined by $\mathcal{A}(\gamma) = \sum_i k_i \int_{\gamma_i}\alpha$.

\subsubsection{Holomorphic curves} \label{Subsection: Holomorphic Curves}

We recall here some definitions and properties about holomorphic curves in dimension $4$. We refer the reader to \cite{McDuff} and
\cite{MS} for the general theory and \cite{Hu3} and \cite{CGH2}-\cite{CGH5} for an approach more specialized to our context. 

Let $X$ be an oriented even dimensional manifold. An \emph{almost complex structure} on $X$ is an isomorphism $J : TX \rightarrow TX$ such that
$J(T_P X) = T_P X$ and $J^2 = -id$. If $(X_1,J_1)$ and $(X_2,J_2)$ are two even dimensional manifolds endowed with an almost complex structure,
a map $u : (X_1,J_1) \rightarrow (X_2,J_2)$ is \emph{pseudo-holomorphic} if it satisfies the \emph{Cauchy-Riemann equation}
$$du \circ J_1 = J_2 \circ du.$$
\begin{Def}
 A \emph{pseudo-holomorphic curve} in a four-dimensional manifold $(X,J)$ is a pseudo-holomorphic map $u : (F,j) \rightarrow (X,J)$, where $(F,j)$
 is a Riemann surface.
\end{Def}
Note that here we do not require that $F$ is connected. 

In this paper we will be particularly interested in pseudo-holomorphic curves (that sometimes we will call simply holomorphic curves) in
``symplectizations'' of contact three manifolds. Let $(Y,\alpha)$ be a contact three-manifold and consider the four-manifold $\R \times Y$. Call
$s$ the $\R$-coordinate and let $R = R_{\alpha}$ be the Reeb vector field of $\alpha$. The almost complex structure $J$ on $\R \times Y$
is \emph{adapted to $\alpha$} if
\begin{enumerate}
 \item $J$ is $s$-invariant;
 \item $J({\xi}) = \xi$ and $J(\partial_s) = R$ at any point of $\R \times Y$;
 \item $J|_{\xi}$ is compatible with $d\alpha$, i.e. $d\alpha(\cdot,J\cdot)$ is a Riemannian metric. 
\end{enumerate}

For us, a holomorphic curve $u$ in the symplectization of $(Y,\alpha)$ is a holomorphic curve $u : (\dot{F},j) \rightarrow (\R \times Y,J)$,
where:
\begin{enumerate}
 \item[i.]    $J$ is adapted to $\alpha$;
 \item[ii.]   $(\dot{F},j)$ is a Riemann surface obtained from a closed surface $F$ by removing a finite number of points (called \emph{punctures});
 \item[iii.]  for any puncture $x$ there exists a neighborhood $U(x) \subset F$ such that $U(x) \setminus \{x\}$  is mapped by $u$ asymptotically to a cover of a cylinder
 $\R \times \delta$ over an orbit $\delta$ of $R$ in a way that $\lim_{y \rightarrow x} \pi_{\R}(u(y)) = \pm \infty$, where
 $\pi_{\R}$ is the projection on the $\R$-factor of $\R \times Y$.
\end{enumerate}
We say that $x$ is a \emph{positive puncture} of $u$ if in the last condition above the limit is $+ \infty$: in this case the orbit $\delta$ 
is a \emph{positive end} of $u$. If otherwise the limit is $-\infty$ then $x$ is a \emph{negative puncture} and $\delta$ is a \emph{negative end}
of $u$.

If $\delta$ is the Reeb orbit associated to the puncture $x$, then $u$ near $x$ determines a cover of $\delta$: the number of sheets of this cover
is the \emph{local $x$-multiplicity of $\delta$ in $u$}. The sum of the $x$-multiplicities over all the punctures $x$ associated to $\delta$
is the \emph{(total) multiplicity} of $\delta$ in $u$. 

If $\gamma$ ($\gamma'$) is the orbit set determined by the set of all the positive (negative) ends of $u$ counted with multiplicity, then
we say that $u$ is a \emph{holomorphic curve from $\gamma$ to $\gamma'$}.

\begin{Ex}
 A \emph{cylinder over an orbit set $\gamma$} of $Y$ is the holomorphic curve $\R \times \gamma \subset \R \times Y$. 
\end{Ex}

\begin{Obs}
 Note that if there exists a holomorphic curve $u$ from $\gamma$ to $\gamma'$, then
 $[\gamma] = [\gamma'] \in H_1(Y,\Z)$. 
\end{Obs}

We state now some result about holomorphic curves that will be useful later.

\begin{Lemma}[see for example \cite{V}] \label{Lemma: Holomorphic curves low the actions of the orbits}
 If $u$ is a holomorphic curve in the symplectization of $(Y,\alpha)$ from $\gamma$ to $\gamma'$, then $\mathcal{A}(\gamma) \geq \mathcal{A}(\gamma')$
 with equality if and only if $\gamma = \gamma'$ and $u$ is a union of covers of a cylinder over $\gamma$.
\end{Lemma}

\begin{Thm}[\cite{MS}, Lemma 2.4.1] \label{Theorem: A holomorphic has only isolated critical points}
 Let $u : (F,j) \rightarrow (\R \times Y,J)$ be a non-constant holomorphic curve in $(X,J)$, then the critical points of
 $\pi_{\R} \circ u$ are isolated. In particular, if $\pi_Y$ denotes the projection $\R \times Y \rightarrow Y$, $\pi_Y \circ u$ is transverse
 to $R_{\alpha}$ away from a set of isolated points.
\end{Thm}

From now on if $u$ is a map with image in $\R \times Y$, we will set $u_{\R} := \pi_{\R} \circ u$ and $u_Y := \pi_Y \circ u$.

Holomorphic curves also enjoy the following property, which will be essential for us: see for example \cite{Gromov}.
\begin{Thm}[Positivity of intersection; Gromov, McDuff, Micallef-White] \label{Theorem: Positivity of intersection in dimension 4}
 Let $u$ and $v$ be two distinct holomorphic curves in a four manifold $(W,J)$. Then $\#(\Imm(u) \cap \Imm(v)) < \infty$. Moreover,
 if $P$ is an intersection point between $\Imm(u)$ and $\Imm(v)$, then its contribution $m_P$ to the algebraic intersection number
 $\langle \Imm(u),\Imm(v) \rangle$ is strictly positive, and $m_P = 1$ if and only
 if $u$ and $v$ are embeddings near $P$ that intersect transversely in $P$.
\end{Thm}

When the almost complex structure does not play an important role or is understood it will be omitted from the notations.


\subsubsection{Morse-Bott theory} \label{Subsection: Morse-Bott Theory}

The Morse-Bott theory in contact geometry has been first developed by Bourgeois in \cite{Bourg}. We present in this subsection some basic
notions and applications, mostly as presented in \cite{CGH2}.

\begin{Def} 
 A \emph{Morse-Bott torus} (briefly M-B torus) in a $3$-dimensional contact manifold $(Y,\alpha)$ is an embedded torus
 $T$ in $Y$ foliated by a family $\gamma_t,\ t\in S^1$, of Reeb orbits, all in the same class in $H_1(T)$, that
 are non-degenerate in the Morse-Bott sense. Here this means the following. Given any $P \in T$ and a positive basis $(v_1,v_2)$ of
 $\xi_P$ where $v_2 \in T_P(T)$ (so that $v_1$ is transverse to $T_P(T)$), then the differential of the first return map of the Reeb
 flow on $\xi_P$ is of the form 
 $$\left(
  \begin{array}{cc}
   1 & 0 \\ a & 1
  \end{array}
  \right)
 $$
 for some $a \neq 0$. If $a > 0$ (resp. $a < 0$) then $T$ is a \emph{positive} (resp. \emph{negative}) M-B torus. 
\end{Def}

We say that $\alpha$ is a \emph{Morse-Bott contact form} if all the Reeb orbits of $\alpha$ are either isolated and non-degenerate or come in
$S^1$-families foliating M-B tori.  

As explained in \cite{Bourg} and \cite[Section 4]{CGH2} it is possible to modify the Reeb vector field in a small neighborhood
of a M-B torus $T$ preserving only two orbits, say $e$ and $h$, of the $S^1$-family of Reeb orbits associated to $T$. 

Moreover, for any fixed $L >0$, the perturbation can be done in a way that $e$ and $h$ are the only orbits in a neighborhood of $T$
with action less then $L$.

If $T$ is a positive (resp. negative) M-B torus and $\tau$ is the trivialization of $\xi$ along the orbits given pointwise by the
basis $(v_1, v_2)$ above, then one can make the M-B perturbation in a way that $h$ is positive hyperbolic with $\mu_{\tau}(h)=0$
and $e$ is elliptic with $\mu_{\tau}(e) = 1$ (resp. $\mu_{\tau}(e) = -1$). 

The orbits $e$ and $h$ can be
seen as the only two critical points of a Morse function $f_T:S^1 \rightarrow \mathbb{R}$ defined on the $S^1$-family of Reeb orbits foliating
$T$ and with maximum corresponding to the orbit with higher C-Z index. Often M-B tori will be implicitly given with such a function.

\begin{Obs} \label{Observation: A M-B perturbation creates non isolated and degenerate orbits}
It is important to remark that, before the perturbation, $T$ is foliated by Reeb 
orbits of $\alpha$ and so these are non-isolated. Moreover the form of the differential of the first return map
of the flow of $\xi$ implies that these orbits are also degenerate.

After the perturbation, $T$ contains only two isolated and non degenerate orbits, but other orbits are created in a neighborhood of $T$ and these
orbits can be non-isolated and degenerate. See Figure \ref{Figure: Dynamic of the M-B perturbation near K} later for an example of M-B perturbations. 
\end{Obs}

\begin{Prop}[\cite{Bourg}, Section 3] For any M-B torus $T$ and any $L \in \R$ there exists a M-B perturbation of $T$ such that, with the exception
of $e$ and $h$, all the periodic orbits in a neighborhood of $T$ have action greater then $L$.
\end{Prop}

A torus $T$ foliated by Reeb orbits all in the same class of $H_1(T)$ (like for example a Morse-Bott torus) can be used to obtain
constraints about the behaviour of a holomorphic curve near $T$.

Following \cite[Section 5]{CGH2}, if $\gamma$ is any of the Reeb orbits in $T$, we can define the \emph{slope of $T$} as the equivalence
class $s(T)$ of $[\gamma] \in H_1(T,\R) - \{0\}$ up to multiplication by positive real numbers. 

Let $T \times [-\epsilon,\epsilon]$ be a neighborhood of $T = T\times \{0\}$ in $Y$ with coordinates $(\vartheta,t,y)$ such that
$(\partial_{\vartheta},\partial_t)$ is a positive basis for $T(T)$ and $\partial_y$ is directed as a positive normal vector to $T$. 

Suppose that $u : (F,j) \rightarrow (\R \times Y,J)$ is a holomorphic curve in the symplectization of $(Y,\alpha)$; by Theorem
\ref{Theorem: A holomorphic has only isolated critical points},
there exist at most finitely many points in $T \times [-\epsilon,\epsilon]$ where $u_Y(F)$ is not transverse to $R_{\alpha}$. Then,
if $T_{y} := T \times \{y\}$ and $u(F)$ intersects $\R \times T_y$, we can associate a slope $s_{T_y}(u)$ to $u_Y(F) \cap T_{y}$,
for any $y \in [-\epsilon,\epsilon]$: this
is defined exactly like $s(T)$, where $u_Y(F) \cap T_{y}$ is considered with the orientation induced by
$\partial \left( u_Y(F) \cap (T \times [-\epsilon,y])\right)$.
\begin{Obs}
 Note that if $u$ has no ends in $T \times [y,y']$, then 
 $$\partial(u_Y(F) \cap T \times [y,y']) = u_Y(F) \cap T_{y'} - u_Y(F) \cap T_{y}$$
 and $s_{T_y}(u) = s_{T_y'}(u)$.
\end{Obs}

The following Lemma is a consequence of the positivity of intersection in dimension four (see \cite[Lemma 5.2.3]{CGH2}).
\begin{Lemma}[Blocking Lemma]  \label{Lemma: Blocking lemma}
 Let $T$ be linearly foliated by Reeb trajectories with slope $s = s(T)$ and $u$ a holomorphic curve be as above.
 \begin{enumerate}
  \item If $u$ is homotopic, by a compactly supported homotopy, to a map whose image is disjoint from $\R \times T$, then
   $u_Y(F) \cap T = \emptyset$.
  \item Let $T \times [-\epsilon,\epsilon]$ be a neighborhood of $T = T \times \{0\}$. Suppose that, for some $y \in [-\epsilon,\epsilon] \setminus \{0\}$,
   $u$ has no ends in $T \times (0,y]$ if $y \in (0,\epsilon]$ or in $T \times [y,0)$ if $y  \in [-\epsilon,0)$. If
   $s_{T_y}(u) = \pm s(T)$ then $u$ has an end which is a Reeb orbit in $T$. 
 \end{enumerate}
\end{Lemma}

Let now $x$ be a puncture of $F$ whose associated end is an orbit $\gamma$ in $T$; if there exists a neighborhood $U(x)$ of $x$ in $F$
such that $u_Y(U(x) \setminus \{x\}) \cap T = \emptyset$ then $\gamma$ is a \emph{one sided end} of $u$ in $x$. This is equivalent to requiring
that $u_Y(U(x))$ is contained either in $T \times (-\epsilon,0)$ or in $T \times (0,\epsilon)$.

The following is proved in \cite{CGH2}
(Lemma 5.3.2).

\begin{Lemma}[Trapping Lemma]  \label{Lemma: Trapping lemma}
 If $T$ is a positive (resp. negative) M-B torus and $\gamma \subset T$ is a one sided end of $u$ associated to the puncture $x$, then $x$ is
 positive (resp. negative).
\end{Lemma}

\subsubsection{Open books} \label{Subsection: Open Books}

\begin{Def}
 Given a surface $S$ and a diffeomorphism $\phi \colon S \rightarrow S$, the \emph{mapping torus} of $(S,\phi)$ is the three dimensional manifold
 $$N(S,\phi) := \frac{S \times [0,2]}{(x,2)\sim (\phi(x),0)}.$$
\end{Def}

In this paper we use the following definition of open book decomposition of a 3-manifold $Y$. This is not the original definition but a more specific
version based on \cite{CGH2}.

\begin{Def}
 An \emph{open book decomposition for $Y$} is a triple $(L,S,\phi)$ such that
 \begin{itemize}
  \item $L = K_1 \sqcup\ldots \sqcup K_n$ is an $n$-component link in $Y$;
  \item $S$ is a smooth, compact, connected, oriented surface with an $n$-components boundary;
  \item $\phi:S \rightarrow S$ is an orientation preserving diffeomorphism such that on a small neighborhood
   $\{1,\ldots,n\} \times [0,1] \times S^1$ of $\partial S =\{1,\ldots,n\} \times \{1\} \times S^1$, with coordinates $(y,\vartheta)$ near each component, it acts by
   \begin{equation}
    (y,\vartheta) \stackrel{\phi}{\longmapsto} (y,\vartheta - y + 1)
    \label{Equation: phi near the binding in CGH}
   \end{equation}
   (and in particular $\phi|_{\partial S} = id_{\partial S}$); 
  \item for each $K_i$ there exists a tubular neighborhood $\mathcal{N}(K_i) \subset Y$ of $K_i$ such that $Y$ is
   diffeomorphic to
   $N(S,\phi) \sqcup_{i=1}^n \mathcal{N}(K_i)$ where the union symbol means that for any $i$, $\{i\} \times \{1\} \times S^1 \times \frac{[0,2]}{0\sim 2}$
   is glued to $\mathcal{N}(K_i)$ in a way that, for any $\vartheta \in S^1$,
   $\{i\} \times \{1\} \times \{\vartheta\} \times \frac{[0,2]}{(0\sim 2)}$ is identified with a meridian of $K_i$ in $\partial \mathcal{N}(K_i)$. 
\end{itemize}

The link $L$ is called the \emph{binding}, the surfaces $S \times \{t\}$ are the \emph{pages} and the diffeomorphism $\phi$ is the
\emph{monodromy} of the open book.
\end{Def}

When we are interested mostly in the mapping torus part of an open book decomposition we will use a notation of the form $(S,\phi)$,
omitting the reference to its binding. Sometimes we will call $(S,\phi)$ an \emph{abstract open book}.

Following \cite{CGH2}, we will often consider each $\mathcal{N}(K_i)$ as a union of a copy of $\frac{[0,2]}{(0\sim 2)}\times [1,2] \times S^1$,
endowed with the extension of the coordinates $(t,y,\vartheta)$, glued along $\{y=2\}$ to a smaller neighborhood $V(K_i)$ of $K_i$.
The gluing is done in a way that the sets $\{\vartheta = \mathrm{const.}\}$ are identified with meridians for $K$ and the sets
$\{t = \mathrm{const.}\}$ are identified to longitudes.

\vspace{0.2 cm}

By the Giroux's work in \cite{Gi} there is a one to one correspondence between contact structures (up to isotopy) and open book decompositions
(up to \emph{Giroux stabilizations}) of $Y$. In order to simplify the notations, we consider here open books with connected binding.

Given $(K,S,\phi)$ we can follow the Thurston-Wilkenkemper construction (\cite{TW}) to associate to it an
\emph{adapted contact form} $\alpha$ on $Y$ as explained in \cite[Section 2]{CGH2}. In $N$ the resulting Reeb vector field $R=R_{\alpha}$ enjoys
the following properties:
\begin{itemize}
 \item $R$ is transverse to the pages $S\times \{t\}$ $\forall t \in [0,2]$;
 \item the first return map of $R$ is isotopic to $\phi$;
 \item each torus $T_y = S^1 \times \frac{[0,2]}{(0\sim 2)} \times \{y\}$, for $y \in [0,1]$, is linearly foliated by Reeb orbits and the first
  return map of $R$ on $T_y$ is
  $$(y,\vartheta) \mapsto (y,\vartheta - y + 1).$$
\end{itemize}

The last implies that when the set of orbits foliating $T_y$ comes in an $S^1$-family, $T$ is Morse-Bott.

To explain the behaviour of $R$ on $\mathcal{N}(K)$, let us extend the coordinates $(\vartheta,t,y)$ to
$V\setminus K \cong \frac{[0,2]}{(0\sim 2)}\times [2,3) \times S^1$, where $K = \{y = 3\}$. For $y \in [0,3)$ set
$T_y =  \frac{[0,2]}{(0\sim 2)}\times \{y\} \times S^1$. Given a curve $\gamma(x) = (\gamma_t(x),y,\gamma_{\vartheta}(x))$
in $T_y$ we can define the \emph{slope} of $\gamma$ in $x_0$ by 
$$s_{T_y}(\gamma,x_0) = \frac{\gamma'_t(x_0)}{\gamma'_{\vartheta}(x_0)} \in \R \cup \{\pm \infty\}.$$
In particular if a meridian has constant slope, this must be $+ \infty$ and $\partial S$ has slope $0$.
Note that the slope of $T_y$ as given by
$$s(T_y) = \frac{\gamma'_t(x)}{\gamma'_{\vartheta}(x)} \in \R \cup \{\pm\infty\},$$
where now $\gamma$ is a parametrization of a Reeb trajectory in $T_y$ and $x \in \Imm(\gamma)$. Note in particular that if $s(T_y)$ is
irrational then $T_y$ does not contain Reeb orbits, and if $T_y$ is foliated by meridians (like $T_1$) then $s(T_y) =+ \infty$.

On $\frac{[0,2]}{(0\sim 2)}\times [1,2] \times S^1$ the contact form will depend on a small real constant $\delta > 0$: call $\alpha_{\delta}$ the contact form
on all $Y$. Let $f_{\delta}:[1,3)\rightarrow \mathbb{R}$ be a smooth function such that:
\begin{itemize}
 \item $f_{\delta}$ has minimum in $y=1.5$ of value $-\delta$;
 \item $f_{\delta}(1)=f_{\delta}(2)=0$;
 \item $f_{\delta}(y)=-y+1$ near $\{y=1\}$;
 \item $f'_{\delta}(y)<0$ for $y \in [1,1.5)$ and $f'_{\delta}(y)>0$ for $y\in (1.5,3)$.
\end{itemize}


Then the Reeb vector field $R$ of $\alpha_{\delta}$ in $\mathcal{N}(K) \setminus int(V)$ is such that:
\begin{itemize}
 \item $R$ is transverse to the annuli $\{t\} \times [1,2] \times S^1$ $\forall t \in \frac{[0,2]}{0\sim 2}$;
 \item the tori $T_y$, $y \in [1,2]$ are foliated by Reeb orbits with constant slope and first return map given by
  $(y,\vartheta) \mapsto (y,\vartheta + f_{\delta}(y))$.
\end{itemize}

Finally in $V$ each torus $T_y$ is linearly foliated by Reeb orbits whose slope vary in $(C,+\infty]$ for $y$ going from $3$
(not included) to $2$ and, where $C$ is a positive real number. Moreover $K$ is also a Reeb orbit.

Note that for every $\delta$, $T_1$ is a negative M-B torus foliated by orbits with constant slope $+\infty$. As explained in
\ref{Subsection: Morse-Bott Theory} we can perturb the associated $S^1$-family of orbits into a pair of simple Reeb orbits $(e,h)$,
where $e$ is an elliptic orbit with C-Z index $-1$ and $h$ is positive hyperbolic with C-Z index $0$ (the indexes are computed with
respect to the trivialization given by the torus).

Similarly the positive M-B torus $T_2$ is also foliated by orbits with constant slope $+\infty$ and a M-B perturbation gives a pair of simple
Reeb orbits $(e_+,h_+)$ in $T_2$, where $e_+$ is elliptic of index $1$ and $h_+$ is hyperbolic of index $0$ (in the papers
\cite{CGH2}-\cite{CGH5} the orbits $e_+$ and $h_+$ are called $e'$ and $h'$ respectively).

\begin{figure} [h] 
  \begin{center}
   \includegraphics[scale = .251]{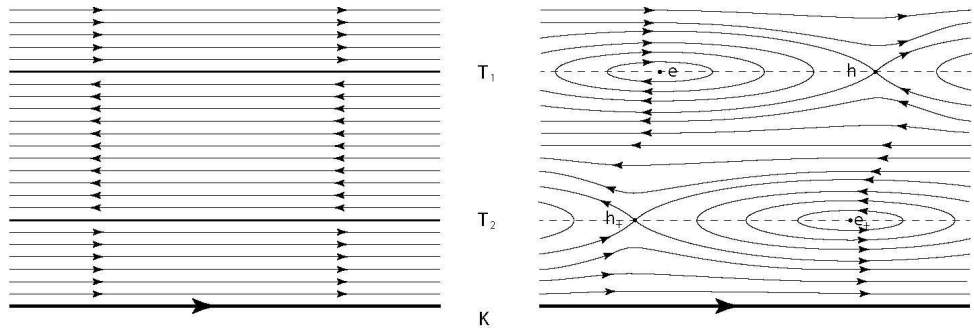}
  \end{center}
  \caption{Reeb dynamic before and after a M-B perturbation of the tori $T_1$ and $T_2$. Both pictures take place in a page of the open book.
   Each flow line represents an invariant subset of $S$ under the Reeb flow near $K$; the orientation gives the direction in which any point is
   mapped under the first return map of the flow.}
  \label{Figure: Dynamic of the M-B perturbation near K}
\end{figure}

In the rest of the paper, if not stated otherwise, when we talk about contact forms and their Reeb vector fields adapted to an open book we
will always refer to them assuming the notations and the properties explained in this subsection. In particular the M-B tori $T_1$ and $T_2$
will be always assumed to be perturbed into the respective pairs of simple orbits. 

\begin{Obs}
 In the case of open books with non-connected binding $L$, the Reeb vector field of an adapted contact form satisfies the same properties above near
 each component of $L$.
\end{Obs}

\vspace{0.3 cm}

We saw that to any open book decomposition $(L,S,\phi)$ of $Y$ it is possible to associate an adapted contact form. Let us now say something
about the inverse map of the Giroux correspondence.
\begin{Thm}[Giroux] \label{Theorem: Construction of an open book adapted to a contact structure}
 Given a contact three-manifold $(Y,\xi)$, there exists an open book decomposition $(L,S,\phi)$ of $Y$
 and an adapted contact form $\alpha$ such that $\ker(\alpha) = \xi$.
\end{Thm}
\proof[Sketch of the proof.]
 Given any contact structure $\xi$ on $Y$, in \cite{Gi} Giroux explicitly constructs an open book decomposition $(L,S,\phi)$ of $Y$ for
 which there
 exists a compatible contact form $\alpha$ such that $\ker(\alpha) = \xi$. Following \cite[Section 3]{Co}, the proof can be carried on in three
 main steps.
 
 The first step consists in providing a cellular decomposition $\mathcal{D}$ of $Y$ that is, in a precise sense, ``\emph{compatible with $\xi$}''.
 It is important to remark that, up to take a refinement (in a way that each $3$-cell is contained in a Darboux ball) any cellular
 decomposition of $Y$ can be isotoped to make it compatible with $\xi$.

 In the second step, $\mathcal{D}$ is used to explicitly build $(L,S,\phi)$. We describe now some of the properties of $S$, seen as the embedded $0$-page
 of the open book.\\
 Let $\mathcal{D}^i$ be the $i$-skeleton of $\mathcal{D}$ and let $\mathcal{N}(\mathcal{D}^1)$ be a tubular neighborhood of $\mathcal{D}^1$. 
 Suppose that $\mathcal{N}(\mathcal{D}^0) \subset \mathcal{N}(\mathcal{D}^1)$ is a tubular neighborhood of $\mathcal{D}^0$ such that 
 $\mathcal{N}(\mathcal{D}^1) \setminus \mathcal{N}(\mathcal{D}^0)$ is homeomorphic to a tubular neighborhood of
 $\mathcal{D}^1 \setminus \mathcal{N}(\mathcal{D}^0)$. Then:
 \begin{enumerate}
  \item $S \subset \mathcal{N}(\mathcal{D}^1)$, $L :=\partial S \subset \partial \mathcal{N}(\mathcal{D}^1)$ and $\mathcal{D}^1 \subset int(S)$;
  \item $S \cap (\mathcal{N}(\mathcal{D}^1) \setminus \mathcal{N}(\mathcal{D}^0))$ is a disjoint union of strips which are diffeomorphic to
   $(\mathcal{D}^1 \setminus \mathcal{N}(\mathcal{D}^0)) \times [-1,1]$ with
   $\mathcal{D}^1 \setminus \mathcal{N}(\mathcal{D}^0)$ corresponding to $(\mathcal{D}^1 \setminus \mathcal{N}(\mathcal{D}^0)) \times \{0\}$;
 \end{enumerate}
 The fact that $\mathcal{D}$ is compatible with $\xi$ implies that $L$ intersects each $2$-simplex exactly twice and it is possible to use this
 fact to prove that the complement of $L$ in $Y$ fibers in circles over $S$, which implies that $L$ is the binding of an open book with $0$-page
 the complement in $S$ of a small neighborhood of $S$.
  
 The third step consists finally in defining the contact form $\alpha$ with the required properties.
\endproof

\begin{Thm}[Giroux correspondence] \label{Theorem: Giroux correspondance}
 Let $\alpha$ and $\alpha'$ be contact structures on $Y$ that are adapted to the open books $(L,S,\phi)$ and, respectively, $(L',S',\phi')$.
 Then $\alpha$ and
 $\alpha'$ are isotopic if and only if $(L',S',\phi')$ can be obtained from $(L,S,\phi)$ by a sequence of Giroux stabilizations and destabilizations.
\end{Thm}

A \emph{Giroux stabilization of an open book} is an operation that associates to an
open book decomposition $(L,S,\phi)$ of $Y$ another open book decomposition $(L',S',\phi')$ of $Y$, obtained as follows. Choose two points $P_1$ and $P_2$ in $\partial S$
(not necessarily in the same connected
component) and let $\gamma$ be an oriented embedded path in $S$ from $P_1$ to $P_2$. Let now $S'$ be the oriented surface obtained by attaching
a $1$-handle to $S$ along the attaching sphere $(P_1,P_2)$. Consider the closed oriented loop $\bar{\gamma} \subset S'$ defined by 
$\bar{\gamma} := \gamma \sqcup c$, where $c$ is the core curve of the $1$-handle, oriented from $P_2$ to $P_1$, and the gluing is done along
the common boundary ${(P_1,P_2)}$ of the two paths.

By the definition of monodromy of open book that we gave, the $\phi$ is the identity along $\partial S$. So $\phi$ extends to the identity map on the handle:
we keep calling $\phi$ the resulting diffeomorphism on $S'$. If $\tau_{\bar{\gamma}}$ is a positive Dehn twist along $\bar{\gamma}$, define
$\phi' = \tau_{\bar{\gamma}} \circ \phi$. 

It results that $N' := N(S',\phi')$ embeds in $Y$ and that $Y \setminus N'$ is a disjoint union of solid tori. Then, if $L'$ is the set of
the core curves of these tori, $(L',S',\phi')$ is an open book decomposition of $Y$, which is said to be obtained by
\emph{Giroux stabilization} of $(L,S,\phi)$ along $\gamma$.

There is an obvious inverse operation of the stabilization: with the notations above, we say that
$(L,S,\phi)$ is obtained by \emph{Giroux destabilization} of $(L',S',\phi')$ along $\gamma'$.

Note that a Giroux stabilization does not change the components of $L$ that do not intersect the attaching sphere. 
Moreover it is not difficult to see that the number of connected components of $L$ and $L'$ differs by $1$: if $P_1$ and $P_2$ are chosen
in the same component then $L'$ has one component more than $L$; otherwise $L'$ has one component less then $L$.


\subsection{$ECH$ for closed three-manifolds} \label{Subsection: ECH for three-manifolds}

We briefly remind here the Hutchings' original definition of $ECH(Y,\alpha)$ and $\widehat{ECH}(Y,\alpha)$ for a closed contact three-manifold
$(Y,\alpha)$.




Let $(Y,\alpha)$ be a closed contact three-manifold and assume that $\alpha$ is non-degenerate (i.e., that any Reeb orbit of $\alpha$
is non-degenerate).

For a fixed $\Gamma \in H_1(Y)$, define $ECC(Y,\alpha,\Gamma)$ to be the free $\mathbb{Z}_2$-module generated by the orbit sets of $Y$ (recall
Definition \ref{Definition: Orbit sets}) in the homology class $\Gamma$ and pose 
$$ECC(Y,\alpha) = \bigoplus_{\Gamma \in H_1(Y)}ECC(Y,\alpha,\Gamma).$$
This is the \emph{$ECH$ chain group of $(Y,\alpha)$}.

The \emph{$ECH$-differential} $\partial^{ECH}$ (called simply $\partial$ when no risk of confusion occurs) is defined in \cite{Hu2} in
terms of holomorphic curves in the symplectization $(\R \times Y,d\alpha, J)$ of $(Y,\alpha)$ as follows.

Given $\gamma, \delta \in \mathcal{O}(Y)$, let $\mathcal{M}(\gamma, \delta)$ be the set of (possibly disconnected) holomorphic curves
$u : (\dot{F},j) \rightarrow (\R \times Y, J)$ from $\gamma$ to $\delta$, where $(\dot{F},j)$ is a punctured compact Rieamannian surface.
It is clear that $u$ determines a relative homology class $[\Imm(u)] \in H_2(\R \times Y;\gamma,\delta)$ and that if such a curve exists then
$[\gamma] = [\delta] \in H_1(Y)$.

If $\xi = \ker(\alpha)$ and a trivialization $\tau$ of $\xi|_{\gamma \cup \delta}$ is given, to any surface $C \subset \R \times Y$ with
$\partial C = \gamma - \delta$ it is possible to associate an \emph{$ECH$-index}
$$I(C) := c_{\tau}(C) + Q_{\tau}(C) + \mu_{\tau}^I(\gamma,\delta),$$
which depends only on the relative homology class of $C$. Here 
\begin{itemize}
 \item $c_{\tau}(C) := c_1(\xi|_C,\tau)$ is the \emph{first relative Chern class} of $C$;
 \item $Q_{\tau}(C)$ is the \emph{$\tau$-relative intersection paring} of $\R \times Y$ applied to $C$;
 \item $\mu_{\tau}^I(\gamma,\delta) := \sum_i \sum_{j=1}^{k_i}\mu_{\tau}(\gamma_i^{j}) - \sum_i \sum_{j=1}^{k_i}\mu_{\tau}(\delta_i^{j})$, where
  $\mu_{\tau}$ is the Conley-Zehnder index defined in Section \ref{Subsection: Contact geometry}.
\end{itemize}
We refer the reader to \cite{Hu3} for the details about these quantities. If $u$ is a holomorphic curve from $\gamma$ to $\delta$ set
$I(u) = I(\Imm(u))$ (well defined up to approximating $\Imm(u)$ with a surface in the same homology class).

Define $\mathcal{M}_k(\gamma, \delta) := \{u \in \mathcal{M}(\gamma, \delta)\ |\ I(u) = k\}$. The $ECH$-differential is then defined on the generators of
$ECC(Y,\alpha)$ by
\begin{equation} \label{Equation: ECH differential}
 \partial^{ECH}(\gamma) = \sum_{\delta \in \mathcal{O}(Y)}\sharp \left(\frac{\mathcal{M}_1(\gamma, \delta)}{\R}\right)\cdot \delta
\end{equation}
where the fraction means that we quotient $\mathcal{M}_1(\gamma, \delta)$ by the $\R$-action on the curves given by the translation in the $\R$-direction in
$\R \times Y$. In \cite[Section 5]{Hu3} Hutchings proves that $\frac{\mathcal{M}_1(\gamma, \delta)}{\R}$ is a compact $0$-dimensional manifold, so that 
$\partial^{ECH}(\gamma)$ is well defined.

The \emph{(full) embedded contact homology of $(Y,\alpha)$} is
$$ECH_*(Y,\alpha) := H_*(ECC(Y,\alpha),\partial^{ECH}).$$
It turns out that these groups do not depend either on the choices $J$ in the symplectization or the contact form for $\xi$.

It is possible to endow $ECH(Y,\xi)$ with a
canonical absolute $\Z/2$-grading as follows. If $\gamma = \prod_i \gamma_i^{k_i}$ set
$$\epsilon(\gamma) = \prod_i \epsilon (\gamma_i)^{k_i},$$
where $\epsilon (\gamma_i)$ is the Lefschetz sign of the simple orbit $\gamma_i$. Note that $\epsilon(\gamma)$ is given by the parity of the number of positive
hyperbolic simple orbits in $\gamma$.

If $u$ is a holomorphic curve from $\gamma$ to $\delta$, by simple computations it is possible to prove the following
\emph{index parity formula} (see for example Section 3.4 in \cite{Hu3}):
\begin{equation} \label{Equation: Index parity formula of the ECH index}
 (-1)^{I(u)} = \epsilon(\gamma)\epsilon(\delta).
\end{equation}

It follows then that the Lefschetz sign endows embedded contact homology with a well defined absolute grading.

\vspace{0.3 cm}

Fix now a generic point $(0,z) \in \R \times Y$. Given two orbit sets $\gamma$ and $\delta$, let
\begin{equation*}
 \begin{array}{cccc}
  U_z : & ECC_*(Y,\alpha) & \longrightarrow & ECC_{*-2}(Y,\alpha) 
 \end{array}
\end{equation*}
be the map defined on the generators by
\begin{equation*} 
  U_z(\gamma) =  \sum_{\delta \in \mathcal{O}(Y)}\# \left\{u \in \mathcal{M}_2(\gamma, \delta)\ |\ (0,z) \in \Imm (u) \right\}\cdot \delta.
\end{equation*}
Hutchings proves that $U_z$ is a chain map that counts only a finite number of holomorphic curves and that this count does not depend on the
choice of $z$. So it makes sense to define the map $U := U_z$ for any $z$ as above. This is called the \emph{U-map}.

The \emph{hat version of embedded contact homology of $(Y,\alpha)$} is defined as the homology $\widehat{ECH}(Y,\alpha)$ of the mapping cone of
the U-map. By this we mean that $\widehat{ECH}(Y,\alpha)$ is defined to be the homology of the chain complex
$${ECC}_{*-1}(Y,\alpha) \oplus {ECC}_*(Y,\alpha)$$
with differential defined by the matrix
\begin{equation} \nonumber
 \left(\begin{array}{cc}
        -\partial_{*-1} & 0 \\
        U               & \partial_*
   \end{array}\right)
\end{equation}
where the element of the complex are thought as columns. Also $\widehat{ECH}(Y,\alpha)$ has the relative and the absolute gradings above.

\begin{Obs} \label{Observation: ECH groups split in direct sum over H_1 Y}
Note that $\partial^{ECH}$ and $U$ respect the homology class of the generators of $ECC_*(Y,\alpha)$. This implies that there are natural splits:
\begin{equation} \label{Equation: splits of ECH in terms of homology classes}
 \begin{array}{ccc}
   ECH(Y,\xi) & = & \bigoplus_{\Gamma \in H_1(Y)}ECC(Y,\xi,\Gamma);\\
   \widehat{ECH}(Y,\xi) & = & \bigoplus_{\Gamma \in H_1(Y)}\widehat{ECH}(Y,\xi,\Gamma).
 \end{array}
\end{equation}
\end{Obs}

We end this section by stating the following result (see for example \cite{Hu3}).
\begin{Thm}
 Let $\emptyset$ be the empty orbit. Then $[\emptyset] \in ECH(Y,\xi)$ is an invariant of the contact structure $\xi$.
\end{Thm}
The class $[\emptyset]$ is called $ECH$ \emph{contact invariant} of $\xi$.

\subsection{$ECH$ for manifolds with torus boundary}
\label{Subsection: ECH for manifolds with boundary}

In order to define $ECH$ for contact three-manifolds $(N,\alpha)$ with nonempty boundary, some compatibility between $\alpha$ and $\partial N$
should be assumed. In this paper we are particularly interested in three-manifolds whose boundary is a collection of disjoint tori.

In \cite[Section 7]{CGH2} Colin, Ghiggini and Honda analyze this situation when $\partial N$ is connected. 
If $\mathcal{T} = \partial N$ is homeomorphic to a torus, then they prove that the $ECH$-complex and the
differential can be defined almost as in the closed case, provided that $R = R_{\alpha}$ is tangent to $\mathcal{T}$ and that $\alpha$
is non-degenerate in $int(N)$. 

If the flow of $R|_{\mathcal{T}}$ is irrational they define $ECH(N,\alpha) = ECH(int(N),\alpha)$
while, if it is rational, they consider the case of $\mathcal{T}$ Morse-Bott and do a M-B perturbation of $\alpha$ near
$\mathcal{T}$;
this gives two Reeb orbits $h$ and $e$ on $\mathcal{T}$ and, since $\alpha$ is now a M-B contact form, the
$ECH$-differential counts special holomorphic curves, called \emph{M-B buildings}. 

\begin{Def}
 Let $\alpha$ be a Morse-Bott contact form on the three manifold $Y$ and $J$ a regular almost complex structure on $\mathbb{R} \times Y$.
 Suppose that any
 M-B torus $T$ in $(Y,\alpha)$ comes with a fixed a Morse function $f_T$. Let $\mathcal{P}(Y)$ be the set of simple Reeb orbits in
 $Y$ minus the set of the orbits which correspond to some regular point of some $f_T$.  
 
 A \emph{Morse-Bott building in $(Y,\alpha)$} is a disjoint union of objects $u$ of one of the following two types:

 \begin{enumerate}
  \item $u$ is the submanifold of a M-B torus $T$ corresponding to a gradient flow line of $f_T$: in this case the positive and negative end of $u$ are the
   positive and, respectively, the negative end of the flow line;
  
  \item $u$ is a union of curves $\widetilde{u} \cup u_1 \cup \ldots \cup u_n$ of the following kind. $\widetilde{u}$ is a $J$-holomorphic curve
   in $\mathbb{R} \times Y$ with $n$ ends $\{\delta_1,\ldots,\delta_n\}$ corresponding to regular values of some $\{f_{T_1},\ldots,f_{T_n}\}$.
   Then, for each $i$, $\widetilde{u}$ is augmented by a gradient flow trajectory $u_i$
   of $f_{T_i}$: $u_i$ goes from the maximum $\epsilon^+_i$ of $f_{T_i}$ to $\delta_i$ if $\delta_i$ is a positive end and goes from $\delta_i$ to the minimum
   $\epsilon^-_i$ of $f_{T_i}$ if $\delta_i$ is a negative end. The ends of $u$ are obtained from the ends of $\widetilde{u}$ by substituting
   each $\delta_i$ with the respective $\epsilon^+_i$ or $\epsilon^-_i$.
 \end{enumerate}
 
 A Morse-Bott building is \emph{nice} the $\widetilde{u}$ above has at most one connected component which is not a cover of a trivial cylinder.
\end{Def}


Suppose now that $Y$ is closed and $\mathcal{N} \cong D^2 \times S^1$ is a solid torus embedded in $Y$. If $N = Y \setminus int(\mathcal{N})$, under some assumption
on the behaviour of $\alpha$ in a neighborhood of $\mathcal{N}$, in \cite{CGH2} the authors define relative versions $ECH(N,\partial N,\alpha)$ and
$\widehat{ECH}(N,\partial N,\alpha)$ of embedded contact homology groups and prove that
\begin{eqnarray}
 ECH(N,\partial N,\alpha) \cong ECH(Y,\alpha); \label{Equation: ECH in terms of the relative version}\\
 \widehat{ECH}(N,\partial N,\alpha) \cong \widehat{ECH}(Y,\alpha). \label{Equation: ECH widehat in terms of the relative version}
\end{eqnarray}
The notation suggests that these new homology groups are obtained by counting only orbits in $N$ and quotienting by orbits on $\partial N$.
Let us see the definition of these homologies in more details.

As mentioned, to define these versions of embedded contact homology and prove the isomorphisms above, some compatibility between $\alpha$ and $\mathcal{N}$ is required.
We refer the
reader to \cite[Section 6]{CGH2} for the details. Essentially two conditions are required. The first one fixes $\alpha$ near
$\mathcal{N}$ in a way that $R$ behaves similarly to the Reeb vector field defined in Section \ref{Subsection: Open Books} near $\mathcal{N}(K)$,
where $K$ was the binding of an open book decomposition of $Y$.\\

Briefly, this means that there exists a smaller closed solid torus $V \subset \mathcal{N}$ and a neighborhood $T^2 \times [0,2]$ of
$\partial \mathcal{N} = T^2 \times \{1\}$ in $Y$ such that:
\begin{enumerate}
 \item $T^2 \times [0,1] \subset N$, $\mathcal{N} = (T^2 \times [1,2]) \cup V$ and $\partial V = T^2 \times \{2\}$; 
 \item $T^2 \times \{y\}$ is foliated by Reeb trajectories for any $y \in [0,2]$;
 \item if $K = \{0\} \times S^1 \subset \mathcal{N}$, then $K$ is a Reeb orbit and $\mathrm{int}(V) \setminus K$ is foliated by concentric tori,
  which in turn are linearly foliated by Reeb trajectories that intersect positively a meridian disk for $K$ in $V$.
 \item $T_1 := T^2 \times \{1\}$ and $T_2 := T^2 \times \{2\}$ are negative and, respectively, positive M-B tori foliated by Reeb
  orbits which are meridians of $K$.  
\end{enumerate}
As in Subsection \ref{Subsection: Open Books}, the families of Reeb orbits in $T_1$ and $T_2$ are perturbed into two pairs of Reeb orbits $(e,h)$ and, respectively,
$(e_+,h_+)$: here $e$ and $e_+$ are elliptic and $h$ and $h_+$ are positive hyperbolic (see figure \ref{Figure: Dynamic of the M-B perturbation near K}).
If $\alpha$ satisfies the conditions above we say that \emph{$\alpha$ is adapted to $K$}.

The second condition of compatibility is that there must exist a Seifert surface $S \subset Y$ for $K$ such that $R$
is positively transverse to $int(S)$. In this case we say that $\alpha$ is \emph{adapted to $S$}.
\begin{Lemma}(see Theorem 10.3.2 in \cite{CGH2}) \label{Lemma: There exists alpha copatible with K and S}
 Given a null-homologous knot $K$ and a contact structure $\xi$ on $Y$ there exists a contact form $\alpha$ for $\xi$ and a genus minimizing
 Seifert surface $S$ for $K$ such that:
 \begin{enumerate}
  \item $\alpha$ is adapted to $K$;
  \item $\alpha$ is adapted to $S$.
 \end{enumerate}
\end{Lemma}
\proof
 We give here only the proof of 1), referring the reader to \cite{CGH2} for 2). 
 
 Up to isotopy, we can assume that $K$ is transverse to $\xi$ and let $\alpha'$ be any contact form for $\xi$. Up to isotopy of $\alpha'$ we can
 suppose that $K$ is a Reeb orbit. Since the compatibility condition with $K$ can be arranged on a neighborhood of $K$, by the Darboux-Weinstein
 neighborhood theorem (see for example \cite{Ge}) there exists a contact form $\alpha$ which is compatible with $K$ and contactomorphic to $\alpha'$.
\endproof

\begin{Ex} If $(K,S,\phi)$ is an open book decomposition of $Y$ and $\alpha$ is a contact form adapted to $(K,S,\phi)$, then it is
 adapted also to $K$ and to any page of $(K,S,\phi)$.
\end{Ex}

In \cite{CGH2} the authors prove that it is possible to define the $ECH$-chain groups without taking into
account the orbits in $int(V)$ and in $T^2 \times (1,2)$, so that the only interesting orbits in $\mathcal{N}(K)$ are the four orbits above
(plus, obviously, the empty orbit). Moreover the only curves counted by the (restriction of the) $ECH$-differential
$\partial$ have projection on $Y$ as depicted in figure \ref{Figure: Dynamic near binding in W}. These curves give the following relations:
\begin{equation}
 \begin{array}{ccc}
  \partial(e) & = & 0 \\
  \partial(h) & = & 0 \\
  \partial(h_+) & = & e + \emptyset \\
  \partial(e_+) & = & h.
 \end{array}
 \label{Equation: ECH boundary near the binding}
\end{equation}
Note that the two holomorphic curves from $h$ to $e$, as well as the two from $e_+$ to $h_+$, cancel one each other since we work with
coefficients in $\Z/2$.
\begin{Obs}
 The compactification of the projection of the holomorphic curve that limits to the empty orbit is topologically a disk with boundary $h_+$,
 which should be seen as
 a cylinder closing on some point of $K$. This curve contribute to the ``$\emptyset$ part'' of the third of the equations above, which gives
 $[e] = [\emptyset]$ in $ECH$-homology.
 In the rest of this manuscript the fact that this disk is the only $ECH$ index $1$ connected holomorphic curve that crosses $K$ will be
 essential.
\end{Obs}
 \begin{figure} [h] 
  \begin{center}
   \includegraphics[scale = .35]{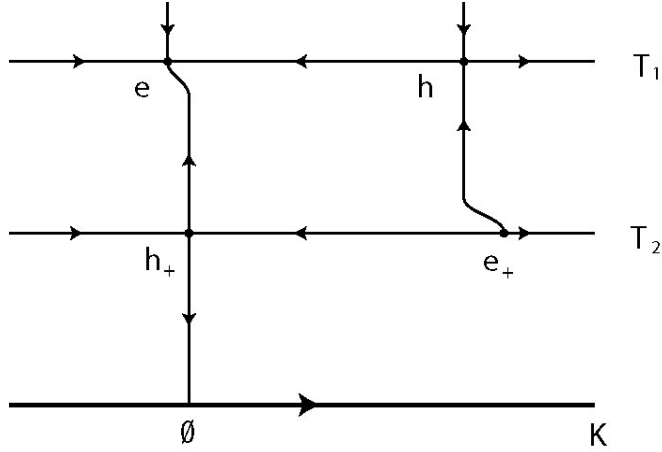}
  \end{center}
  \caption{Orbits and holomorphic curves near $K$. Here the marked points denote the simple Reeb orbits and the flow lines represent
   projections of the holomorphic curves counted by $\partial^{ECH}$. The two flow lines arriving from the top on $e$ and $h$ are depicted
   only to remember that, by the Trapping Lemma, holomorphic curves can only arrive to $T_1$.}
  \label{Figure: Dynamic near binding in W}
\end{figure}

\begin{Notation}
 From now on we will use the following notation. If $(Y,\alpha)$ is understood, given a submanifold $X \subset Y$ and a set of Reeb
 orbits $\{\gamma_1,\ldots,\gamma_n\} \subset \mathcal{P}(Y\setminus X)$, we will denote $ECC^{\gamma_1,\ldots,\gamma_n}(X,\alpha)$ the free $\mathbb{Z}/2$-module
 generated by orbit sets in $\mathcal{O}(X \sqcup \{\gamma_1,\ldots,\gamma_n\})$.
 
 Unless stated otherwise, the group $ECC^{\gamma_1,\ldots,\gamma_n}(X,\alpha)$ will
 come with the natural restriction, still denoted $\partial^{ECH}$, of the $ECH$-differential of $ECC(Y,\alpha)$: if this restriction
 is still a differential the associated homology is
 $$ECH^{\gamma_1,\ldots,\gamma_n}(X,\alpha) := H_*(ECC^{\gamma_1,\ldots,\gamma_n}(X,\alpha),\partial^{ECH}).$$
 
 This notation is not used in \cite{CGH2}, where the authors introduced a specific notation for each relevant $ECH$-group. In particular with their notation:
 \begin{eqnarray*}
  ECC^{\flat}(N,\alpha) & = & ECC^{e}(int(N),\alpha);\\
  ECC^{\sharp}(N,\alpha) & = & ECC^{h}(int(N),\alpha);\\
  ECC^{\natural}(N,\alpha) & = & ECC^{h_+}(N,\alpha).
 \end{eqnarray*}
\end{Notation}

As mentioned before, even if in $\mathcal{N}$ there are other Reeb orbits, it is possible to define chain complexes for the $ECH$ homology 
of $(Y,\alpha)$ only taking into account the orbits $\{e,h,e_+,h_+\}$. 

The Blocking and Trapping lemmas and the relations above imply that the restriction of the full $ECH$-differential of $Y$ to 
the chain group $ECH^{e_+,h_+}(N,\alpha)$ is given by:
\begin{equation} \label{Equation: The complete ECH differential with all the orbits}
 \partial(e_+^a h_+^b \gamma) =  e_+^{a-1} h_+^b h \gamma + e_+^a h_+^{b-1} (1+e) \gamma + e_+^a h_+^b \partial \gamma,
\end{equation}
where $\gamma \in \mathcal{O}(N)$ and a term in the sum is meant to be zero if it contains some elliptic orbit with negative total
multiplicity or a hyperbolic orbit with total multiplicity not
in $\{0,1\}$ (see \cite[Section 9.5]{CGH2}). We remark that the Blocking Lemma implies also that
$\partial \gamma \in \mathcal{O}(N)$.

The further restriction of the differential to $ECH^{h_+}(N,\alpha)$ is then given by
\begin{equation} \label{Equation: The hat ECH differential with all the orbits}
 \partial(h_+^b \gamma) = h_+^{b-1} (1+e) \gamma + h_+^b \partial \gamma.
\end{equation}

Combining the computations of sections 8 and 9 of \cite{CGH2} the authors get the following result.
\begin{Thm} \label{Theorem: ECH of Y iso to the relative versions}
 Suppose that $\alpha$ is adapted to $K$ and there exists a Seifert surface $S$ for $K$ such that $\alpha$ is adapted to $S$. Then
 \begin{eqnarray}
  ECH(Y,\alpha) & \cong & ECH^{e_+,h_+}(N,\alpha) \label{Equation: ECH Y in terms of N e+ and h+};\\
  \widehat{ECH}(Y,\alpha) & \cong & ECH^{h_+}(N,\alpha). \label{Equation: widehat ECH Y in terms of N and h+}
 \end{eqnarray}
\end{Thm}
\begin{Obs} 
 It is important to remark that the empty orbit is always taken into account as a generator of the groups above. This implies that if
 orbit sets with $h_+$ are
 considered, $\partial^{ECH}$ counts also the holomorphic plane that contributes to the third of relations
 \ref{Equation: ECH boundary near the binding}. Later we will give the definition of another
 differential, that we will call $\partial^{ECK}$, which is obtained from $\partial^{ECH}$ by simply deleting that disk. 
\end{Obs}

\vspace{0.3 cm}

Define now the \emph{relative embedded contact homology groups of $(N, \partial N)$} by
\begin{eqnarray*}
 ECH(N,\partial N,\alpha) & = & \frac{ECH^{e}(int(N),\alpha)}{[e \gamma] \sim [\gamma]}\\
 \widehat{ECH}(N,\partial N,\alpha) & = & \frac{ECH(N,\alpha)}{[e \gamma] \sim [\gamma]}.
\end{eqnarray*}
Since $h_+$ does not belong to the complexes $ECC^{e}(int(N),\alpha)$ and $ECC(N,\alpha)$, the Blocking Lemma implies that the
$ECH$-differentials count only holomorphic curves in $N$. This ``lack'' is balanced by the quotient by the equivalence relation
\begin{equation} \label{Equation: equivalence relation on ECC}
 [e \gamma] \sim [\gamma].
\end{equation}
The reason behind this claim lie in the third of the relations \ref{Equation: ECH boundary near the binding}. Indeed we can prove the following:
\begin{Lemma} \label{Lemma: How kill h+ and get the quotient e gamma sim gamma}
 $$ECH^{e_+,h_+}(N,\alpha) \cong \frac{ECH^{e_+}(N,\alpha)}{[e\gamma] \sim [\gamma]}.$$
\end{Lemma}
\proof
Using the fact that $h_+$ can have multiplicity at most $1$, it is not difficult to see that the long
exact homology sequence associated to the pair
$$\left(ECC^{e_+}(N,\alpha),ECC^{e_+,h_+}(N,\alpha)\right)$$
is
\begin{equation*}
 \begin{array}{c}
  \ldots \longrightarrow  ECH^{e_+}(N,\alpha) \stackrel{i_*}{\longrightarrow } ECH^{e_+,h_+}(N,\alpha) \stackrel{\pi_*}{\longrightarrow } \\
   \stackrel{\pi_*}{\longrightarrow } H(h_+ECC^{e_+}(N,\alpha),\partial) \stackrel{d}{\longrightarrow }ECH^{e_+}(N,\alpha) \stackrel{i_*}{\longrightarrow } \ldots
 \end{array}
\end{equation*}
where:
\begin{itemize}
 \item $i : ECC^{e_+}(N,\alpha) \hookrightarrow ECC^{e_+,h_+}(N,\alpha)$ is the inclusion map;
 \item $h_+ECC^{e_+}(N)$ is the module generated by orbit sets of the form $h_+ \gamma$ with $\gamma \in \mathcal{O}(N \sqcup e_+)$;
 \item $\pi : ECC^{e_+,h_+}(N,\alpha) \twoheadrightarrow h_+ECC^{e_+}(N)$ is the quotient map sending to $0$ all generators having
  no contributions of $h_+$;
 \item $d$ is the standard connecting morphism, that in this case is defined by
 $$d([h_+ \gamma]) = [\gamma + e\gamma].$$
\end{itemize}

We can then extract the short exact sequence 
$$0 \longrightarrow \mathrm{coker}(d) \stackrel{i_*}{\longrightarrow} ECH^{e_+,h_+}(N,\alpha) \stackrel{\pi_*}{\longrightarrow} \ker(d) \longrightarrow 0$$
where
$$\mathrm{coker}(d) = \frac{ECH^{e_+}(N,\alpha)}{[e\gamma] \sim [\gamma]}.$$
Since $\ker(d)=\{0\}$, the map $i_*$ is an isomorphism.

\endproof

Similarly, the fourth line of Equation \ref{Equation: ECH boundary near the binding} ``explains'' why we can
avoid considering $h$ in the full $ECH(Y,\alpha)$. In fact with similar arguments of the proof of last lemma, one can prove:
\begin{Lemma}[\cite{CGH2}, Section 9]
 $ECH^{e_+}(N,\alpha) \cong ECH^{e}(int(N),\alpha)$.
\end{Lemma}

Observe that since $\partial(e\gamma) = e\partial(\gamma)$, the differential is compatible with the equivalence relation. So, instead
of take the quotient by $[e \gamma] \sim [\gamma]$ of the homology, we could take the homology of the quotient of the chain groups
under the relation $e\gamma \sim \gamma$, and we would obtain the same homology groups. We will use this fact later. Note moreover
that for every $k$, $[e^k] = [\emptyset]$.

Equations \ref{Equation: ECH in terms of the relative version} and \ref{Equation: ECH widehat in terms of the relative version} follow then
from last two lemmas and Theorem \ref{Theorem: ECH of Y iso to the relative versions}.

\subsubsection{$ECH$ and $\widehat{ECH}$ from open books}
\label{Subsubsection: ECH from open books}

An important example of the situation depicted above is when $K$ is the binding of an open book decomposition $(K,S,\phi)$ of a closed 
three manifold $Y$, and $N$ is the associated mapping torus considered in Subsection \ref{Subsection: Open Books}. Using the same notations,
define the \emph{extended pages} of $(S,\phi)$ to be the surfaces
$$S' \times \{t\} := (S \times \{t\}) \sqcup_{\partial S \times \{t\}} (S^1 \times \{t\} \times [1,3)),\ t \in \frac{[0,2]}{0\sim 2}.$$
Let $\alpha$ be a contact form on $Y$ compatible with $(K,S,\phi)$. In particular $\alpha$ is adapted to both $K$ and any page of $(K,S,\phi)$.
\begin{Def} \label{Definition: Degree of multi orbits}
 If $\gamma$ is a Reeb orbit in $Y \setminus K$, define the \emph{degree} of $\gamma$ by 
 $$\Deg(\gamma) = \langle \gamma, S' \times \{0\}\rangle$$
 If $\gamma=\prod_i\gamma_i^{k_i}$ is some orbit set, we define $\Deg(\gamma)= \sum_i k_i \Deg(\gamma_i)$. If $X \subset (Y \setminus K)$,
 we indicate by $\mathcal{O}_{i}(X)$ (resp. $\mathcal{O}_{\leq i}(X)$) the set of multiorbits in $X$ with degree equal (resp. less or equal) to $i$. 
\end{Def}
Note that $\Deg(\gamma)$ depends only on the homology class of $\gamma$ in $Y \setminus K$. In this context the relative embedded contact
homology groups can also be defined in terms of limits as follows.

Define $ECC_j^{e}(int(N),\alpha)$ to be the free $\mathbb{Z}_2$-module generated by orbit sets in $\mathcal{O}_j(int(N) \cup \{e\})$. Similarly
let $ECC_j(N,\alpha)$ be generated by orbit sets in $\mathcal{O}_j(N)$. Define the inclusions
\begin{eqnarray*}
 \mathcal{I}_j^{e} :& ECC_j^{e}(int(N),\alpha) \rightarrow ECC_{j+1}^{e}(int(N),\alpha)\\
 \mathcal{I}_j :& ECC_j(N,\alpha) \rightarrow ECC_{j+1}(N,\alpha)
\end{eqnarray*}
given by the map $\gamma \mapsto e\gamma.$ Each of these chain groups can be endowed with (the restriction of) the $ECH$-differential, which
counts M-B buildings in $N$.
Let $ECH_j^{e}(N,\alpha)$
and $ECH_j(N,\alpha)$ be the associated homology groups. Then the relative embedded contact homology groups above can be defined also by
\begin{eqnarray*}
 ECH(N,\partial N,\alpha) & = & \lim_{j \rightarrow \infty} ECH_j^{e}(int(N),\alpha);\\
 \widehat{ECH}(N,\partial N,\alpha) & = & \lim_{j \rightarrow \infty} ECH_j(N,\alpha).
\end{eqnarray*}

\begin{Obs}
 If $ECC_{\leq k}(N,\alpha) := \bigoplus_{j=0}^{k} ECC_j(N,\alpha)$, let $ECH_{\leq k}(N,\alpha)$ be the homology of $ECC_{\leq k}(N,\alpha)$
 with the $ECH$-boundary map. The ``stabilization'' Theorem 1.0.2 of \cite{CGH4} implies that for the definition of $\widehat{ECH}(N,\partial N,\alpha)$ it is sufficient
 to take into account just orbit sets in $\mathcal{O}_{\leq 2g}(N)$. Then:
\begin{equation}
 \widehat{ECH}(N,\partial N,\alpha) \cong \frac{ECH_{\leq 2g}(N,\alpha)}{[e\gamma] \sim [\gamma]}.
 \label{Equation: Equivalence ECH N,partial N,alpha and ECH_2g N,partial N}
\end{equation}
 
\end{Obs}

\subsection{$\widehat{ECH}$ for knots} \label{Subsection: widehat ECH for knots}

Let $K$ be a homologically trivial knot in a contact three-manifold $(Y,\alpha)$. In this subsection we recall the definition 
of a hat version of contact homology for the triple $(K,Y,\alpha)$. This was first defined in \cite[Section 7]{CGHH} as a particular
case of \emph{sutured} contact homology. On the other hand, following \cite[Section 10]{CGH2}, it is possible to proceed without dealing
directly with sutures: we follow here this approach.

Let $S$ be a Seifert surface for $K$. By standard arguments in homology, it is easy to compute that 
\begin{equation} \label{Equation: Isomorphism of H_1 Y minus K to H_1 Y times Z}
 \begin{array}{ccc}
   H_1(Y \setminus K) & \longrightarrow & H_1(Y) \times \mathbb{Z} \\
      \ [a]     &   \longmapsto    &  (i_*[a]\ ,\langle a , S \rangle)
 \end{array}
\end{equation}
is an isomorphism. Here $i : Y \setminus K \rightarrow Y$ is the inclusion and $\langle a , S \rangle$ denotes the intersection number
between $a$ and $S$: this is a homological invariant of the pair $(a,S)$ and is well defined up to a slight perturbation of $S$ (to make it transverse
to $a$). Note that a preferred
generator of $\Z$ is given by the homology class of a meridian for $K$, positively oriented with respect to the orientations of $S$ and $Y$.  

\begin{Ex} If $Y$ is a homology three-sphere, the number $\langle a , S \rangle$ depends only on $a$ and $K$. This is the
\emph{linking number between $a$ and $K$} and it is usually denoted by $lk(a,K)$. 
\end{Ex}

If $\gamma =\prod_i \gamma_i^{k_i}$ is a finite formal product of closed curves in $Y \setminus K$, then
$\langle \gamma , S \rangle = \sum_i k_i \langle \gamma_i , S \rangle$. 

\begin{Ex} If $(K,S,\phi)$ is an open book decomposition of $Y$ and $\alpha$ is an adapted contact form, then
 $\langle \gamma, S \rangle = \Deg(\gamma)$ for any orbit set $\gamma \in \mathcal{O}(Y \setminus K)$, where $\Deg$ is given in Definition
 \ref{Definition: Degree of multi orbits}.
\end{Ex}

\begin{Prop}[See Proposition 7.1 in \cite{CGHH}] \label{Proposition: Holomorphic curves preserve the filtration in general}
 Suppose that $K$ is an orbit of $R_{\alpha}$ and let $S$ be any Seifert surface for $K$. If
 $\gamma$ and $\delta$ are two orbit sets in $Y \setminus K$ and $u : (F,j) \rightarrow (\R \times Y, J)$ is a holomorphic curve from $\gamma$ to $\delta$, then
 $$\langle \gamma , S \rangle \geq \langle \delta , S \rangle.$$ 
\end{Prop}
\proof
Let $\widehat{u}$ be the compactification of $u$ in $[-1,1] \times Y$. Since $u$ has no limits in $K$, then
\begin{equation} 
 \langle \widehat{u} , [-1,1] \times K  \rangle = \langle u , \mathbb{R} \times K  \rangle \geq 0,
\end{equation}
where the inequality follows by the positivity of intersection in dimension 4 (since $K$ is a Reeb orbit, $\mathbb{R} \times K$ is holomorphic).
Consider the two surfaces
$$L_{-1} = \{-1\} \times S\ \mbox{ and }\ L_{1} = \{1\} \times -S$$
and define the closed surface
\begin{equation} \nonumber
 \begin{array}{ccc}
  L  & = & L_{-1} \cup ([-1,1] \times K) \cup L_{1}
 \end{array}
\end{equation}
where the first gluing is made along ${\{-1\} \times K }$ and the second along ${\{1\} \times -K }$. 
Since $0 = [L] \in H_2([-1,1] \times Y)$, then:
\begin{equation*}
 \begin{array}{cclc}
  0 & = & \langle \widehat{u} , L \rangle & = \\
    & = & \langle \widehat{u} , L_{-1} \rangle + \langle \widehat{u}, [-1,1] \times K \rangle + \langle \widehat{u}, L_{1} \rangle & = \\
    & = & \langle \delta , S \rangle + \langle \widehat{u}, [-1,1] \times K \rangle - \langle \gamma , S \rangle. & 
 \end{array}
\end{equation*}
The result then follows by observing that the last equation implies that
\begin{equation} \label{Equation: Difference of linking numbers equals intersections of K with u}
 \langle \gamma , S \rangle - \langle \delta , S \rangle = \langle \widehat{u}, [-1,1] \times K \rangle \geq 0.
\end{equation}
\endproof

Suppose that $\alpha$ is adapted to $K$ in the sense of Subsection \ref{Subsection: ECH for manifolds with boundary}. A choice of
(a homology class for) the Seifert surface $S$ for the orbit $K$ defines a \emph{knot filtration} on the chain complex
$(ECC^{h_+}(N,\alpha),\partial^{ECH})$ for $\widehat{ECH}(Y,\alpha)$, where, recall, 
$N$ is the complement of a neighborhood $\mathcal{N}(K)$ of $K$ in which the only ``interesting'' orbits and holomorphic curves
are the ones represented in Figure \ref{Figure: Dynamic near binding in W}.

Let $ECC^{h_+}_{d}(N,\alpha)$ be the free sub-module of $ECC^{h_+}(N,\alpha)$ generated
by orbit sets $\gamma$ in $\mathcal{O}(N \sqcup \{h_+\})$ such that $\langle \gamma, S \rangle = d$. Define moreover
$$ECC^{h_+}_{\leq d}(N,\alpha) := \bigoplus_{j \leq d} ECC^{h_+}_{j}(N,\alpha).$$
\begin{Obs}
 The direct sum above is not in general finite. On the other hand if $\alpha$ is adapted to $S$ then $\langle \gamma, S \rangle \geq 0$
 for any $\gamma$ and the sum is finite for any $d$.
\end{Obs}
Even if $\alpha$ is not adapted to $S$, the intersection number induces an exhaustive filtration
$$\ldots \subseteq ECC^{h_+}_{\leq d-1}(N,\alpha) \subseteq ECC^{h_+}_{\leq d}(N,\alpha) \subseteq  ECC^{h_+}_{\leq d+1}(N,\alpha) \subseteq \ldots$$
on $ECC^{h_+}(N,\alpha)$.
\begin{Def}
 The filtration above is the \emph{knot filtration} induced by $K$. If $\gamma$ is a generator of $ECC^{h_+}_{d}(N,\alpha)$, the integer $d$ is
 the \emph{filtration degree} of $\gamma$.
\end{Def}
\begin{Cor} The differential $\partial^{ECH}$ of $ECC^{h_+}(N,\alpha)$ respects the knot filtration.
\end{Cor}
\proof
 Proposition \ref{Proposition: Holomorphic curves preserve the filtration in general} applied to the M-B buildings counted by
 $\partial^{ECH}$ implies immediately that
 $$\partial^{ECH}\left(ECC^{h_+}_{\leq d}(N,\alpha)\right) \subseteq ECC^{h_+}_{\leq d}(N,\alpha)$$
 for any $d$ and the result follows.
\endproof

Suppose now that $\alpha$ is adapted to $S$. By standard arguments in algebra, the filtration above induces a spectral sequence whose
page $\infty$ is isomorphic to
$ECH^{h_+}(N,\alpha) \cong \widehat{ECH}(Y,\alpha)$ and whose page $0$ is the chain complex
\begin{equation} \label{Equation: ECK-hat chain complex}
 \bigoplus_d \left(ECC^{h_+}_{d}(N,\alpha) ,\partial_d^{ECK}\right)
\end{equation}
where $ECC^{h_+}_{d}(N,\alpha)$ should be seen as $\frac{ECC^{h_+}_{\leq d}(N,\alpha)}{ECC^{h_+}_{\leq d-1}(N,\alpha)}$ and
$$\partial_d^{ECK} : ECC^{h_+}_{d}(N,\alpha) \rightarrow ECC^{h_+}_{d}(N,\alpha)$$
is the map induced by $\partial^{ECH}$ on the quotient, i.e, it is the part of $\partial^{ECH}|_{ECC^{h_+}_{d}(N,\alpha)}$ that strictly 
preserves the filtration degree.
\begin{Obs} \label{Observation: The only lowing-degree curve is the disk from h_+ in the hat case}
 The proof of Proposition \ref{Proposition: Holomorphic curves preserve the filtration in general} implies that 
 the holomorphic curves counted by $\partial^{ECH}$ that strictly decrease the degree are exactly the curves that intersect $K$.
 So we can interpret $\partial^{ECK}$ as the restriction of $\partial^{ECH}$ (given by Equation \ref{Equation: The complete ECH differential with all the orbits})
 to the count of curves that do not cross a thin neighborhood of $K$. This is indeed the proper
 $ECH$-differential of the manifold $Y \setminus int(V(K))$ (and not the restriction of the $ECH$-differential of $Y$ to the orbit sets in $Y \setminus int(V(K))$).
 
 Note that, by definition of $ECC^{h_+}(N,\alpha)$, all the holomorphic curves contained in $\R \times N$ strictly preserve the filtration
 degree. In fact the only holomorphic curve that contributes to $\partial^{ECH}|_{ECC^{h_+}(N,\alpha)}$ and decreases the degree (by $1$) is the plane from
 $h_+$ to $\emptyset$. Equation \ref{Equation: The hat ECH differential with all the orbits} gives then
 \begin{equation}\label{Equation: The hat ECK differential with all the orbits}
  \partial(h_+^d \gamma) = h_+^{d-1} e \gamma + h_+^d \partial \gamma.
 \end{equation}
 where $\gamma \in \mathcal{O}(N)$ and any term is meant to be zero if it contains some orbit with total multiplicity that is negative or not
 in $\{0,1\}$ if the orbit is hyperbolic
\end{Obs}

\begin{Def} The \emph{hat version of embedded contact (knot) homology} of the triple $(K,Y,\alpha)$ is
 $$\widehat{ECK}_*(K,Y,\alpha) := H_*\left(ECC^{h_+}(N,\alpha),\partial^{ECK}\right).$$
\end{Def}
\begin{Obs}
 In \cite{CGH2} $\widehat{ECK}(K,Y,\alpha)$ is called $ECH(M(K),\alpha)$ and in Theorem 10.3.2 it is proved that 
 $$\widehat{ECK}(K,Y,\alpha) = ECH^{\sharp}(N,\alpha)$$
 where, recall, with our notation $ECH^{\sharp}(N,\alpha) = ECH^h(int(N),\alpha)$. On the other hand, by using exactly the same arguments of 
 Lemma \ref{Lemma: How kill h+ and get the quotient e gamma sim gamma}, it is easy to see that
 $$ECH^h(int(N),\alpha) \cong H_*\left(ECC^{h_+}(N,\alpha),\partial^{ECK}\right).$$
\end{Obs}

\begin{Obs} \label{Observation: Alpha compatible or not with S to define ECK}
 Note that in order to define $\widehat{ECK}(K,Y,\alpha)$, we supposed that $\alpha$ is compatible with $S$. This hypothesis is not present
 in the original definition (via sutures) in \cite{CGHH}. Indeed, without this condition we can still apply all the arguments above and
 define the knot filtration on $ECC^{h_+}(N,\alpha)$ exactly in the same way. The page $1$ of the spectral sequence is again the well defined
 homology in the definition above, and the page $\infty$ is still isomorphic to $ECH^{h_+}(N,\alpha)$.
 
 The only difference (a priori, see Lemma \ref{Lemma: No need of alpha adapted to S} below) is that now we do not know that
 $ECH^{h_+}(N,\alpha) \cong \widehat{ECH}(Y,\alpha)$, since in Theorem \ref{Theorem: ECH of Y iso to the relative versions} the hypothesis 
 that $\alpha$ is adapted to $S$ is assumed.
\end{Obs}

This homology comes naturally with a further relative degree, inherited by the filtered degree: if
$\widehat{ECK}_{*,d}(K,Y,\alpha) := H_*\left(ECC_d^{h_+}(N,\alpha),\partial_d^{ECK}\right)$ then
$$\widehat{ECK}_{*}(K,Y,\alpha) = \bigoplus_d \widehat{ECK}_{*,d}(K,Y,\alpha).$$
Sometimes, in analogy with Heegaard Floer, we will call this degree the \emph{Alexander degree}.
\begin{Ex} Suppose that $(K,S,\phi)$ is an open book decomposition of $Y$ and that $\alpha$ is an adapted contact form. Since any non-empty
 Reeb orbit in $Y \setminus K$ has strictly positive intersection number with $S$,
 $$\widehat{ECK}_{*,0}(K,Y,\alpha) \cong \langle [\emptyset] \rangle_{\Z/2}.$$
 This is the $ECH$-analogue of the fact that if $K$ is fibered, then 
 $$\widehat{HFK}_{*,-g}(K,Y) \cong \langle[c]\rangle_{\Z/2},$$
 where $g$ is the genus of $K$ and $c$ is the associated contact element (see \cite{OS4}).
\end{Ex}

\begin{Obs}\label{Observation: The relative Alexander degree is well defined}
 We remark that the Alexander degree can be considered as an absolute degree only once a relative homology class in $H_2(Y,K)$ for
 $S$ has been fixed, since the function $\langle \cdot,S \rangle$ defined on $H_1(Y \setminus K)$ changes if $[S]$ varies.
 
 On the other
 hand, suppose that $[\gamma] = [\delta] \in H_1(Y \setminus K)$ and let $F \subset Y$ be a surface such that $\partial F = \gamma - \delta$.
 Computations analogue to that in the proof of Proposition \ref{Proposition: Holomorphic curves preserve the filtration in general} imply that
 \begin{equation} \label{Equation: The relative Alexander degree is well defined}
  \langle \gamma , S \rangle - \langle \delta , S \rangle = \langle F , K \rangle,
 \end{equation}
 and the Alexander degree, considered as a relative degree, does not depend on the choice of a homology class for $S$. 
 
 Obviously if $H_2(Y) = 0$, the Alexander degree can be lifted to an absolute degree.
\end{Obs}

In \cite{CGHH} the authors conjectured that their sutured embedded contact homology is isomorphic to sutured 
Heegaard-Floer homology. Both the hat version of embedded contact knot homology and of Heegaard Floer knot homology can be defined in terms
of sutures. In this case their conjecture becomes
\begin{Cnj} \label{Conjecture: ECK iso to HFK hat}  For any knot $K$ in $Y$:
 $$\widehat{ECK}(K,Y,\alpha) \cong \widehat{HFK}(-K,-Y),$$
 where $\alpha$ is a contact form on $Y$ adapted to $K$.
\end{Cnj}

\section{Generalizations of $\widehat{ECK}$}
\label{Section: Generalizations of ECK}

Let $K$ be a homologically trivial knot in a contact three-manifold $(Y,\alpha)$. As recalled in Subsection
\ref{Subsection: widehat ECH for knots}, if $\alpha$ is adapted to $K$, a choice of a Seifert surface $S$ for $K$ induces a filtration on the
chain complex $\left({ECC}^{h_+}(N,\alpha),\partial^{ECH}\right)$, where $int(N)$ is homeomorphic to $Y \setminus K$. Moreover 
if $\alpha$ is also adapted to $S$, the homology of $\left({ECC}^{h_+}(N,\alpha),\partial^{ECH}\right)$ is isomorphic to
$\widehat{ECH}(Y,\alpha)$, and the first page of the spectral sequence associated to the filtration is the hat version of
embedded contact knot homology $\widehat{ECK}(K,Y,\alpha)$.

In this section we generalise the knot filtration in two natural ways.

In Subsection \ref{Subsection: The full ECK} we extend the filtration induced by $K$ on the chain
complex $\left({ECC}^{h_+,e_+}(N,\alpha),\partial^{ECH}\right)$.
This filtration is defined in a way completely analogue to the hat case. We define the \emph{full version of embedded contact knot homology}
of $(K,Y,\alpha)$ to be the first page $ECK(K,Y,\alpha)$ of the associated spectral sequence. Moreover we remove the condition that
$\alpha$ must be compatible with $S$, in order to consider a wider class of contact forms: the knot spectral sequence is still well defined, but at the 
price of renouncing to a proof of the existence of an isomorphism between $ECH(Y,\alpha)$ and the page $\infty$ of the spectral sequence.

In Subsection \ref{Subsection: ECK for links} we generalise the knot filtration to $n$-components links $L$.
The resulting homologies, defined in a way analogue to the case of knots, are the \emph{full and hat versions of embedded contact knot
homologies of $(L,Y,\alpha)$}, which will be still denoted $ECK(L,Y,\alpha)$ and, respectively, $\widehat{ECK}(L,Y,\alpha)$.
Similarly to Heegaard-Floer link homology, these homologies come endowed with an \emph{Alexander (relative) $\Z^n$-degree}.

\subsection{The full $ECK$} \label{Subsection: The full ECK}

Let $K$ be a homologically trivial knot in a contact three-manifold $(Y,\alpha)$ and suppose that $\alpha$ is adapted to $K$ in the sense of
Subsection \ref{Subsection: ECH for manifolds with boundary}. Recall in particular that there
exist two concentric neighborhoods $V(K) \subset \mathcal{N}(K)$ of $K$ whose boundaries are M-B tori $T_1 = \partial \mathcal{N}(K)$ and
$T_2 = \partial V(K)$ foliated by orbits of $R_{\alpha}$ in the homology
class of meridians for $K$. These two families of orbits are modified into the two couples of orbits $\{e,h\}$ and, respectively,
$\{e_+,h_+\}$. Let moreover $N = Y \setminus int(\mathcal{N}(K))$.

Consider the chain complex $\left(ECC^{e_+,h_+}(N,\alpha), \partial^{ECH}\right)$ where, recall, the chain group is freely generated on
$\Z/2$ by the orbit sets $\gamma$ in $\mathcal{O}(N) \sqcup \{h_+,e_+\}$ and $\partial^{ECH}$ is the $ECH$-differential (obtained by restricting
the differential on $ECC(Y,\alpha)$) given by  
Equation \ref{Equation: The complete ECH differential with all the orbits}. 

A Seifert surface $S$ for $K$ induces an \emph{Alexander degree} $\langle \cdot,S \rangle$ on the generators of ${ECC}^{h_+,e_+}(N,\alpha)$
exactly like in the case of ${ECC}^{h_+}(N,\alpha)$.
Let ${ECC}_{d}^{h_+,e_+}(N,\alpha)$ be the submodule of ${ECC}^{h_+,e_+}(N,\alpha)$ generated by the $\gamma \in \mathcal{O}(N) \sqcup \{h_+,e_+\}$
with $\langle \gamma,S \rangle = d$. If
$${ECC}_{\leq d}^{h_+,e_+}(N,\alpha) := \bigoplus_{j \leq d} {ECC}_{j}^{h_+,e_+}(N,\alpha),$$
we have the exhaustive filtration
$$\ldots \subseteq {ECC}_{\leq d-1}^{h_+,e_+}(N,\alpha) \subseteq {ECC}_{\leq d}^{h_+,e_+}(N,\alpha) \subseteq {ECC}_{\leq d+1}^{h_+,e_+}(N,\alpha) \subseteq \ldots $$
of ${ECC}^{h_+,e_+}(N,\alpha)$. Proposition \ref{Proposition: Holomorphic curves preserve the filtration in general} again implies that
$\partial^{ECH}$ preserves the filtration. Let
$$\partial_d^{ECK} : {ECC}_{d}^{h_+,e_+}(N,\alpha) \longrightarrow {ECC}_{d}^{h_+,e_+}(N,\alpha) $$
be the part of $\partial^{ECH}$ that strictly preserves the filtration degree $d$, that is, the differential
induced by $\partial^{ECH}|_{{ECC}_{\leq d}^{h_+,e_+}(N,\alpha)}$ on the quotient
$$\frac{{ECC}_{\leq d}^{h_+,e_+}(N,\alpha)}{{ECC}_{\leq d-1}^{h_+,e_+}(N,\alpha)} = {ECC}_{d}^{h_+,e_+}(N,\alpha).$$
Set
$$\partial^{ECK}:=\bigoplus_d \partial_d^{ECK} \colon ECC^{e_+,h_+}(N,\alpha) \longrightarrow ECC^{e_+,h_+}(N,\alpha).$$
\begin{Def}
We define the \emph{full embedded contact knot homology of $(K,Y,\alpha)$} by
 $$ECK(K,Y,\alpha) := H_*\left(ECC^{e_+,h_+}(N,\alpha), \partial^{ECK}\right). $$
\end{Def}

Note that, as in the hat case, the only holomorphic curves counted by $\partial^{ECH}$ that do not strictly respect the filtration degree
are the curves that
contain the plain from $h_+$ to $\emptyset$ (see Observation \ref{Observation: The only lowing-degree curve is the disk from h_+ in the hat case}).
Recalling the expression of $\partial^{ECH}$ given in Equation
\ref{Equation: The complete ECH differential with all the orbits}, it follows that $\partial^{ECK}$ is given by  
\begin{equation} \label{Equation: The complete ECK differential with all the orbits}
 \partial^{ECK}(e_+^a h_+^b \gamma) =  e_+^{a-1} h_+^b h \gamma +  e_+^a h_+^{b-1} e\gamma + e_+^a h_+^b \partial \gamma,
\end{equation}
where $\gamma \in \mathcal{O}(N)$ and any term is meant to be $0$ if it contains an orbit with total multiplicity that is negative or not in
$\{0,1\}$ if the orbit is hyperbolic.

Again the homology comes with an \emph{Alexander degree}, which is well defined once the an homology class for $S$ is fixed.
In fact we have the natural splitting:
\begin{equation} \label{Equation: ECK for knots splits in direct sum}
 ECK_*(K,Y,\alpha) \cong  \bigoplus_{d \in \Z} ECK_{*,d}(K,Y,\alpha)
\end{equation}
where
$$ECK_{*,d}(K,Y,\alpha) := H_*(ECC_d^{h_+,e_+}(N,\alpha),\partial_d^{ECK}).$$

Recalling that $Y \setminus \mathcal{N}(K)$ is homeomorphic to $Y \setminus K$, it is interesting to state the following:
\begin{Lemma} \label{Lemma: ECK for knots full iso to ECH Y minus N(K)}
 If $\mathcal{N}(K)$ is a neighborhood of $K$ as above then
 $$ECK(K,Y,\alpha) \cong ECH(Y \setminus \mathcal{N}(K),\alpha).$$
\end{Lemma}
\proof
 By arguments similar to those in the proof of Lemma \ref{Lemma: How kill h+ and get the quotient e gamma sim gamma} it is easy to prove that:
 \begin{equation*}
  \begin{array}{ccl}
   ECK(K,Y,\alpha) & \cong & H_*\left(ECC^{e_+,h_+}(N,\alpha), \partial^{ECK}\right)\\
                   & \cong & H_*\left(ECC^{e,h_+}(int(N),\alpha), \partial^{ECK}\right)\\
                   & \cong & H_*\left(ECC(int(N),\alpha), \partial^{ECK}\right)\\
                   & \cong & ECH(int(N),\alpha),
  \end{array}
 \end{equation*}
 where the last comes from the fact that $\partial^{ECK}(\gamma) = \partial^{ECH}(\gamma)$ for any $\gamma \in \mathcal{O}(N)$.
\endproof

 

 Note that so far we only assumed that $\alpha$ is compatible with $K$, while we did not suppose the condition
 \begin{itemize}
  \item[($\spadesuit$)] $\alpha$ is compatible with a Seifert surface $S$ for $K$.
 \end{itemize}
 As remarked in Observation \ref{Observation: Alpha compatible or not with S to define ECK}, without $\spadesuit$ we can not apply Theorem
 \ref{Theorem: ECH of Y iso to the relative versions}, and so we do not know if the spectral sequence whose $0$-page is the $ECK$-chain complex
 limits to $ECH(Y,\alpha)$. On the other hand we have the following
 
 \begin{Lemma} \label{Lemma: No need of alpha adapted to S}
  Theorem \ref{Theorem: ECH of Y iso to the relative versions} holds even without assuming condition $\spadesuit$.
 \end{Lemma}
 \proof
  Reading carefully the proof of Theorem \ref{Theorem: ECH of Y iso to the relative versions} given in \cite{CGH2} one can see that $\spadesuit$ 
  is not really necessary. It is explicitly used only in Section 9.7 to prove that the map $\sigma_k$ is nilpotent, but this fact can be proved
  also without assuming $\spadesuit$. Indeed $\partial'_N$ strictly decreases the Alexander degree. Then, if $(\partial'_N)^i \neq 0$
  for every $i \in \N$, for $j$ arbitrarily large $(\partial'_N)^j \Gamma$ would contain (as factor) an orbit set with arbitrarily large negative
  Alexander degree, and so also with arbitrarily large action, which is not possible by Lemma \ref{Lemma: Holomorphic curves low the actions of the orbits}.
 \endproof

%
 
 
 \begin{Obs}
 A rough explanation of last lemma is the following. 
 By direct limit arguments the orbits in the no man's land
 $int(\mathcal{N}(K)) \setminus V(K)$ can be avoided also if $\spadesuit$ is not assumed, so that we can still write
 $$ECC(Y,\alpha) \cong ECC(V,\alpha) \otimes ECC(N,\alpha)$$
 (up to the restriction on the action of the orbits made in \cite[Section 9]{CGH2}).
 The computations for $ECH(V,\alpha)$ in \cite[Section 8]{CGH2} do not use $\spadesuit$, and in fact here the hypothesis
 is not even assumed. Similarly, $ECH(N,\alpha)$ is still well defined as in \cite[Subsection 7.1]{CGH2} and does not depend on the choice
 of $S$. Moreover the Blocking Lemma still implies that holomorphic curves with positive limit in $N$ can not
 cross $\partial N$, so that $ECC(N,\alpha)$ is again a subcomplex of $ECC(Y,\alpha)$. This suggests that what happens in
 $N$ should not influence the direct limits computations in $V$.

\end{Obs}

\vspace{0.3 cm}



In analogy with Conjecture \ref{Conjecture: ECK iso to HFK hat} we state the following:
\begin{Cnj} \label{Conjecture: ECK iso to HFK full} For any knot $K$ in $Y$:
 $$ECK(K,Y,\alpha) \cong HFK^+(-K,-Y),$$
 where $\alpha$ is any contact form on $Y$ adapted to $K$.
\end{Cnj}

\subsection{The generalization to links}
\label{Subsection: ECK for links}


In this subsection we extend the definitions of $ECK$ and $\widehat{ECK}$ to the case of homologically trivial links with more than one component.
For us a (\emph{strongly}) \emph{homologically trivial $n$-link} in $Y$ is a disjoint union of $n$ knots, each of which is homologically trivial in $Y$. 


Suppose that 
$$L = K_1 \sqcup \ldots \sqcup K_n$$
is a homologically trivial $n$-link in $Y$.
We say that a contact form $\alpha$ on $Y$ is \emph{adapted to $L$} if it is adapted to $K_i$ for each $i$.
\begin{Lemma} For any link $L$ and contact structure $\xi$ on $Y$ there exists a contact form compatible with $\xi$ which is adapted to $L$. 
\end{Lemma}
\proof
 The proof of part 1) of Lemma \ref{Lemma: There exists alpha copatible with K and S} is local near the knot $K$ and can then be applied
 recursively to each $K_i$.
\endproof

Fix $L = K_1 \sqcup \ldots \sqcup K_n$ homologically trivial and $\alpha$ an adapted contact form. Since $\alpha$
is adapted to each $K_i$, there exist pairwise disjoint tubular neighborhoods 
$$V(K_i) \subset \mathcal{N}(K_i)$$
of $K_i$ where $\alpha$ behaves exactly like in 
the neighborhoods $V(K) \subset \mathcal{N}(K)$ in Subsection \ref{Subsection: ECH for manifolds with boundary}. 

In particular, for each $i$, the tori $T_{i,1} := \partial \mathcal{N}(K_i)$ and $T_{i,2} := \partial V(K_i)$ are
M-B and foliated by families of orbits of $R_{\alpha}$ in the homology class of a meridian of $K_i$. We will consider these two families
as perturbed into two pairs $\{e_{i},h_{i}\}$ and $\{e^+_{i},h^+_{i}\}$ in the usual way.

Let
$$V(L) := \bigsqcup_i V(K_i)\; \mbox{ and }\; \mathcal{N}(L) := \bigsqcup_i \mathcal{N}(K_i)$$
and set
$$N := Y \setminus int(\mathcal{N}(L)).$$
Define moreover $\bar{e} := \bigsqcup_i e_i$ and let $\bar{h},\ \bar{e}_+$ and $\bar{h}_+$ be similarly defined.

\vspace{0.3 cm}

Consider now $ECC^{\bar{e}_+,\bar{h}_+}\left(N,\alpha \right)$ endowed with the restriction
$\partial^{ECH}$ of the $ECH$ differential of $(Y,\alpha)$ and let $ECH^{\bar{e}_+,\bar{h}_+}\left(N,\alpha \right)$ be the associated
homology.

\begin{Lemma} $ECH^{\bar{e}_+,\bar{h}_+}\left(N,\alpha \right)$ is well defined and the curves counted by $\partial^{ECH}$ inside each
 $\mathcal{N}(K_i)$ are given by expressions analogue to those in \ref{Equation: ECH boundary near the binding}.
\end{Lemma}
\proof
 The Blocking and Trapping lemmas can be applied locally near each component of $\partial N$ and the proofs of lemmas 7.1.1 and 7.1.2 in \cite{CGH2} work
 immediately in this context too. This imply that the homology of $\left(ECC(N,\alpha), \partial^{ECH} \right)$ is well defined.
 
 Again the Blocking and Trapping lemmas together with the local homological arguments in lemmas 9.5.1 and 9.5.3 in \cite{CGH2}, imply that 
 the only holomorphic curves counted by $\partial^{ECH}$ inside each $\mathcal{N}(K_i)$ are as required (see Figure \ref{Figure: Dynamic near binding in W}), and so that 
 $ECH^{\bar{e}_+,\bar{h}_+}\left(N,\alpha \right)$ is well defined.
\endproof

An explicit formula for $\partial^{ECH}$ can be obtained by generalizing Equation \ref{Equation: The complete ECH differential with all the orbits}
in the obvious way.


For each $i \in \{1,\ldots,n\}$, fix now a (homology class for a) Seifert surface $S_i$ for $K_i$.
These surfaces are not necessarily pairwise disjoint and it is even possible that $S_i \cap K_j \neq \emptyset$ for some $i \neq j$.

Consider then the \emph{Alexander $\Z^n$-degree} on $ECC^{\bar{e}_+,\bar{h}_+}\left(N,\alpha \right)$ given by the function
\begin{equation} \label{Equation: Multiple Alexander filtration degree}
 \begin{array}{ccc}
  ECC^{\bar{e}_+,\bar{h}_+}\left(N,\alpha \right) & \longrightarrow & \Z^n\\
          \gamma                          & \longmapsto     &  (\langle \gamma ,S_1 \rangle,\ldots,\langle \gamma ,S_n \rangle).
 \end{array}
\end{equation}

Define the partial ordering on $\Z^n$ given by
$$(a_1,\ldots,a_n) \leq (b_1,\ldots,b_n) \Longleftrightarrow a_i \leq b_i\ \forall\ i.$$ 
Proposition \ref{Proposition: Holomorphic curves preserve the filtration in general} applied to each $K_i$ implies that if
$\gamma$ and $\delta$ are two orbit sets in $\mathcal{O}(N \sqcup \{\bar{e}_+,\bar{h}_+\})$, then for any $k$
$$\frac{\mathcal{M}_k(\gamma, \delta)}{\R} \neq 0\ \Longrightarrow \left(\langle \delta ,S_1 \rangle,\ldots,\langle \delta ,S_n \rangle\right) \leq \left(\langle \gamma ,S_1 \rangle,\ldots,\langle \gamma ,S_n \rangle\right).$$
This implies that $\partial^{ECH}$ does not increase the Alexander degree, which induces than a $\Z^n$-filtration on
$\left(ECC^{\bar{e}_+,\bar{h}_+}(N,\alpha),\partial^{ECH}\right)$.

Reasoning as in the previous subsection, we are interested in the part of $\partial^{ECH}$ that strictly respects the
filtration degree. This can be defined again in terms of quotients as follows. 

Let $d \in \Z^n$ and let $ECC_d^{\bar{e}_+,\bar{h}_+}(N,\alpha)$
be the submodule of $ECC^{\bar{e}_+,\bar{h}_+}(N,\alpha)$ freely generated by orbit sets $\gamma \in \mathcal{O}(N \sqcup \{\bar{e}_+,\bar{h}_+\})$
such that
$$(\langle \gamma ,S_1 \rangle,\ldots,\langle \gamma ,S_n \rangle) = d.$$
Define
$$ECC_{\leq d}^{\bar{e}_+,\bar{h}_+}(N,\alpha) := \bigoplus_{j \leq d}ECC_d^{\bar{e}_+,\bar{h}_+}(N,\alpha)$$
and let $ECC_{< d}^{\bar{e}_+,\bar{h}_+}(N,\alpha)$ be similarly defined.

Define the \emph{full $ECK$-differential in degree $d$} to be the map
$$\partial^{ECK}_d \colon ECC_d^{\bar{e}_+,\bar{h}_+}(N,\alpha) \longrightarrow ECC_d^{\bar{e}_+,\bar{h}_+}(N,\alpha)$$
induced by $\partial^{ECH}|_{{ECC}_{\leq d}^{\bar{e}_+,\bar{h}_+}(N,\alpha)}$ on the quotient
$$\frac{{ECC}_{\leq d}^{\bar{e}_+,\bar{h}_+}(N,\alpha)}{{ECC}_{< d}^{\bar{e}_+,\bar{h}_+}(N,\alpha)} \cong {ECC}_{d}^{\bar{e}_+,\bar{h}_+}(N,\alpha).$$
Define then the \emph{full $ECK$-differential} by
$$\partial^{ECK}:=\bigoplus_d \partial_d^{ECK} \colon ECC^{\bar{e}_+,\bar{h}_+}(N,\alpha) \longrightarrow ECC^{\bar{e}_+,\bar{h}_+}(N,\alpha).$$

\begin{Obs} \label{Observation: The only lowing-degree curves are the disks from h_+ for links}
 Observing the form of $\partial^{ECH}$, it is easy again to see that the only holomorphic curves that are counted by $\partial^{ECH}$ and
 not by $\partial^{ECK}$ are the ones containing a holomorphic plane from some $h_i^+$ to $\emptyset$.
\end{Obs}

\begin{Def}
 The \emph{full embedded contact knot homology of $(L,Y,\alpha)$} is
 $$ECK(L,Y,\alpha) := H_*\left(ECC^{\bar{e}_+,\bar{h}_+}(N,\alpha), \partial^{ECK}\right). $$
\end{Def}
The fact that $ECK(L,Y,\alpha)$ is well defined is a direct consequence of the good definition of $ECH^{\bar{e}_+,\bar{h}_+}(N,\alpha)$
and the fact that $\partial^{ECH}$ respects the Alexander filtration.

Note that also for links we have a natural splitting
\begin{equation} \label{Equation: Split of ECK of L full}
 ECK_*(L,Y,\alpha) = \bigoplus_{d \in \Z^n} ECK_{*,d}(L,Y,\alpha)
\end{equation}
where
$$ECK_{*,d}(L,Y,\alpha) = H_*\left(ECC_d^{\bar{e}_+,\bar{h}_+}(N,\alpha), \partial_d^{ECK}\right). $$

The proof of the following lemma is the same of that of the analogous Lemma \ref{Lemma: ECK for knots full iso to ECH Y minus N(K)} for knots
applied to each component of $L$.
\begin{Lemma} \label{Lemma: ECK iso to ECH of int N}
 If $\mathcal{N}(L)$ is a neighborhood of $L$ as above then
 $$ECK(L,Y,\alpha) \cong ECH(Y \setminus \mathcal{N}(L),\alpha).$$
\end{Lemma}

\vspace{0.3 cm}

Consider now the submodule $ECC^{\bar{h}_+}(N,\alpha)$ of $ECC^{\bar{e}_+,\bar{h}_+}(N,\alpha)$ endowed with the restriction of
$\partial^{ECH}$. Again its homology $ECH^{\bar{h}_+}(N,\alpha)$ is well defined.

Proceeding exactly like above, the choice of a Seifert surface 
$S_i$ for each component $K_i$ of $L$ gives (up to small perturbations of $S$) an Alexander degree on the orbit sets defined by Equation
\ref{Equation: Multiple Alexander filtration degree}. This induces a $\Z^n$-filtration on the chain complex
$\left(ECC^{\bar{h}_+}(N,\alpha),\partial^{ECH}\right)$.

For any $d \in \Z^n$, define $ECC_d^{\bar{h}_+}(N,\alpha)$ and
$$\partial_d^{ECK} \colon ECC_d^{\bar{h}_+}(N,\alpha) \longrightarrow ECC_d^{\bar{h}_+}(N,\alpha)$$
exactly as above.
\begin{Def}
  The \emph{hat version of embedded contact knot homology of $(L,Y,\alpha)$} is
 $$\widehat{ECK}(L,Y,\alpha) := H_*\left(ECC^{\bar{h}_+}(N,\alpha), \partial^{ECK}\right). $$
\end{Def}

Observation \ref{Observation: The only lowing-degree curves are the disks from h_+ for links} and a splitting like the one in equation
\ref{Equation: Split of ECK of L full} hold also for $\widehat{ECK}(L,Y,\alpha)$. Moreover it is easy to see that if $L$ has only one connected
component we get the same theories of subsections \ref{Subsection: widehat ECH for knots} and \ref{Subsection: The full ECK}.

\vspace{0.3 cm}

We state the following
\begin{Cnj} \label{Conjecture: ECK iso to HFK for links full and hat}
 If $L$ is a link in $Y$:
 \begin{eqnarray*}
  ECK(L,Y,\alpha) & \cong & HFK^-(L,Y),\\
  \widehat{ECK}(L,Y,\alpha) & \cong & \widehat{HFK}(L,Y),
 \end{eqnarray*}
 where $\alpha$ is any contact form on $Y$ adapted to $L$.
\end{Cnj}

\begin{Obs}
 Note that the analogous conjectures stated before, as well as Theorem \ref{Theorem: HF and ECH in intro}, suggest that we should use the plus
 version of $HFL$ and not the minus one. The problem is that in \cite{OS5} the authors define Heegaard-Floer homology for links only
 in the hat and minus versions.
 
 On the other hand this switch is not really significant. Indeed one could define 
 \emph{Heegaard-Floer cohomology} groups by taking the duals, with coefficients in $\Z/2$, of the chain groups
 $\widehat{CF}_*(Y)$, $CF^+_*(Y)$ and $CF^-_*(Y)$ in the usual way
 and get cohomology groups (for the three-manifold $Y$)
 \begin{equation*}
  \begin{array}{lcr}
   \widehat{HF}^*(Y),\ & HF_+^*(Y),\ & HF_-^*(Y). 
  \end{array}
 \end{equation*}
 Since we are working in $\Z/2$ we have that each of this cohomology group is isomorphic to its respective homology group.
 
 On the other hand one can prove also that (see Proposition 2.5 in \cite{OS2}):
 \begin{eqnarray*}
  \widehat{HF}^*(Y) \cong \widehat{HF}_*(Y) & \mbox{and} & HF_{\pm}^*(Y) \cong HF^{\mp}_*(Y).
 \end{eqnarray*}
 Analogous formulae hold also for knots. The conjecture above is then consistent with those stated in the previous subsections.
\end{Obs}

\begin{Obs}
 As in the definition of $\widehat{ECK}(K,Y,\alpha)$ and ${ECK}(K,Y,\alpha)$ also here we used the hypothesis that $\alpha$ is adapted to $L$,
 while, in view of Lemma \ref{Lemma: No need of alpha adapted to S}, we dropped condition $\spadesuit$ of last subsection.
 One could wonder if it is possible to further relax the assumptions and get still a good definition of
 the $ECK$ homology groups.
 
 The conjectures above suggest indeed that
 $ECK(L,Y,\alpha)$ (as well as the other homologies) would be independent from $\alpha$ and so, in particular, that we could be able
 to define it simply as the $ECH$ homology
 of the complement of (any neighborhood of) $L$, provided that $L$ is a disjoint union of Reeb orbits of $\alpha$. Indeed,
 even if we could not have an easy description of the curves counted by $\partial^{ECH}$ that cross $L$, Proposition 
 \ref{Proposition: Holomorphic curves preserve the filtration in general} still holds in this more general case.
 
 
 On the other hand, technical aspects about contact flows and holomorphic curves suggest that the components of $L$ should be at least
 elliptic orbits. This property will be necessary even in computing Euler characteristics in next section, where we will need a circularity
 property of $R_{\alpha}$ near $L$ that cannot be assumed in an evident way if a component of $L$ is hyperbolic.
\end{Obs}
 

\begin{Notations}
 In order to simplify the notation, in the rest of the paper we will indicate the $ECH$ chain groups for
 the knot embedded contact homology groups of links and knots by:
 \begin{eqnarray*}
  ECC(L,Y,\alpha) & := & ECC^{\bar{e}_+,\bar{h}_+}(N,\alpha),\\
  \widehat{ECC}(L,Y,\alpha) & := &  ECC^{\bar{h}_+}(N,\alpha),
 \end{eqnarray*}
 where $N$ and $\alpha$ are as above. In particular, if not stated otherwise, we will always assume
 that the contact form $\alpha$ is adapted to $L$. These groups will implicitly come endowed with the differential $\partial^{ECK}$.
\end{Notations}

\vspace{0.3 cm}

We end this section by saying some word about a further generalization of $ECK$ to \emph{weakly homologically trivial} links. We say that
$L \subset Y$ is a weakly homologically trivial (or simply \emph{weakly trivial}) $n$-component link if there exist surfaces
with boundary $S_1,\ldots, S_m \subset Y$, with
$m \leq n$ and such that $\partial S_i \cap \partial S_j = \emptyset$ if $i \neq j$ and $ \bigsqcup_{i=1}^m \partial S_i = L$.
Also here we do not require that $S_i$ or even $\partial S_i$ is disjoint from $S_j$ for $j \neq i$.

Clearly $L$ is a strongly trivial link if and only if it is weakly trivial with $m = n$.

In this case we cannot in general define a homology with a filtered $n$-degree. If $L$ is a weakly trivial link
with $m \lneq n$ and $\alpha$ is an adapted contact form, then there exists $S \in \{S_1,\ldots,S_m\}$ such that $\partial S$ has more then 
one connected component. Suppose for instance that $\partial S = K_1 \sqcup K_2$. The arguments of proposition
\ref{Proposition: Holomorphic curves preserve the filtration in general} say then that if $u : (F,j) \rightarrow (\R \times Y, J)$ is a
holomorphic curve from $\gamma$ to $\delta$, then
$$\langle \gamma , S \rangle - \langle \delta , S \rangle = \langle \Imm(u), \R \times (K_1 \sqcup K_2) \rangle \geq 0.$$ 
So in this case we can still apply the arguments above and get well defined $ECH$ invariants for $L$. However this time they will come
only with a filtered (relative) $\Z^m$-degree on the generators $\gamma$ of an $ECH$ complex of $Y$, which is given by the
$m$-tuple $(\langle \gamma, S_1 \rangle,\ldots,\langle \gamma, S_m \rangle)$.

\begin{Ex} \label{Example: ECK for open books}
 Let $(L,S,\phi)$ be an open book decomposition of $Y$ with, possibly, disconnected boundary. Using a (connected) page of $(L,S,\phi)$ to compute
 the Alexander degree and, with the notations of Subsection \ref{Subsubsection: ECH from open books}, we get
 $$ECK_d(L,Y,\alpha) \cong ECH_d(int(N),\alpha)$$
 for any $d \in \Z$. 
\end{Ex}

\section{Euler characteristics} \label{Section: Euler characteristic in S^3}

In this section we compute the graded Euler characteristics of the embedded contact homology groups for knots and links in homology three
spheres $Y$ with respect to suitable contact forms. The computations will be done in terms of the Lefschetz zeta function of the flow
of the Reeb vector field.

Before proceeding we briefly recall what the graded Euler characteristic is. Given a collection of chain complexes 
$$(C,\partial) = \{(C_{*,(i_1,\ldots,i_n)},\partial_{(i_1,\ldots,i_n)})\}_{(i_1,\ldots,i_n) \in \mathbb{Z}^n},$$ where
$*$ denotes a relative homological degree, its \emph{graded Euler characteristic} is 
$$\chi(C) = \sum_{i_1,\ldots,i_n} \chi \left(C_{*,(i_1,\ldots,i_n)}\right) t_1^{i_1} \cdots t_n^{i_n} \in \Z[t_1^{\pm 1},\ldots,t_n^{\pm 1}]$$
where $\chi \left(C_{*,(i_1,\ldots,i_n)}\right)$ is the standard Euler characteristic of $C_{*,(i_1,\ldots,i_n)}$ and the $t_j$'s are formal
variables. By definition, $\chi(C)$ is a Laurent polynomial and the properties of the standard Euler characteristic imply
$$\chi(C) = \chi\left(H(C,\partial)\right).$$
In this case the homology $H(C,\partial)$ is a \textit{categorification} of the polynomial $\chi(C)$.

When we want to highlight the variables of these polynomials we will indicate them
as subscripts of the symbol $\chi$. For example if $L$ is an $n$-link and we want to express its Euler characteristic by a polynomial in the
$n$ variables $t_1,\ldots,t_n$, we will write $\chi\left(ECK(L,Y,\alpha)\right) = \chi_{t_1,\ldots,t_n}\left(ECK(L,Y,\alpha)\right)$.

\vspace{0.3 cm}

The most important result of this section relates the Euler characteristic of $ECK$ homologies of a link in $S^3$ with the multivariable
Alexander polynomial $\Delta_L$.
\begin{Thm} \label{Theorem: Euler characteristic of ECK}
 Let $L$ be any $n$-link in $S^3$. Then there exists a contact form $\alpha$ adapted to $L$ such that:
 \begin{eqnarray} \label{Equation: Euler characteristics of ECK full}
  \chi\left(ECK(L,S^3,\alpha)\right) \doteq \left\{ \begin{array}{cc}
                                                \Delta_L(t_1,\ldots,t_n) & \mbox{if } n > 1\\
                                                 & \\
                                                \Delta_L(t)/({1-t})  & \mbox{if } n = 1
                                               \end{array}\right.
 \end{eqnarray}
 and 
 \begin{eqnarray} \label{Equation: Euler characteristics of ECK hat}
  \chi\left(\widehat{ECK}(L,S^3,\alpha)\right) \doteq \left\{ \begin{array}{cc}
                                                \Delta_L \cdot \prod_{i=1}^n(1 - t_i) & \mbox{if } n > 1\\
                                                 & \\
                                                \Delta_L(t)  & \mbox{if } n = 1.
                                               \end{array}\right. 
 \end{eqnarray}
\end{Thm}

Last theorem imply that \emph{the homology $ECK$ categorifies the Alexander polynomial of knots and links
in $S^3$}. This is the third known categorification
of this kind, after the ones in Heegaard-Floer homology and in Seiberg-Witten-Floer homology (see \cite{KM1} and \cite{KM2}).

\vspace{0.3 cm}

An immediate consequence of Theorem \ref{Theorem: Euler characteristic of ECK} and Equations \ref{EquationIntro: Euler characteristics of HFL^-}
and \ref{EquationIntro: Euler characteristics of widehat HFL} is:
\begin{Cor}
 For any link $L$ in $S^3$ there exists a contact form $\alpha$ such that:
 \begin{eqnarray} 
  \chi({ECK}(L,S^3,\alpha)) &\doteq & \chi(HFL^-(L,S^3)), \nonumber\\
  \chi(\widehat{ECK}(L,S^3,\alpha)) &\doteq &\chi(\widehat{HFL}(L,S^3)). \nonumber
 \end{eqnarray}
\end{Cor}

The last corollary implies that conjecture \ref{Conjecture: ECK iso to HFK for links full and hat} (which generalizes conjectures
\ref{Conjecture: ECK iso to HFK hat} and \ref{Conjecture: ECK iso to HFK full}) holds for links in $S^3$ at least at the level of Euler
characteristic.

\vspace{0.3 cm}


A key ingredient to prove Theorem \ref{Theorem: Euler characteristic of ECK} is the dynamical formulation of the Alexander quotient given by
Fried in \cite{Fri}.

\subsection{A dynamical formulation of the Alexander polynomial}
\label{Subsection: The dynamical definition of Delta}

Given any link $L=K_1 \sqcup \ldots \sqcup K_n$ in $S^3$ we can associate to it its \emph{multivariable Alexander polynomial}
$$\Delta_L(t_1,\ldots,t_n) \in \frac{\Z[t_1^{\pm 1},\ldots,t_n^{\pm 1}]}{\pm t_1^{a_1}\cdots t_n^{a_n}}.$$
with $a_i \in \Z$. The quotient means that the Alexander polynomial is well defined only up to multiplication by monomials of the form 
$\pm t_1^{a_1}\cdots t_n^{a_n}$.

A slightly simplified version is the \textit{(classical) Alexander polynomial} $\Delta_L(t)$, defined by setting
$t_1 = \ldots = t_n = t$, i.e.:
$$\Delta_L(t) := \Delta_L(t,\ldots,t).$$
If $L$ is a knot the two notions obviously coincide.

There are many possible definitions of the Alexander polynomial $\Delta_L$. In this section we give a formulation of
$\Delta_L$ in terms of the dynamics of suitable vector fields in $S^3 \setminus L$. The details about the proof of the statements can be found
in the references.

The fact that the Alexander polynomial
is related to dynamical properties of its complement in $S^3$ origins with the study of fibrations of $S^3$. For example in \cite{Ac} A'Campo
studied the \textit{twisted Lefschetz zeta function} of the monodromy of an open book decomposition $(S,\phi)$ of $S^3$ associated to a Milnor
fibration of a complex algebraic singularity.
More in general, if $(K,S,\phi)$ is any open book decomposition of $S^3$, one can easily prove (see for example \cite{Ro}) that
\begin{equation*} 
 \Delta_K(t) \doteq \mathrm{det}(\mathbbm{1} - t \phi_*^1),
\end{equation*}
where $\mathbbm{1}$ and $\phi_*^1$ are the identity map and, respectively, the application induced by $\phi$, on $H_1(S,\Z)$. The basic idea in this
context is to express the right-hand side of equation above in terms of traces of iterations of $\phi_*^1$; then to apply the Lefschetz
fixed point theorem to get expressions in terms of periodic points, (i.e. periodic orbits) for the flow of some vector field in $S^3 \setminus K$
whose first return on a page is $\phi$.


Suppose now that $L$ is not a fibered link, so that its complement is not globally fibered over $S^1$ and let $R$ be a vector field in
$S^3 \setminus L$.
If one wants to apply arguments like above, it is necessary to decompose $S^3 \setminus L$ in ``fibered-like'' pieces with respect to $R$, 
in which it is possible to define at least a local first return map of the flow $\phi_R$ of $R$. Obviously some condition on $R$ is required.
For example, in his beautiful paper \cite{Fra}, Franks consider \textit{Smale vector fields}, that is, vector fields whose chain recurrent set is
one-dimensional and hyperbolic (cf. \cite{Sm}). 

Here we are more interested in the approach used by Fried in \cite{Fri}. Consider a three-dimensional manifold $X$.
Any \emph{abelian cover}
$\widetilde{X} \stackrel{\pi}{\rightarrow} X$ with deck transformations group isomorphic to a fixed abelian group $G$ is uniquely determined
by the choice of a class $\rho = \rho(\pi) \in H^1(X,G) \cong \mathrm{Hom}\left( H_1(X,\Z), G \right)$. Here $\rho$ is determined by the 
following property: for any $[\gamma] \in H_1(X)$, if $\widetilde{\gamma} : [0,1] \rightarrow \widetilde{X}$ is any lifting of the loop
$\gamma : [0,1] \rightarrow X$, then $\rho([\gamma])$ is determined by $\rho([\gamma])(\widetilde{\gamma}(0)) = \widetilde{\gamma}(1)$.

Since the correspondence between Abelian covers and cohomology classes is bijective, with abuse of notation sometimes we will refer to an
abelian cover directly by identifying it with the corresponding $\rho$.

\begin{Ex}
 The \emph{universal abelian cover} of $X$ is the abelian cover with deck transformation group $G = H_1(X,\Z)$ and corresponding to $\rho = id$.
 
\end{Ex}
\begin{Ex} \label{Example: Abelian cover determined by a link}
 Let $L= K_1 \sqcup \ldots \sqcup K_n$ be an $n$-components link in a three manifold $Y$ such that $K_i$ is homologically trivial for any $i$
 and fix a Seifert surface $S_i$ for $K_i$. Let moreover $\mu_i$ be a positive meridian for $K_i$. If $i: Y \setminus L \hookrightarrow Y$ is
 the inclusion, the isomorphism 
 \begin{equation} \label{Equation: H_1 complement of the knots iso to the summands with surfaces}
  \begin{array}{ccc}
   H_1(Y \setminus L) & \longrightarrow & H_1(Y) \oplus \Z_{[\mu_1]} \oplus \ldots \oplus \Z_{[\mu_n]}\\
           \ [\gamma]       & \longmapsto     & \left( i_*([\gamma]), \langle \gamma, S_1\rangle, \ldots, \langle \gamma, S_n \rangle \right)
  \end{array}
 \end{equation}
 gives rise naturally to the abelian cover
 $$\rho_L \in \mathrm{Hom}\left( H_1(Y \setminus L,\Z), \Z^n \right)$$
 of $Y \setminus L$ defined by 
 $$\rho_L([\gamma]) = \left(\langle \gamma, S_1\rangle, \ldots, \langle \gamma, S_n \rangle\right).$$
 Setting $t_i =[\mu_i] \in H_1(Y \setminus L,\Z)$, we can regard $\rho_L([\gamma])$ as a
 monomial in the variables $t_i$:
 $$\rho_L([\gamma]) = t_1^{\langle \gamma, S_1\rangle}\cdots t_n^{\langle \gamma, S_n\rangle}.$$
 In the rest of the paper we will often use this notation.
 
 Note finally that if $Y$ is a homology three-sphere, $\rho_L$ coincides with the universal 
 abelian cover of $Y \setminus L$.
\end{Ex}

If $R$ is a vector field on $X$ satisfying some compatibility condition with $\rho$ (and with $\partial X$ if this is
non-empty), the author relates the Reidemeister-Franz torsion of $(X,\partial X)$ with the \emph{twisted Lefschetz zeta function} of
the flow $\phi_R$. 

\subsubsection{Twisted Lefschetz zeta function of flows}

Let $R$ be a vector field on $X$ and $\gamma$ a closed isolated orbit of $\phi_R$. Pick any point $x \in \gamma$ and let $D$ be a small disk transverse to
$\gamma$ such that $D \cap \gamma = \{x\}$. With this data it is possible to define the Lefschetz sign of $\gamma$ exactly like we did in
Section \ref{Subsection: Contact geometry} for orbits of Reeb vector fields associated to a contact structure $\xi$, but using now $T_xD$ instead of
$\xi_x$. Indeed it is possible to prove that the Lefschetz sign of $\gamma$ does not depend on the choice of $x$ and $D$ and it is an invariant
$\epsilon(\gamma) \in \{-1,1\}$ of $\phi_R$ near $\gamma$.
\begin{Def} \label{Definition: Local zeta function of an orbit}
 The \emph{local Lefschetz zeta function} of $\phi_R$ near $\gamma$ is the formal power series
 ${\zeta}_{\gamma}(t) \in \Z[[t]]$ defined by
 $${\zeta}_{\gamma}(t) := \exp \left(\sum_{i \geq 1} \epsilon(\gamma^i) \frac{t^i}{i}\right).$$
\end{Def}

Let now $\widetilde{X} \stackrel{\pi}{\rightarrow} X$ be an abelian cover with deck transformation group $G$ and let
$\rho = \rho(\pi) \in H^1(X,G)$. Suppose that all the periodic orbits of $\phi_R$ are isolated.
\begin{Def}
 We define the \emph{$\rho$-twisted Lefschetz zeta function of $\phi_R$} by
 $${\zeta}_{\rho}(\phi_R) := \prod_{\gamma}{\zeta}_{\gamma}\left( \rho([\gamma]) \right),$$
 where the product is taken over the set of simple periodic orbits of $\phi_R$.
\end{Def}
When $\rho$ is understood we will write directly $\zeta(\phi_R)$ and we will call it twisted Lefschetz zeta function of $\phi_R$.

We remark that in \cite{Fri} the author defines ${\zeta}_{\rho}(\phi_R)$ in a slightly different way and then he proves (Theorem 2) that,
under some assumptions that we will state in the next subsection, the two definitions coincide.
\begin{Notation}
 Suppose that $\rho \in H^1(X,\Z^n)$ is an abelian cover of $X$ and chose a generator
 $(t_1,\ldots,t_n)$ of $\Z^n$. Then, with a similar notation to that of Example \ref{Example: Abelian cover determined by a link}, 
 we will often identify ${\zeta}_{\rho}(\phi_R)$ with an element of $\Z[[t_1^{\pm 1},\ldots,t_n^{\pm 1}]]$.  
\end{Notation}

\subsubsection{Torsion and flows}
In \cite{Fri} Fried relates the Reidemeister torsion of an abelian cover $\rho$ of a (non necessarily closed) three-manifold $X$
with the twisted Lefschetz zeta
function of certain flows. In particular in Section 5 he considers a kind of torsion that he calls \emph{Alexander quotient} and denotes by
$\mathrm{ALEX}_{\rho}(X)$: the reason for the ``quotient'' comes from the fact that Fried uses a definition of the Reidemeister torsion only
up to the choice of a sign (this is the ``refined Reidemeister torsion'' of \cite{Tu}), while $\mathrm{ALEX}_{\rho}(X)$ is defined up to an
element in the Abelian group of deck transformations of $\rho$ (see also \cite{Coh}).

In fact one can check that $\mathrm{ALEX}_{\rho}(X)$ is exactly the Reidemister-Franz torsion $\tau$ considered in \cite{OS5}.
In particular, when $X$ is the complement of an $n$-component link $L$ in $S^3$ and $\rho$ is the universal abelian cover of $X$, then
\begin{equation} \label{Equation: Relation between ALEX and Delta}
   \mathrm{ALEX}(S^3 \setminus L) \doteq \left\{ \begin{array}{cc}
                                                \Delta_L(t_1,\ldots,t_n) & \mbox{if } n > 1\\
                                                 & \\
                                                \Delta_L(t)/({1-t})  & \mbox{if } n = 1
                                               \end{array}\right. .
\end{equation}
where we removed $\rho = id_{H_1(S^3 \setminus L,\Z)}$ from the notation (see \cite[Section 8]{Fri} and \cite{Tu}). 

Since the notation ``$\tau$'' is ambiguous, we follow \cite{Fri} and we refer to the Reidemeister-Franz as the Alexander quotient, that will be
indicated $\mathrm{ALEX}_{\rho}(X)$.

\vspace{0.3 cm}

In order to relate $\mathrm{ALEX}_{\rho}(X)$ to the twisted Lefschetz zeta function of the flow $\phi_R$ of a vector field $R$,
Fried assumes some hypothesis on $R$.

The first condition that $R$ must satisfy is the \emph{circularity}. 
\begin{Def}
 A vector field $R$ on $X$ is \emph{circular} if there exists a $C^1$ map $\theta : X \rightarrow S^1$ such that $d \theta (R) > 0$.
\end{Def}
If $\partial X = \emptyset$ this is equivalent to say that $R$ admits a global cross section. Intuitively, the circularity condition on $R$ 
allows to define a kind of first return map of $\phi_R$. 

Suppose $R$ circular and consider $S^1 \cong \frac{\R}{\Z}$ with $\R$-coordinate $t$. The cohomology class
$$u_{\theta} := \theta^*([dt]) \in H^1(X,\Z)$$
is then well defined.
\begin{Def}
 Given an abelian cover $\widetilde{X} \stackrel{\pi}{\rightarrow} X$ with deck transformations group $G$, let $\rho = \rho(\pi) \in H^1(X,G)$ be
 the corresponding cohomology class. A circular vector field $R$ on $X$ 
 is \emph{compatible} with $\rho$ if there exists a homomorphism $v : G \rightarrow \R$ such that $v \circ \rho = u_{\theta}$, where $\theta$
 and $u_{\theta}$ are as above.
\end{Def}
\begin{Ex} \label{Example: The universal Abelian cover is always compatible}
 The universal abelian cover corresponds to $\rho = id : H_1(X,\Z) \rightarrow H_1(X,\Z)$, so it is automatically
 compatible with any circular vector field on $X$. 
\end{Ex}

The following theorem is not the most general result in \cite{Fri} but it will be enough for our purposes:

\begin{Thm}[Theorem 7, \cite{Fri}] \label{Theorem: ALEX equals the twisted zeta function} Let $X$ be a three manifold and $\rho \in H^1(X,G)$ an
 abelian cover. Let $R$ be a non-singular, circular and non degenerate vector field on $X$ compatible with $\rho$. Suppose moreover that, if
 $\partial X \neq \emptyset$, then $R$ is transverse to $\partial X$ and pointing out of $X$. Then
 $$\mathrm{ALEX}_{\rho}(X) \doteq {\zeta}_{\rho}(\phi_R),$$
 where the symbol $\doteq$ denotes the equivalence up to multiplication for an element $\pm g$, $g \in G$.
\end{Thm}
An immediate consequence is the following
\begin{Cor} 
 If $L$ is any $n$-component link in $S^3$, let $\mathcal{N}(L)$ be a tubular neighborhood of $L$ and pose $N = S^3 \setminus\mathcal{N}(L)$.
 Let $R$ be a non-singular circular vector field on $N$, transverse to $\partial N$ and pointing out of $N$. Then
 \begin{equation} \label{Equation: Relation between zeta and Delta}
   {\zeta}(\phi_R) \doteq \left\{ \begin{array}{cc}
                                                \Delta_L(t_1,\ldots,t_n) & \mbox{if } n > 1\\
                                                 & \\
                                                \Delta_L(t)/({1-t})  & \mbox{if } n = 1
                                               \end{array}\right. .
\end{equation}
\end{Cor}

\vspace{0.5 cm}

\subsection{Results}

In the next subsections we prove Theorem \ref{Theorem: Euler characteristic of ECK}, that will be obtained as a consequence of the following
more general result.
Recall that an $n$-link $L \subset Y$ determines the abelian cover $\rho_L \in H^1(Y\setminus L, \Z^n)$ of $Y\setminus L$ 
given in Example \ref{Example: Abelian cover determined by a link}. When $Y$ is a homology three-sphere, we have
$$\rho_L \equiv \mathbbm{1} \colon H_1(Y\setminus L)\longrightarrow H_1(Y\setminus L) \cong\Z^n.$$
In order to simplify the notations, we remove $\rho_L$ from the notations of the Alexander quotient and of the twisted Lefschetz zeta function:
\begin{eqnarray*}
 \mathrm{ALEX}(Y\setminus L) & := & \mathrm{ALEX}_{\mathbbm{1}}(Y\setminus L);\\
 \zeta(\phi) & := & \zeta_{\mathbbm{1}}(\phi).
\end{eqnarray*}

Let $(t_1,\ldots,t_n)$ be a basis for $H_1(Y\setminus L)$, where $[\mu_i] = t_i$ for $\mu_i$ positively oriented meridian of $K_i$.
\begin{Thm} \label{Theorem: Euler characteristics of ECK and ALEX}
 Let $L$ be an $n$-link in a homology three-sphere $Y$.
 Then there exists a contact form $\alpha$ such that
 \begin{equation*}
  \chi_{t_1,\ldots,t_n}(ECK(L,Y,\alpha)) \doteq \mathrm{ALEX}(Y \setminus L).
 \end{equation*}
\end{Thm}

The proofs of theorems \ref{Theorem: Euler characteristic of ECK} and \ref{Theorem: Euler characteristics of ECK and ALEX} will be carried
on in two main steps: in Subsection \ref{Subsection: Fibered links} we will prove the theorems in the case of fibered links, while the general
case will be treated in Subsection \ref{Subsection: The general case}.

\subsection{Fibered links} \label{Subsection: Fibered links}

In this subsection we prove theorems \ref{Theorem: Euler characteristic of ECK} and  \ref{Theorem: Euler characteristics of ECK and ALEX} for
fibered links. Let $(L,S,\phi)$ be an open book decomposition of a homology three-sphere $Y$ and let $\alpha$ be an adapted
contact form on $Y$. In particular, with our definition, $\alpha$ is also adapted to $L$.

In order to prove the theorems above we want to express the Euler characteristic $\chi_{t_1,\ldots,t_n}(ECK(L,Y,\alpha))$ in terms of the
twisted Lefschetz zeta function of
the Reeb flow $\phi_{R}$ of $R=R_{\alpha}$ and then apply Theorem \ref{Theorem: ALEX equals the twisted zeta function}. The first thing that one should
do is then to check if $\phi_{R}$ and $\rho_L$ satisfy the hypothesis of that theorem. Unfortunately this is not the case.
The needed properties are in fact the following:
\begin{enumerate}
 \item $R$ is non-singular and circular;
 \item $R$ is compatible with $\rho_L$;
 \item $R$ is non-degenerate;
 \item $R$ is transverse to $\partial V(L)$ and pointing out of $Y \setminus \mathring{V}(L)$,
\end{enumerate}
where $\mathring{V}(L) = int(V(L))$.

In our situation only properties 1 and 2 are satisfied. Indeed, by the definition of open book decomposition, there is a natural fibration
$\theta : Y \setminus \mathring{V}(L) \rightarrow S^1 \cong \frac{\R}{\Z}$ such that the surfaces $\theta^{-1}(t)$ are the pages of the open book.
The fact that $\alpha$ is adapted to $(L,S,\phi)$ implies that $R$ is always positively transverse to the pages. This evidently implies that
$d\theta(R) >0$ so that $R$ is circular.


The fact that $R$ is compatible with $\rho_L$ (that coincides with the universal abelian cover of $Y \setminus \mathring{V}(L)$) comes from
Example \ref{Example: The universal Abelian cover is always compatible}.

However properties 3 and 4 above are not satisfied. Indeed, after the M-B perturbation of $T_2$, $R$ is tangent to $\partial V(L)$ on
$\bar{e}_+$ and $\bar{h}_+$. Moreover, as observed in Subsection \ref{Subsection: Morse-Bott Theory}, the M-B perturbations near the two tori $T_1$ and
$T_2$ may create degenerate orbits. 

What we will do is then to perturb $R$ to get a new vector field $R'$. This vector field will be defined in $Y \setminus V'(L)$, where
$V'(L) \subset \mathring{V}(L)$ is an open tubular neighborhood of $L$ defined by $V'(L) = V'(K_1) \sqcup \ldots \sqcup V'(K_n)$,
where, using the coordinates of Subsection \ref{Subsection: Open Books}, $\partial(V'(K_i)) = \{y = 2.5\}$.

\begin{Lemma}
 There exists a (non-contact) vector field $R'$ such that:
\begin{enumerate}
 \item[(i)]   $R'$ coincides with $R$ outside a neighborhood of $\mathcal{N}(L)$;
 \item[(ii)]  $R'$ satisfies properties 1-4 above with $V(L)$ replaced by $V'(L)$; 
 \item[(iii)] the only periodic orbits of $R'$ in $\mathcal{N}(V) \setminus V'(L)$ are the four sets of non-degenerate orbits
 $\bar{e},\bar{h},\bar{e}_+,\bar{h}_+$.
\end{enumerate}
\end{Lemma}

Observe that Property (i) implies that the twisted Lefschetz zeta functions of the restrictions of the flows $\phi_{R}$ and $\phi_{R'}$ to
$Y \setminus \mathcal{N}(K)$ coincide, while Property (ii) allows to apply Theorem \ref{Theorem: ALEX equals the twisted zeta function} to
$\phi_{R'}$. 

\proof
A perturbation of $R$ into an $R'$ satisfying the conditions (i)-(iii) can be obtained in more than one way. An example is pictured in Figure
\ref{Figure: Modification of the dynamics non contact} (cf. also Figure \ref{Figure: Dynamic of the M-B perturbation near K}).
We briefly explain how it is obtained. Since the modification of $R$ is non trivial only
inside disjoint neighborhoods of each $K_i$, we will describe it only for a fixed component $K$ of $L$. The characterization of the
perturbation will be presented in terms of perturbation of the lines in a page $S$ of $(L,S,\phi)$ that are invariant under the
first return map $\phi$ of $\phi_R$: we will refer to these curves as to \emph{$\phi$-invariant lines on $S$}.
Note that these curves are naturally oriented by the flow.

Outside a neighborhood of $\partial V'$ one
can see this perturbation in terms of a perturbation of $\phi$ into another monodromy $\phi'$, 
and $R'$ is the vector field $\partial_t$ in $Y \setminus V'(L) \cong \frac{S \times [0,1]}{(x,1) \sim (\phi'(x),0)}$, where $t$ is the 
coordinate of $[0,1]$.

\begin{figure} [h] 
  \begin{center}
   \includegraphics[scale = .251]{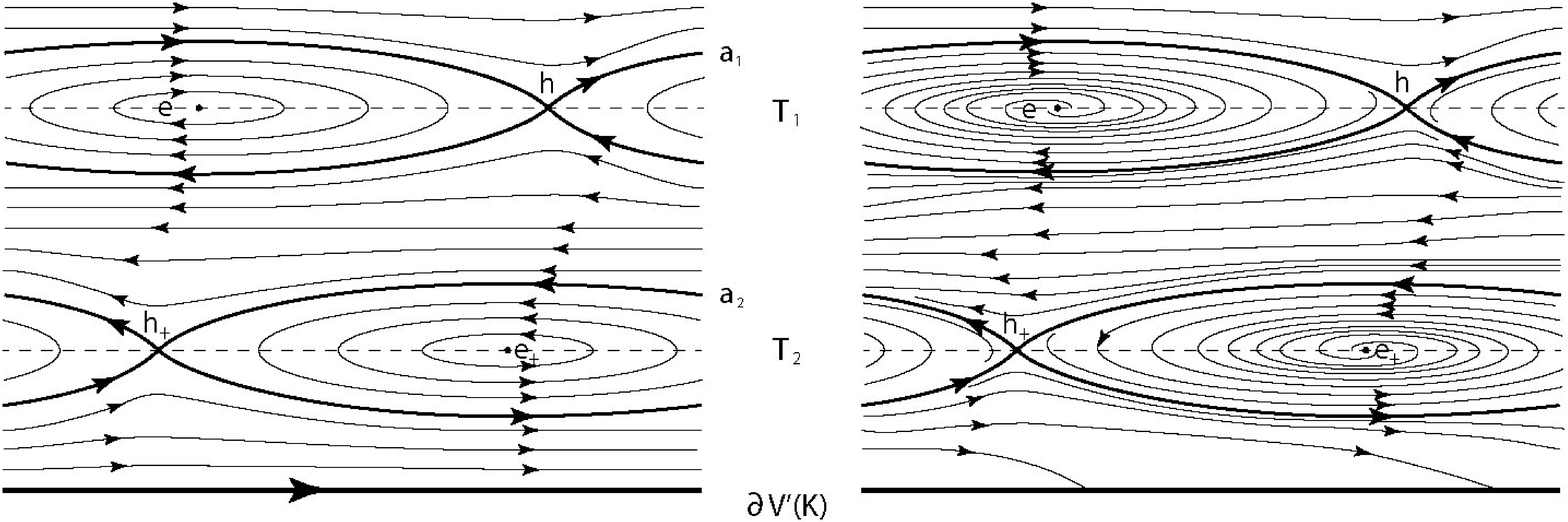}
  \end{center}
  \caption{The dynamics of the vector fields $R$ and $R'$ near $\mathcal{N}(V) \setminus V'(L)$. Each oriented line represents an invariant subset
  of a page of $(L,S,\phi)$ under the first return map $\phi$ at the left and $\phi'$ at the right (the invariant lines $a_1$ and $a_2$ are stressed).
  The situation at the left is the same depicted in Figure \ref{Figure: Dynamic of the M-B perturbation near K}.}
  \label{Figure: Modification of the dynamics non contact}
\end{figure}

Observe first that the only periodic orbit in the (singular) $\phi$-invariant line $a_1$ containing $h$ (in correspondence to the
singularity) is exactly $h$. Similarly, the only periodic orbit in the $\phi$-invariant singular flow line $a_2$ containing $h_+$
is precisely $h_+$. Denote $A_i \subset Y$ the mapping torus of $(a_i,\phi|_{a_i})$, $i = 1,2$. We modify $R$ separately inside the regions
of $(Y \setminus V'(K)) \setminus (A_1 \sqcup A_2)$ as follows.

In the region containing $e$ (and with boundary $A_1$), the set of $\phi$-invariant lines (the elliptic lines in the picture at left)
is perturbed in a set of $\phi'$-invariant spiral-kind lines (at right), each of which is
negatively asymptotic to $a_1$ and positively asymptotic to $e$. It is easy to see that after the perturbation the only
periodic orbit in the interior of this region is $e$. Moreover, we can arrange the perturbation in a way
that the differential $\mathfrak{L}^{R'}_{e}$ of the first
return map on $S$ of $\phi_{R'}$ along $e$, coincides, up
to a positive factor smaller then $1$, with $\mathfrak{L}^{R}_{e}$, so that the Lefschetz sign $\epsilon(e)$ of $e$ is still $+1$.

A similar perturbation is done in the region of $(Y \setminus V'(K)) \setminus (A_1 \sqcup A_2)$ containing $e_+$, in a way that $e_+$ is the only
periodic orbit of the perturbed vector field $R'$, with still $\epsilon(e_+) = +1$.

The perturbation in the region between $A_1$ and $A_2$ is done by slightly pushing the monodromy
in the positive $y$-direction in a way that the set of $\phi$-invariant lines is perturbed into a set of $\phi'$-invariant lines, 
each of which is negatively asymptotic to $a_1$ and positively asymptotic to $a_2$
(and so in particular there can not exist periodic orbits in this region).

A similar perturbation is done also inside the region between $A_2$ and $\partial V'(K)$, but in this case each $\phi'$-invariant
line is negatively asymptotic to $a_2$ and intersects $\partial V'(K)$ pointing out of the three-manifold. 

Finally we leave $R' = R$ in the rest of the manifold, where $R$ was supposed having only isolated and non degenerate periodic orbits.

Note that the two basis of eigenvectors of $\mathfrak{L}^{R}_{h}$ and $\mathfrak{L}^{R}_{h_+}$ are contained in the tangent spaces of the
curves $a_1$ and, respectively, $a_2$. Since on these curves $\phi_{R} = \phi_{R'}$, the Lefschetz signs of the two orbits are not
changed by the perturbation.

It is easy to convince ourselves that $R'$ satisfies the properties i-iii above.
\endproof

Call $\zeta = \zeta_{\mathbbm{1}}$. Since the Lefschetz
zeta function of a flow depends only on its periodic orbits and their signs, we have the following:
\begin{Cor} \label{Corollary: Zeta functions of R and R' are equals}
 If $R'$ is obtained from $R$ as above, then 
 \begin{eqnarray*}
   \zeta(\phi_{R'}) & = & \zeta(\phi_{R'}|_{(Y \setminus \mathcal{N}(K) \sqcup \{\bar{e},\bar{h},\bar{e}_+,\bar{h}_+\})}) =\\
        & = &\zeta(\phi_{R|_{(Y \setminus \mathcal{N}(K)}})\cdot \prod_{\gamma \in \{\bar{e},\bar{h},\bar{e}_+,\bar{h}_+\}} \zeta_{\gamma}([\gamma])).
 \end{eqnarray*}
 where $[\gamma]$ is the homology class of $\gamma$ in $H_1(Y \setminus \mathcal{N}(K))$.
\end{Cor}

Now we want to compute more explicitly the twisted Lefschetz zeta function $\zeta(\phi_{R'})$. Let us begin with the local
Lefschetz zeta function of the simple orbits (see Definition \ref{Definition: Local zeta function of an orbit}). 
\begin{Lemma} \label{Lemma: Zeta functions of the orbits}
 Let $\gamma$ be an orbit of $R$ or $R'$. Then:
 \begin{equation}
 {\zeta}_{\gamma}(t) =\left\{\begin{array}{cll}
                                          (1-t)^{-1} & = 1+t+t^2+\ldots & \mbox{if } \gamma \mbox{ elliptic;} \\
                                          1-t &  & \mbox{if } \gamma \mbox{ positive hyperbolic;} \\
                                          1+t &  & \mbox{if } \gamma \mbox{ negative hyperbolic;}
                                           \end{array} \right.
 \end{equation}
\end{Lemma}
\proof Remember that the Lefschetz number of $\gamma$ is $\varepsilon(\gamma) = +1$ if $\gamma$ is elliptic or negative hyperbolic and
 $\varepsilon(\gamma) = - 1$ if $\gamma$ is positive hyperbolic. We prove here only the case of $\gamma$ positive hyperbolic, leaving to the
 reader the other similar computations.

 If $\gamma$ is positive hyperbolic then all the iterated are also positive hyperbolic and $\epsilon(\gamma^i) = -1$ for every $i >0$.
  Then:
 \begin{equation*}
 \begin{array}{cclc}
   {\zeta}_{\gamma}(t) & = & \exp \left(\sum_{i \geq 1} - \dfrac{t^i}{i}\right) & = \\
                       & = & \exp \left( \sum_{i \geq 1} (-1)^{i+1} \dfrac{(-t)^i}{i}\right) & = \\ 
                               & = & \exp \left( \log(1-t) \right) & = \\
                               & = & 1-t .    &
 \end{array}
 \end{equation*}

\endproof
\begin{Obs}
 Note that the equations above are exactly the generating functions given by Hutchings in \cite[Section 2]{Hu3}.
\end{Obs}
Let $\mu_i$ be a positive meridian of $K_i$ for $i \in \{1,\ldots,n\}$ and set $t_i = [\mu_i] \in H_1(Y \setminus K)$; fix moreover a Seifert
surface $S_i$ for each $K_i$. Recall that, for a given $X \subset Y$, $\mathcal{P}(X)$ denotes the set of simple Reeb orbits contained in $X$. 
\begin{Cor}
 The twisted Lefschetz zeta function of $\phi_{R|_{(Y \setminus \mathcal{N}(L))}}$ is
 \begin{equation*} \label{Equation: Euler characteristics of ECC in temrs of zeta function}
  \zeta(\phi_{R|_{(Y \setminus \mathcal{N}(L))}}) = \prod_{\gamma \in \mathcal{P}(Y \setminus \mathcal{N}(L))} \zeta_{\gamma}([\gamma]),
 \end{equation*}
 where $\zeta_{\gamma}([\gamma])$ is determined as follows:
 \begin{itemize}
  \item[\textbullet]  if $\gamma$ is elliptic then:
   $$\zeta_{\gamma}(\rho_L(\gamma)) = \left(1-\prod_{i=1}^n t_i^{\langle\gamma,S_i \rangle}\right)^{-1} = \sum_{l=0}^{\infty}\left(\prod_{i=1}^n t_i^{\langle\gamma,S_i \rangle}\right)^l;$$
  \item[\textbullet] if $\gamma$ is positive hyperbolic then:
   $$\zeta_{\gamma}(\rho_L(\gamma)) = 1-\prod_{i=1}^n t_i^{\langle\gamma,S_i \rangle};$$
  \item[\textbullet] if $\gamma$ is negative hyperbolic then:
   $$\zeta_{\gamma}(\rho_L(\gamma)) = 1+\prod_{i=1}^n t_i^{\langle\gamma,S_i \rangle}.$$
 \end{itemize}
\end{Cor}
\proof
 This is an easy computation. It suffices to substitute the monomial representation of $\rho_L([\gamma)] = [\gamma]$ given in Example
 \ref{Example: Abelian cover determined by a link} in the expression of the Lefschetz zeta function of Lemma
 \ref{Lemma: Zeta functions of the orbits}. 
\endproof

\begin{proof}[Proof. (of Theorem \ref{Theorem: Euler characteristics of ECK and ALEX} for fibered links)]
 To finish the proof it remains essentially to prove that 
 \begin{equation}\label{Equation: Product formula of ECC in terms of zeta functions}
   \chi_{t_1,\ldots,t_n}\left(ECC(L,Y,\alpha)\right) = \zeta(\phi_{R|_{(Y \setminus \mathcal{N}(L)}})\cdot \prod_{\gamma \in \{\bar{e},\bar{h},\bar{e}_+,\bar{h}_+\}} \zeta_{\gamma}([\gamma])).
 \end{equation}
 This is easy to verify recursively on the set of simple orbits. Suppose
 $\delta = \prod_j \delta_j^{k_j}$ is an orbit set and let $\gamma$ be an orbit such that $\gamma \neq \delta_j$ for any $j$.
 Then the set of all multiorbits that we can build using $\delta$ and $\gamma$ can be expressed via the product formulae:
 \begin{equation} \label{Equation: Recurrence definition of the orbit sets}
  \begin{array}{lcl}
   \delta \cdot \{\emptyset,\gamma,\gamma^2,\ldots\} & & \mbox{ if } \gamma \mbox{ is elliptic};\\
   \delta \cdot \{\emptyset,\gamma\} & & \mbox{ if } \gamma \mbox{ is hyperbolic}.
  \end{array}
 \end{equation}
 As remarked in Subsection \ref{Subsection: ECH for three-manifolds}, the index parity formula
 \ref{Equation: Index parity formula of the ECH index} implies that the Lefschetz sign endows
 the $ECH$-chain complex with an absolute degree and it coincides with the parity of the $ECH$-index. 
 Then the contribution to the graded Euler characteristic of $\delta \cdot \gamma^l$, for any $l$ ($l \in \mathbb{N}$ if $\gamma$ is elliptic
 and $l \in \{0,1\}$ if $\gamma$ is hyperbolic) is:
 $$\epsilon(\delta) \prod_{i=1}^n t_i^{\langle\delta,S_i \rangle} \cdot \left(\epsilon(\gamma) \prod_{i=1}^n t_i^{\langle\gamma,S_i \rangle}\right)^l.$$

 Substituting the last formula in Expressions \ref{Equation: Recurrence definition of the orbit sets}, the total contribution of the product
 formulae to the Euler characteristic are:
 \begin{itemize}
  \item[\textbullet] $\epsilon(\delta) \prod_{i=1}^n t_i^{\langle\delta,S_i \rangle} \cdot \sum_{l=0}^{\infty}\left(\prod_{i=1}^n t_i^{\langle\gamma,S_i \rangle}\right)^l$ if $\gamma$ is elliptic,
  \item[\textbullet] $\epsilon(\delta) \prod_{i=1}^n t_i^{\langle\delta,S_i \rangle} \cdot \left(1 - \prod_{i=1}^n t_i^{\langle\gamma,S_i \rangle}\right)$ if $\gamma$ is positive hyperbolic,
  \item[\textbullet] $\epsilon(\delta) \prod_{i=1}^n t_i^{\langle\delta,S_i \rangle} \cdot \left(1 + \prod_{i=1}^n t_i^{\langle\gamma,S_i \rangle}\right)$ if $\gamma$ is negative hyperbolic,
 \end{itemize}
 that is
 $$\epsilon(\delta) \prod_{i=1}^n t_i^{\langle\delta,S_i \rangle} \cdot \zeta_{\gamma}([\gamma]).$$
 Starting from $\delta = \emptyset$, Equation \ref{Equation: Product formula of ECC in terms of zeta functions} follows by induction on the set
 of $\gamma \in \mathcal{P}\left((Y \setminus \mathcal{N}(L)) \sqcup \{\bar{e},\bar{h},\bar{e}_+,\bar{h}_+\}\right)$.

 The theorem follows then by applying Corollary \ref{Corollary: Zeta functions of R and R' are equals} and Theorem
 \ref{Theorem: ALEX equals the twisted zeta function} to the flow of $R'$.
\end{proof}

\vspace{0.2 cm}

\proof[Proof. (of Theorem \ref{Theorem: Euler characteristic of ECK} for fibered links)]
 Theorem \ref{Theorem: Euler characteristics of ECK and ALEX} and Equation \ref{Equation: Relation between ALEX and Delta}
 immediately imply Equation \ref{Equation: Euler characteristics of ECK full}.

 To prove the result in the hat version we reason again at the level of chain complexes. Recall that, if $N := Y \setminus \mathring{\mathcal{N}}(L)$,
 by the definition of the $ECK$-chain complexes:
 \begin{eqnarray*}
  ECC(L,Y,\alpha) & = & ECC^{\bar{e}_+,\bar{h}_+}(N,\alpha) = \\
                  & = & ECC^{\bar{h}_+}(N,\alpha) \bigotimes_{i = 1}^n \langle \emptyset, e_i^+, (e_i^+)^2,\ldots \rangle = \\
                  & = & \widehat{ECC}(L,Y,\alpha) \bigotimes_{i = 1}^n \langle \emptyset, e_i^+, (e_i^+)^2,\ldots \rangle
 \end{eqnarray*}
 where the second line comes from the product formula \ref{Equation: Recurrence definition of the orbit sets} and the fact that $e_i^+$ is 
 elliptic for any $i$. Taking the graded Euler characteristics as above we have:
 \begin{eqnarray*}
  \chi(ECC(L,Y,\alpha)) & = & \chi(\widehat{ECC}(L,Y,\alpha)) \cdot \prod_{i= 1}^n \zeta_{e_i^+}([e_i^+]) =\\
                        & = & \chi(\widehat{ECC}(L,Y,\alpha)) \cdot \prod_{i= 1}^n \frac{1}{1-t_i},
 \end{eqnarray*}
 where the last equality comes from the fact that $[e_i^+] = [\mu_i] = t_i \in H_1(Y \setminus L)$. If $Y= S^3$, last equation and
 Equation \ref{Equation: Euler characteristics of ECK full} evidently imply Equation \ref{Equation: Euler characteristics of ECK hat}.
\endproof

\subsubsection{Hamiltonian Floer homology}

If $(L,S,\phi)$ is an open book decomposition of $Y$, one can think of $ECK(L,Y,\alpha)$ and $\widehat{ECK}(L,Y,\alpha)$ as
invariants of the pair $(S,\phi)$ and the adapted $\alpha$. It is interesting to note that the Euler characteristic
of $ECK_1(L,Y,\alpha)$ with respect to the surface $S$ (see Example \ref{Example: ECK for open books})
coincides with the sum of the Lefschetz signs of the Reeb orbits of period $1$ in the interior of $S$, i.e. the Lefschetz number
$\Lambda(\phi)$ of $\phi$.

In fact, given $Y$ (not necessarily an homology three-sphere) we can say even more about this fact by relating $ECK_1(L,Y,\alpha)$ to the
\emph{symplectic Floer homology $SH(S,\phi)$} of $(S,\phi)$, whose Euler characteristic is
precisely $\Lambda(\phi)$. Here we are considering the version of $SH(S,\phi)$ for surfaces with boundary that is slightly
rotated by $\phi$ in the positive direction, with respect to the orientation induced by $S$ on $\partial S$  (see for example \cite{Cott} and \cite{Ga}).
\begin{Prop} \label{Proposition: ECK_1 iso to SH for open books}
 Let $(L,S,\phi)$ be an open book decomposition of a three-manifold $Y$ and let $\alpha$ be an adapted contact form. Then
 $$ECK_1(L,Y,\alpha) \cong SH(S,\phi),$$
 where the degree of $ECK(L,Y,\alpha)$ is computed using a page of the open book.
\end{Prop}

The proof of last proposition passes through another Floer homology theory closely related to $ECH$, which is the \emph{periodic Floer homology},
denoted by $PFH$, and defined by Hutchings (see \cite{Hu1}). 
Given a symplectic surface $(S,\omega)$ (here with possibly empty boundary) and a symplectomorphism $\phi :S \rightarrow S$, consider
the mapping torus
$$N(S,\phi) = \frac{S \times [0,2]}{(x,2) \sim (\phi(x),0)}.$$
Then $PFH(N(S,\phi))$ is defined in an analogous way than $ECH$ for an open book but replacing the Reeb vector field with a stable 
Hamiltonian vector field $R$ parallel to $\partial_t$, where $t$ is the coordinate of $[0,2]$: we refer the reader to \cite{Hu1} or
\cite{HS} for the details. 

The chain group $PFC(N(S,\phi))$ is the free $\mathbb{Z}_2$ module generated by orbit sets of $R$ and the boundary map counts index
1 holomorphic curves in the symplectization; then, under some condition on $\phi$, the associated homology $PFH(N(S,\phi))$ is well
defined. Homology groups $PFH_i(Y(S,\phi))$ associated to the chain groups $PFC_i(N(S,\phi))$ generated
by degree-$i$ multiorbits are also well defined.

If $(S,\phi)$ is an open book as in the subsections above, $\partial S$ is connected and $N$ is the associated mapping torus,
in \cite{CGH3} the following is proved:

\begin{Thm}[\cite{CGH3}, Theorem 3.6.1] \label{Theorem: PFH_i iso to ECH_i fo open books}
 If $\alpha$ is a contact form adapted to $(S,\phi)$ then there exists a stable Hamiltonian structure such that
 for any $i \geq 0$, 
\begin{equation}
 \begin{array}{ccc}
  PFH_i(N) & \cong & ECH_i(N,\alpha) \\
 \end{array}
\end{equation}
(here we are using a simplified notation which is different from that in \cite{CGH3}). 
\end{Thm}


$PFC_1(N(S,\phi))$ is generated by orbits of period $1$, which are in bijective correspondence with the set 
$\mathrm{Fix}(\phi)$ of the fixed points of $\phi$ via the map
\begin{equation}
  \begin{array}{ccc}
   \mathcal{O}_1(int(N)) & \longrightarrow & \mathrm{Fix}(\phi)\\
    \gamma               & \longmapsto     & \gamma \cap S ,
  \end{array}
 \end{equation}
which moreover evidently respects the Lefschetz signs. Then this correspondence induces an isomorphism between $PFC_1(N(S,\phi))$ 
and a chain complex of $SH(S,\phi)$. Indeed the following holds (see for example \cite{HS}):
\begin{Prop} \label{Proposition: PFH_1 iso to SH}
 The correspondence above induces an isomorphism
 $$PFH_1(N(S,\phi)) \cong SH(S,\phi).$$
\end{Prop}

\proof[Proof of Proposition \ref{Proposition: ECK_1 iso to SH for open books}]
 This is an easy consequence of the definitions and the results above.
 By Lemma \ref{Lemma: ECK iso to ECH of int N} we have
 $$ECK_1(L,Y,\alpha) \cong ECH_1(int(N),\alpha).$$
 Observing that the proof of Theorem \ref{Theorem: PFH_i iso to ECH_i fo open books} given in \cite{CGH3} works also if
 $\partial S$ is disconnected we get
 $$ECK_1(L,Y,\alpha) \cong PFH_1(N(S,\phi)).$$
 The result then follows applying Proposition \ref{Proposition: PFH_1 iso to SH}.
\endproof
 
We get an interesting consequence of this fact when also the Alexander degree of Heegaard-Floer knot homology of a fibered knot is computed
with respect to (the homology class of) a  page of the associated open book. Indeed, using the symmetrized degree adopted by
Ozsv\'{a}th and Szab\'{o}, we know that $HFK^-_{-g}(K,Y)$ is isomorphic to a copy of $\Z/2$ generated by the class of the corresponding
contact element. Moreover, whenever $\chi(ECK(K,Y,\alpha)) = \chi(HFK^-(K,Y))$, we have also that \emph{$HFK^-_{-g+1}(K,Y)$ categorifies
$\Lambda(\phi)$}. 

Obviously, if the conjectures stated before hold, then
$HFK^-_{-g+1}(K,Y) \cong SH(S,\phi)$.

 

\subsection{The general case} \label{Subsection: The general case}
In this subsection we prove theorems \ref{Theorem: Euler characteristic of ECK} and  \ref{Theorem: Euler characteristics of ECK and ALEX} in the 
general case. 

The first approach that one could attempt to apply Theorem \ref{Theorem: ALEX equals the twisted zeta function} to a general link $L \subset Y$ is to
look for a contact form on $Y$ that is compatible with $L$ and whose Reeb vector field is circular outside a neighborhood of $L$.
Unfortunately we will not be able to find such a contact form.
The basic idea to solve the problem consists in two steps:
\begin{enumerate}
 \item[Step $1$.] find a contact form $\alpha$ on $Y$ which is compatible with $L$ and for which
  there exists a finite decomposition $Y \setminus L = \bigsqcup_i X_i$ for which $R=R_{\alpha}$ is circular in each $X_i$;
 \item[Step $2$.] apply Theorem \ref{Theorem: ALEX equals the twisted zeta function} separately in each $X_i$ to get the result: this can be done
  using the (more general results) in Sections 6 of \cite{Fri}.
\end{enumerate}
On the other hand the special decomposition of $Y \setminus L$ that we find in Step 1 will allow us to follow an easier way and we will substitute
Step 2 by:
\begin{enumerate}
 \item[Step $2'$.] apply repeatedly the \emph{Torres formula} for links to get the result.
\end{enumerate}
Torres formula, first proved in \cite{Tor}, is a classical result about Alexander polynomial, which essentially explains how,
starting from the Alexander polynomial of a given link $L$, to compute the Alexander polynomials of any sub-link of $L$ .

\subsubsection{Preliminary}

The key ingredient to solve the Step 1 of our strategy is the following:
\begin{Prop} \label{Prop: Any L is part o a fibered L'}
 Let $L = K_1 \sqcup \ldots \sqcup K_n \subset Y$ be an $n$-components link and let $\xi$ be any fixed contact structure on $Y$. Then there exists
 an $m$-components link $L' \subset Y$ with $m \geq n$ and such that:
 \begin{enumerate}
  \item $L' = L \sqcup K_{n+1} \sqcup \ldots \sqcup K_m$;
  \item $L'$ is fibered and the associated open book decomposition of $Y$ supports $\xi$.
 \end{enumerate}
\end{Prop}
This result has been proved in the case of knots by Guyard in his Ph.D. thesis (in preparation, \cite{Gu}). Using part of his arguments,
we give here a proof for the case of links. 

%
\proof
 As recalled in Subsection \ref{Subsection: Open Books}, given a contact structure $\xi$ on $Y$, in \cite{Gi} Giroux explicitly constructs an open
 book decomposition of $Y$ that supports a contact form $\alpha$ such that $\ker(\alpha) = \xi$. In the proof 
 of Theorem \ref{Theorem: Giroux correspondance} we saw that such an open book decomposition is built starting from a cellular
 decomposition $\mathcal{D}$ of $Y$ that is compatible with $\xi$. Moreover we recalled that, up to taking a refinement, any cellular
 decomposition of $Y$ can be made compatible with $\xi$ by an isotopy.
 
 Using the simplicial approximation theorem, it is possible to choose a triangulation $\mathcal{D}$ of $Y$ in a way that, up to isotopy, $L$ is
 contained in the $1$-skeleton $\mathcal{D}^1$ of $\mathcal{D}$. Up to take a refinement, we can suppose moreover that $\mathcal{D}$ is adapted
 to $\xi$.
 
 Let $S$ be the $0$-page of the associated open book built via Theorem \ref{Theorem: Construction of an open book adapted to a contact structure}:
 properties 1 and 2 of $S$ reminded during the proof of that theorem, imply that $L \subset \mathrm{int}(S)$ and that,
 if $\mathcal{N}(\mathcal{D}^0)$ is a suitable neighborhood of
 $\mathcal{D}^0$, then it is possible to push $L \setminus \mathcal{N}(\mathcal{D}^0)$ inside $S$ to make it contained in $\partial S$. Note
 that in each strip composing $S \setminus \mathcal{N}(\mathcal{D}^0)$ we have only one possible choice for the direction in which to push
 $L \setminus \mathcal{N}(\mathcal{D}^0)$ to $\partial S$ in a way that the orientation of $L$ coincides with that of $\partial S$.
 
 We would like to extend this isotopy also to $L \cap \mathcal{N}(\mathcal{D}^0)$ to make the whole $L$ contained in $\partial S$. Suppose
 that $B$ is a connected component (homeomorphic to a ball) of $\mathcal{N}(\mathcal{D}^0)$. In particular we suppose that $B \cap S$ is
 connected. Then $L \cap \partial B$  consists of two points $Q_1$ and $Q_2$. The extension is done differently in the following two cases
 (see figure \ref{Figure: Making L contained in the binding}):
 
 \begin{figure}[h]
    \begin{center}
     \includegraphics[width=.46\linewidth]{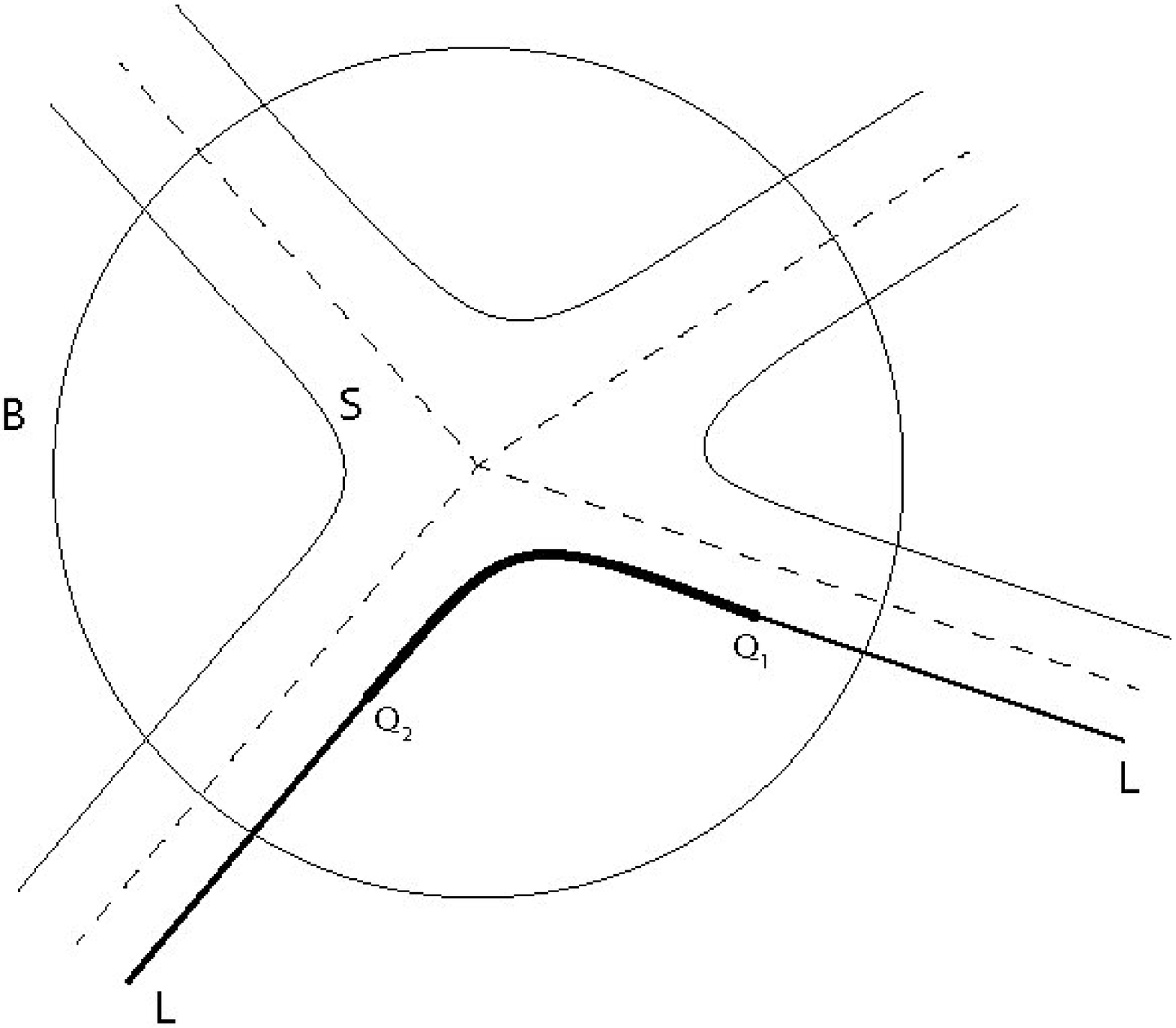}\phantom{dddd}
     \includegraphics[width=.46\linewidth]{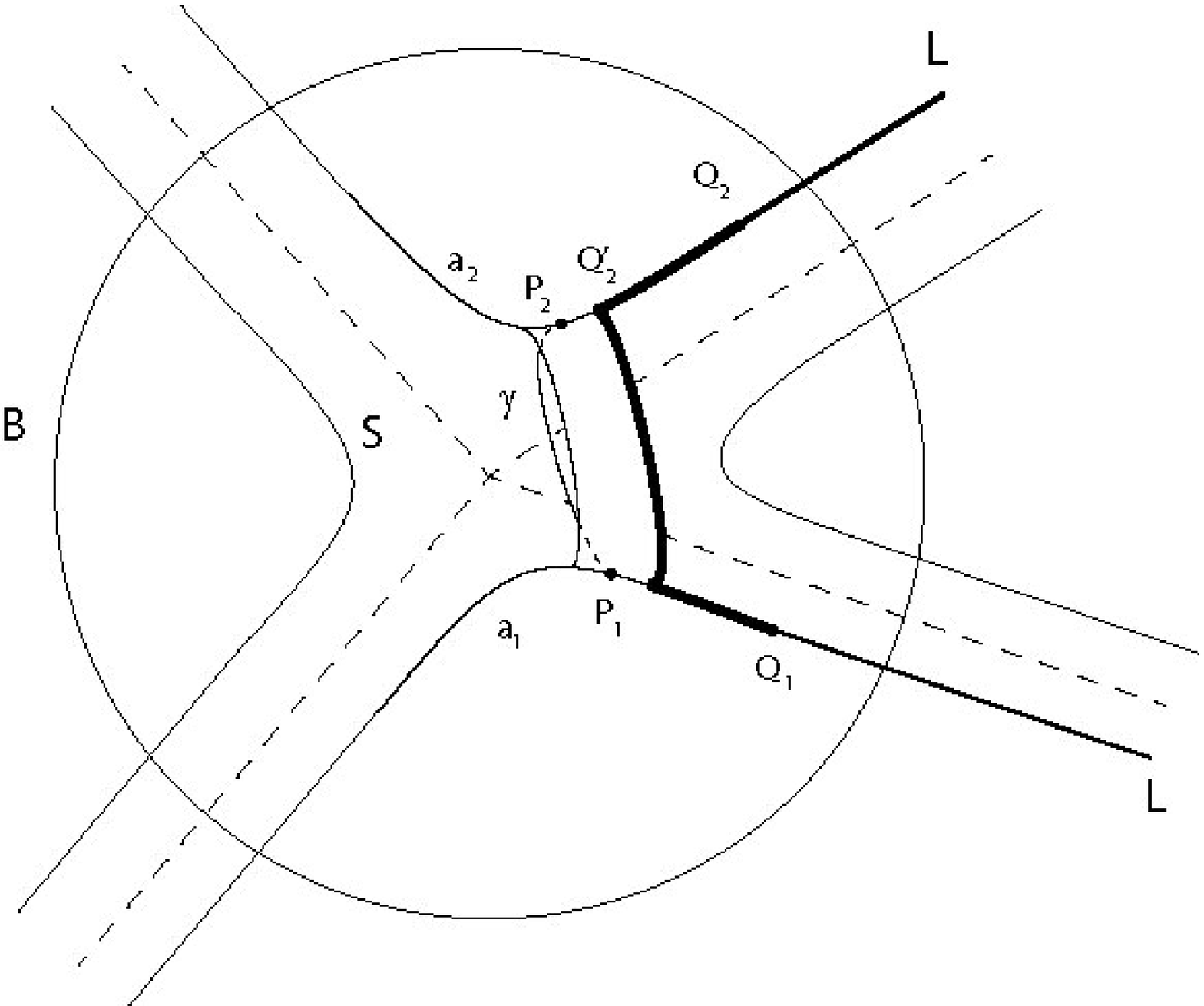}
    \end{center}
    \caption{Making L contained in $\partial S$ in $\mathcal{N}(\mathcal{D}^0)$: easy case at the left and general case at right.
    The dotted lines are $1$-simplexes in $\mathcal{D}^1$, while the bold segments from $Q_1$ to $Q_2$ represent the push-offs of $L$ in
    $\mathcal{N}(\mathcal{D}^0)$.}
    \label{Figure: Making L contained in the binding}
   \end{figure}

 \begin{itemize}
  \item[1.] \emph{Easy case}: this is when $Q_1$ and $Q_2$ belong to the same connected component of $\partial S \cap B$. The isotopy is then
   extended to $B$ by pushing $L \cap B$ to $\partial S \cap B$ inside $S \cap B$ (figure at left);
  \item[2.] \emph{General case}: if $Q_1$ and $Q_2$ belong to (the boundary of) different connected components $a_1$ and $a_2$ of
   $\partial S \cap B$ we proceed as follows.
   
   Let $P_i$ be a point in the interior of $a_i$, $i = 1,2$. Let $\gamma$ be a simple arc in $S \cap B$ from $P_1$ to
   $P_2$ (there exists only one choice for $\gamma$ up to isotopy). Let $S'$ be obtained by positive Giroux stabilization of $S$ along
   $\gamma$ (see figure at the right).
   
   
   Now we can connect $Q_1$ with $a_2$ by an arc in $\partial S'$ crossing once the belt sphere of the $1$-handle of the stabilization;
   let $Q_2'$ be the end point of this arc. Since a Giroux stabilization is compatible with the orientation of $\partial S$,
   $Q_2'$ and $Q_2$ are in the same connected component of $a \setminus \{P_2\}$, so that we can connect them inside
   $\partial S \cap B$ and we are done.
   
  \end{itemize}

  Pushing $L$ to $\partial S$ (and changing $L$ and $S$ as before where necessary) gives a link $\overline{L}$ that is contained in
  $\partial S$. To see that $\overline{L}$ is isotopic to $L$ we have to prove that, for any $B$ as before, the two kinds of push-offs we use
  do no change the isotopy class of $L$.
  
  Clearly the isotopy class of $L$ is preserved in the easy case. For the general case, it suffices to show that substituting the arc
  $L \cap S \cap B$ from $Q_1$ to $Q_2$ with an arc crossing once the belt sphere of the handle does not change the isotopy
  class of $L$. This is
  equivalent to proving that, if $\gamma$ is the path of the Giroux stabilization and $\bar{\gamma} = \gamma \cup c$, where $c$ is the core curve
  of the handle, then $\bar{\gamma}$ bounds a disk in $Y \setminus L$. This can be proved for example by using the particular kind of
  Heegaard diagrams used in \cite{CGH3}. Observe that, if $b$ is the co-core of the handle, then
  $\bar{\gamma}$ is isotopic in $S$ to $b \cup \phi'(b)$, where $\phi'$ is the monodromy on $S'$ given by the Giroux stabilization. We
  finish by observing that, up to a small perturbation near $\partial S$, $b \cup \phi'(b)$ is isotopic to an attaching curve of a Heegaard
  diagram of $Y$.
\endproof

\vspace{0.3 cm}

We recall now the Torres formula that we will use in the second step of our proof of Theorem \ref{Theorem: Euler characteristics of ECK and ALEX}.
Since we need to consider the Alexander quotient as a polynomial, we will use the same convention adopted for the graded Euler characteristic
and we will express the variables as subscripts of the symbol $\mathrm{ALEX}$.

\begin{Thm}[Torres formula] \label{Theorem: Torres formula}
 Let $L= K_1 \sqcup \ldots \sqcup K_n$ be an $n$-link in a homology three-sphere $Y$, $K_{n+1}$ a knot in $Y \setminus L$ and
 $L' = L \sqcup K_{n+1}$. Let $S_i$ be a Seifert surface for $K_i$, $i \in \{1,\ldots,n+1\}$. Then
 \begin{eqnarray*}
  \mathrm{ALEX}_{t_1,\ldots,t_n,1}(Y \setminus L') \doteq \mathrm{ALEX}_{t_1,\ldots,t_n}(Y \setminus L) \cdot \left(1 - \prod_{i=1}^n t_i^{\langle K_{n+1} , S_i \rangle} \right),                                                
 \end{eqnarray*}
 where $\mathrm{ALEX}_{t_1,\ldots,t_n,1}(Y \setminus L')$ indicates the polynomial $\mathrm{ALEX}_{t_1,\ldots,t_{n+1}}(Y \setminus L')$ evaluated
 in $t_{n+1} = 1$.
\end{Thm}

We refer the reader to \cite{Tor} for the original proof. See also \cite{Fra} for a proof making use of techniques of dynamics.
We also mention that in \cite{Cim} a proof of this theorem is provided making use only of elementary techniques about Seifert surfaces;
moreover a generalization of the formula to links in any three-manifold is given in \cite{Tu}.
\begin{proof}[Sketch of the proof.]
 Apply Theorem \ref{Theorem: ALEX equals the twisted zeta function} to $\mathrm{ALEX}(Y \setminus L)$ using a flow $\phi$ for which
 \begin{enumerate}
  \item $K_{n+1}$ is the only periodic orbit of $\phi$ contained in a neighborhood of $K_{n+1}$;
  \item $K_{n+1}$ is elliptic.
 \end{enumerate}
  The factor
 \begin{equation}
  1 - \prod_{i=1}^n t_i^{\langle K_{n+1} , S_i \rangle} = \left(\zeta_{K_{n+1}}(\rho_L(K_{n+1}))\right)^{-1}
 \end{equation}
 expresses then the fact that $K_{n+1}$ is the only orbit counted in $\mathrm{ALEX}(Y \setminus L)$ and not in $\mathrm{ALEX}(Y \setminus L')$.\\
 The condition $t_{n+1} = 1$ comes from the fact that, if $\mu_{n+1}$ is a meridian for $K_{n+1}$, so that $t_{n+1} = [\mu_{n+1}]$, then
 $\zeta_{\mu_{n+1}}(\rho_L([\mu_{n+1}])) = 1$.
\end{proof}
\begin{Obs} \label{Observation: Condition t_n = 1}
 One can see the condition $t_{n+1} = 1$ also from a purely topological point of view. Image to take the manifold $Y \setminus L'$ and then
 to glue back $K_{n+1}$. The effect on $H_1(Y \setminus L')$ is that the generator $[\mu_{n+1}]$ is killed and now the homology
 class of a loop $\gamma \subset Y \setminus L'$ is determined only by the numbers $\langle \gamma , S_i \rangle$, $S_i \in \{1,\ldots,n\}$
 (i.e. by $\rho_L(\gamma)$).
\end{Obs}

\subsubsection{Proof of the result in the general case}
\begin{proof}[Proof of Theorem \ref{Theorem: Euler characteristics of ECK and ALEX}.]

Let $L = K_1 \sqcup \ldots \sqcup K_n$ be a given link in $Y$. Proposition \ref{Prop: Any L is part o a fibered L'} implies that there exists
an open book decomposition $(L',S,\phi)$ of $Y$ with binding
$$L' = L \sqcup K_{n+1} \sqcup \ldots \sqcup K_m$$
for some $m \geq n$.
Let $\alpha$ be a contact form on $Y$ adapted to $(L',S,\phi)$. Let $R =R_{\alpha}$ be its Reeb vector field.
As remarked in Subsection \ref{Subsection: Fibered links}, and using the same notations, $R$ is circular in
$Y \setminus \mathring{V'}(L')$ where, recall, $V'(L)$ is an union of tubular neighborhoods $V'(K_i) \subsetneq V(K_i)$, $i \in \{1,\ldots,m\}$
of $L$.

Since $\alpha$ is also adapted to $L'$, then each $\mathring{V}(K_i)$ is, by definition, foliated by concentric tori, which in turn are
linearly foliated by Reeb orbits that intersect positively a meridian disk for $K_i$ in $V(K_i)$. Now, we can choose $\alpha$ in a way that
for each $i \in \{n+1,\ldots,m\}$ the tori contained in $V'(K_i)$ are foliated by orbits of $R$ with fixed irrational slope.
This condition can be achieved by applying the Darboux-Weinstein theorem in $V(K_i)$ to make $\alpha|_{V'(K_i)}$ like in Example 6.2.3 of
\cite{CGH2}. 

This implies that, for each $i \in \{n+1,\ldots,m\}$, the only closed orbit of $R$ in $V'(K_i)$ is $K_i$. Define
$U(L') = \bigsqcup_{i=1}^m U(K_i)$, where
$$U(K_i) = \left\{\begin{array}{cll}
                   V(K_i) & \mbox{ if } & i \in \{1,\ldots,n\};\\
                   V'(K_i) & \mbox{ if} & i \in \{n+1,\ldots,m\};
                  \end{array} \right.$$
We have:
\begin{eqnarray*}
 \chi(ECC(L,Y,\alpha)) & = & \zeta_{\rho_L}(\phi_R|_{Y \setminus V(L)})\\
                       & = & \zeta_{\rho_L}(\phi_R|_{Y \setminus U(L')})) \cdot \prod_{i = n+1}^m \prod_{\gamma \in \mathcal{P}(V'(K_i))} \zeta_{\gamma}(\rho_L([\gamma]))\\
                       & = & \zeta_{\rho_L}(\phi_R|_{Y \setminus U(L')})) \cdot \prod_{i = n+1}^m  \zeta_{K_i}(\rho_L([K_i]))\\
                       & = & \zeta_{\rho_{L'}}(\phi_R|_{Y \setminus U(L')}))|_{t_1,\ldots,t_n,1,\ldots,1} \cdot \prod_{i = n+1}^m  \zeta_{K_i}(\rho_L([K_i]))\\
                       & \doteq & \mathrm{ALEX}_{t_1,\ldots,t_n,1,\ldots,1}(Y \setminus L') \cdot \prod_{i = n+1}^m \zeta_{K_i}(\rho_L([K_i]))\\
                       & = & \mathrm{ALEX}_{t_1,\ldots,t_n,1,\ldots,1}(Y \setminus L') \cdot \prod_{i = n+1}^m \left(1 - \prod_{j=1}^n t_j^{\langle K_{i} , S_j \rangle} \right)^{-1}\\
                       & = & \mathrm{ALEX}_{t_1,\ldots,t_n}(Y \setminus L),
\end{eqnarray*}
where:
\begin{itemize}
 \item[line 2] follow reasoning like in the proof of Equation \ref{Equation: Product formula of ECC in terms of zeta functions};
 \item[line 3] hold since $K_i$, for $i \in \{n+1,\ldots,m\}$, is the only Reeb orbit of $\alpha$ in $V'(K_i)$;
 \item[line 4] comes from the idea in Observation \ref{Observation: Condition t_n = 1}: $\rho_L$ and $\rho_{L'}$ coincide
  on the generators $t_i$ of $H_1(Y \setminus L)$ for $i \in \{1,\ldots,n\}$ and $t_i = [\mu_i] = 1 \in H_1(Y \setminus L)$ for
  $i \in \{n+1,\ldots,m\}$;
 \item[line 5] $\rho_{L'}$ and $R|_{Y \setminus U(L')}$ satisfy hypothesis of Theorem
  \ref{Theorem: ALEX equals the twisted zeta function}, up to slightly perturb $R$ near $\partial U(K_i)$, $i \in \{1,\ldots,n\}$, to make
  it non degenerate transverse to the boundary like in the proof in Subsection \ref{Subsection: Fibered links};  
 \item[line 6] is due to the fact that the $K_i$'s are elliptic;
 \item[line 7] is obtained by applying repeatedly the Torres formula on the components $K_{n+1},\ldots,K_{m}$.
\end{itemize}
\end{proof}

\begin{Obs}
 As mentioned at the beginning of this subsection, we could also apply the more general results in \cite{Fri} using the fact that the Reeb vector field
 $R$ used is circular in each $U(K_i)$, since here $R$ is positively transverse to any meridian disk of $K_i$ in $\mathring{V}(K_i)$.
\end{Obs}

The proof of Theorem \ref{Theorem: Euler characteristic of ECK} works exactly like in the fibered case.


\begin{thebibliography}{40}
 \bibliographystyle{annotate}

  \bibitem{Ac}
 A'Campo: \textit{La fonction zeta d'une monodromie}; Comm. Math. Helv. 50:233-248, 1975.
 
 \bibitem{Al}
 Alexander: \textit{Topological invariants of knots and links}; Trans. Amer. Math. Soc. 30(2):275-306, 1928.
 
 \bibitem{Bourg}
 Bourgeois: \textit{A Morse-Bott approach to contact homology}; Sympl. and Cont. Top.: Interactions and Perspectives (AMS, Providence, RI) 35:55-77.
 
 \bibitem{Cim}
 Cimasoni: \textit{Studying the multivariable Alexander polynomial by means of Seifert surfaces}; Bol. Soc. Mat. Mexicana 10(3):107–115, 2004.
 
 \bibitem{Coh}
 Cohen: \textit{A course in simple-homotopy theory}; Berlin-Heidelberg-New York, Springer, 1973.
 
 \bibitem{Co}
 Colin: \textit{Livres ouverts en géométrie de contact}; Séminaire Bourbaki n°969, 2006.
 
 \bibitem{Cott}
 Cotton Clay: \textit{Symplectic Floer homology of area-preserving surface diffeomorphisms}; Geom. Top. 13(5):2619–2674 ,2009.
 
 \bibitem{CGH1}
 Colin, Ghiggini, Honda: \textit{Equivalence of Heegaard Floer homology and embedded contact homology via open book decompositions}; PNAS 108(20):8100-8105, 2011.

 \bibitem{CGH2}
 Colin, Ghiggini, Honda: \textit{Embedded contact homology and open book decompositions}; arXiv:1008.2734, 2010.
 
 \bibitem{CGH3}
 Colin, Ghiggini, Honda: \textit{The equivalence of Heegaard Floer homology and embedded contact homology via open book decompositions I}; arXiv:1208.1074, 2012.

 \bibitem{CGH4}
 Colin, Ghiggini, Honda: \textit{The equivalence of Heegaard Floer homology and embedded contact homology via open book decompositions II}; arXiv:1208.1077, 2012.
  
 \bibitem{CGH5}
 Colin, Ghiggini, Honda: \textit{The equivalence of Heegaard Floer homology and embedded contact homology III: from hat to plus}; arXiv:1208.1526, 2012. 
 
 \bibitem{CGHH}
 Colin, Ghiggini, Honda, Hutchings: \textit{Sutures and contact homology I}; Geom. Topol. 15(3):1749-1842, 2011.
  
 \bibitem{ET}
 Eliashberg, Traynor: \textit{Symplectic geometry and topology}; IAS/Park City Mathematics Series, 1999.
 
 \bibitem{Fra}
 Franks: \textit{Knots, links and symbolic dynamics}; Ann. of Math. 113:529-552, 1981.
 
 \bibitem{Fri}
 Fried: \textit{Homological identities for closed orbits}; Invent. Math. 71:419-442, 1983
 
 \bibitem{Ga}
 Gautschi: \textit{Floer homology of algebraically finite mapping classes}; J. Sympl. Geom. 1(4):715–765, 2003. 
 
 \bibitem{Gi}
 Giroux: \textit{G\'{e}om\'{e}trie de contact: de la dimension trois vers les dimensions sup\'{e}rieures}; Proceedings of the International congress
 of Mathematics, Vol.II:405-414, 2002.   

 \bibitem{Ge}
 Geiges: \textit{An Introduction to Contact Topology}; Cambridge studies in advance mathematics 109, 2008. 
 
 \bibitem{Gromov}
 Gromov: \textit{Pseudo-holomorphic curves in symplectic manifolds}; Invent. Math. 82:307-347, 1985.

 \bibitem{GP}
 Guillemin, Pollack: \textit{Differential topology}; ISBN-13: 978-0821851937, AMS, 2010 (reprint). 
 
 \bibitem{Gutt}
 Gutt: \textit{The Conley-Zehnder index for a path of symplectic matrices}; arxiv.org/abs/1201.3728, 2012.
 
 \bibitem{Gu}
 Guyard: Ph.D. thesis (Université de Nantes); in preparation.
  
 \bibitem{Hu1}
 Hutchings: \textit{An index inequality for embedded pseudoholomorphic curves in symplectizations}; J. Eur. Math. Soc. (JEMS) 4:313-361, 2002; arXiv:math/0112165.
 
 \bibitem{Hu2}
 Hutchings: \textit{The embedded contact homology index revisited}; New perspectives and challenges in symplectic field theory,
 CRM Proc. Lecture Notes, 49:263-297, AMS, 2009; arXiv:0805.1240.

 \bibitem{Hu3}
 Hutchings: \textit{Lecture notes on embedded contact homology}; arXiv:1303.5789v1, 2013.
 
 \bibitem{HL1}
 Hutchings, Lee: \textit{Circle-valued Morse theory, Reidemeister torsion, and Seiberg-Witten invariants of 3-manifolds}; Topology 38(4):861-888, 1999.
 
 \bibitem{HL2}
 Hutchings, Lee: \textit{Circle-valued Morse theory and Reidemeister torsion}; Geom. Topol. 3:369–396, 1999. 
 
 \bibitem{HS}
 Hutchings, Sullivan: \textit{The periodic Floer homology of a Dehn twist}; Algebr. Geom. Topology 5:301-354, 2005; arXiv:math/0410059.
 
 \bibitem{Ju}
 Juh\'{a}sz: \textit{Holomorphic discs and sutured manifolds}; Algebr. Geom. Topology 6:1429-1457, 2006.
 
 \bibitem{KM}
 Kronheimer, Mrowka: \textit{Monopoles and Three-Manifolds}; New Mathematical Monographs, 10; Cambridge University Press, Cambridge, 2007.
 
 \bibitem{KM1}
 Kronheimer, Mrowka: \textit{Knots, sutures and excision}; J. Diff. Geom. 84(2):301–364, 2010.
 
 \bibitem{KM2}
 Kronheimer, Mrowka: \textit{Instanton Floer homology and the Alexander polynomial}; Algebr. Geom. Topology 10(3):1715–1738, 2010. 
 
 \bibitem{Li}
 Lipshitz: \textit{A cylindrical reformulation of Heegaard Floer homology}; Geom. Topol. 10:955-109, 2006; arXiv:math/0502404v2. 

 \bibitem{McCleary}
 McCleary: \textit{A User's Guide to Spectral Sequences}; Cambridge Studies in Advanced Mathematics 58, 2001.
 
 \bibitem{McDuff}
 McDuff: \textit{Singularities and positivity of intersections of J-holomorphic curves}; Progress in Mathematics 117:191-215, 1994.
 
 \bibitem{MS}
 McDuff, Salamon: \textit{J-Holomorphic Curves and Symplectic Topology}; American Mathematical Society colloquium publications, 2004.
 
 \bibitem{ML}
 Mac Lane: \textit{Homology} (3d Corrected Printing); Springer-Verlag, 1975.
 
 \bibitem{OS1}
 Ozsv\'{a}th, Szab\'{o}: \textit{Holomorphic disks and topological invariants for closed three-manifolds}; Ann. of Math. (2) 159:1027-1158, 2004; arXiv:math/0101206v4.
 
 \bibitem{OS2}
 Ozsv\'{a}th, Szab\'{o}: \textit{Holomorphic disks and three-manifolds invariants: Properties and applications}; Ann. of Math. (2) 159:1159-1245, 2004; arXiv:math/0105202v4.
 
 \bibitem{OS3}
 Ozsv\'{a}th, Szab\'{o}: \textit{Holomorphic disks and knot invariants}; Adv. Math. 186:58-116, 2004; arXiv:math/0609734v2.
 
 \bibitem{OS4}
 Ozsv\'{a}th, Szab\'{o}: \textit{Heegaard Floer homologies and contact structures}; Duke Math. J. 129:39-61, 2005; arXiv:math/0210127. 

 \bibitem{OS5}
 Ozsv\'{a}th, Szab\'{o}: \textit{Holomorphic disks, link invariants and the multi-variable Alexander polynomial}; Algebr. Geom. Topology 8:615–692, 2008;  arXiv:math/0512286.
 
 \bibitem{OS7}
 Ozsv\'{a}th, Szab\'{o}: \textit{An Introduction to Heegaard Floer Homology}; Floer homology, gauge theory and low dimensional topology;
 Clay Mathematics Proceedings 5:3-27, 2004.
 
 \bibitem{Ra}
 Rasmussen: \textit{Floer homology and knot complements}; Ph.D. Thesis; 2003; ArXiv:math.GT/0306378. 
 
 \bibitem{Ro}
 Rolfsen: \textit{Knots and Links}; Publish or Perish, Huston, 1976.

 \bibitem{Sm}
 Smale: \textit{Differentiable dynamical systems}; Bull. AMS 73:747-817, 1967.

 \bibitem{Tor}
 Torres: \textit{On the Alexander polynomial}; Ann. of Math. 57(2):57-89, 1953.
 
 \bibitem{Tu}
 Turaev: \textit{Reidemeister torsion in knot theory}; Russian Math. Surveys 41:119–182, 1986. 
 
 \bibitem{TW}
 Thurston, Winkelnkemper: \textit{On the existence of contact forms}; Proc. Amer. Math. Soc 52:345-347; 1975.
 
 \bibitem{V}
 Vaugon; \textit{Étude dynamique des champs de Reeb et propriétés de croissance de l’homologie de contact}; Ph.D. thesis, 2011.
 
 \bibitem{We1}
 Wendl: \textit{Finite energy foliations and surgery on transverse links}; Ph.D. thesis, 2005.
 
 \bibitem{We2}
 Wendl: \textit{Finite energy foliations on overtwisted contact manifolds}; Geom. Top. 12:531-616, 2008.
 
\end{thebibliography}
\end{document}